\documentclass[11pt]{amsart}
\usepackage{amsfonts}
\usepackage{color}
\usepackage{amssymb,amsthm,amsmath,hyperref}
\usepackage{epsfig}
\usepackage{graphicx,float}
\usepackage{CJK}
\usepackage{setspace}
\usepackage{cancel}
\usepackage{stmaryrd}
\usepackage{verbatim}
\usepackage{ulem}
\usepackage{makecell}

 \newtheorem{thm}{Theorem}[section]
 \newtheorem{coro}[thm]{Corollary}
 \newtheorem{lem}[thm]{Lemma}
 \newtheorem{prop}[thm]{Proposition}
 \theoremstyle{definition}
 
 \newtheorem{rem}[thm]{Remark}
 \numberwithin{equation}{section}

\def\dl{\delta}
\def\tl{\tilde}

\def\D{\displaystyle}
\def\Dl{\Delta}

\def\sig{\sigma}

\def\N{\mathbb{N}}

\def\nn{\nonumber}

\def\eps{\epsilon}

\def\fr{\frac}
\def\al{\alpha}

\def\la{\langle}
\def\ra{\rangle}
\def\R{\mathbb{R}}
\def\Z{\mathbb{Z}}
\def\T{\mathbb{T}}

\def\pr{\partial}
\def\nb{\nabla}

\def\les{\lesssim}
\def\lm{\lambda}
\def\Lm{\Lambda}

\def\om{\omega}

\def\l|{\left\|}
\def\r|{\right\|}

\newcommand{\beq}{\begin{eqnarray}}
\newcommand{\eeq}{\end{eqnarray}}
\newcommand{\beqno}{\begin{eqnarray*}}
\newcommand{\eeqno}{\end{eqnarray*}}
\newcommand{\be}{\begin{equation}}
\newcommand{\ee}{\end{equation}}
\newcommand{\beno}{\begin{equation*}}
\newcommand{\eeno}{\end{equation*}}
\newtheorem{theorem}{Theorem}[section]

\newtheorem{Lemma A.1}{Lemma A.1}
\theoremstyle{definition}

\theoremstyle{remark}
\newtheorem{remark}[theorem]{Remark}
\allowdisplaybreaks[4]

\begin{document}
\title[Landau damping for VPFP]{Optimal stability threshold in lower regularity spaces for the Vlasov-Poisson-Fokker-Planck equations}

\author{ Weiren Zhao}
\thanks{Corresponding author: Weiren Zhao, email: zjzjzwr@126.com}
\address[W. Zhao]{Department of Mathematics, New York University Abu Dhabi, Saadiyat Island, P.O. Box 129188, Abu Dhabi, United Arab Emirates.}
\email{zjzjzwr@126.com, wz19@nyu.edu}

\author{Ruizhao Zi}
\address[R. Zi]{School of Mathematics and Statistics, and Key Laboratory of Nonlinear Analysis \& Applications (Ministry of Education), Central China Normal University, Wuhan,  430079,  P. R. China.}
\email{rzz@ccnu.edu.cn}

\date{\today}

\begin{abstract}
    In this paper, we study the optimal stability threshold for the Vlasov-Poisson equation with weak Fokker-Planck collision. We prove that if the initial perturbation is of size $\nu^{\frac{1}{2}}$ in the critical weighted space $H_x^{\log}L^2_{v}(\langle v\rangle^m)$, then the solution remains the same size in the same space. Moreover, a space-time type Landau damping holds, namely, $\|E\|_{L^2_tL^2_x}\lesssim \nu^{\fr{1}{2}}$; and a point-wise type Landau damping holds, namely, $\|E(t)\|_{L^2}\lesssim \nu^{1/2}\langle t\rangle^{-N}$ for any $N>0$ for $t\geq \nu^{-1}$. We also prove that there exists initial perturbation in $H^{1}_xL^2_v(\langle v\rangle^m)$ with size $\nu^{\fr12-\fr32\epsilon_0}$ with any ${\epsilon_0>0}$, such that the enhanced dissipation fails to hold in the following sense: there is $0<T\ll \nu^{-\fr13}$ such that
    \begin{align*}
        \|\la v\ra^m f_{\neq}(T)\|_{L^2_xL^2_v}\gtrsim \frac{1}{\nu^{\delta_1}}\|\la v\ra^m f_{\neq}(0)\|_{ H^1_xL^2_v}
    \end{align*}
    with some $\delta_1>0$. 
The paper solves the open problem raised in [Bedrossian; arXiv: 2211.13707] about the sharp stability threshold in lower regularity spaces. 
    %The main idea is to construct a wave operator $\mathbf{D}$ with a very precise expression to absorb the nonlocal term, namely, 
%\begin{align*}
%    \mathbf{D}[\partial_tg+v\cdot \nabla_x g+E\cdot\nb_v \mu]=(\partial_t +v\cdot \nabla_x)\mathbf{D}[g]. 
%\end{align*}

\end{abstract}

\maketitle
\section{Introduction}
In this paper, we consider the following single-species Vlasov-Poisson- Fokker-Planck (VPFP) equations with a neutralizing background on $\mathbb{T}^n_x\times\mathbb{R}^n_v$ (with $\mathbb{T}_x$ normalized to length $2\pi$ and $n=1,2,3$), which models the evolution of the distribution function of electrons which are subject to the electrostatic force coming from their Coulomb interaction and to a Brownian force which models their collisions:
\be\label{VFP}
\begin{cases}
\partial_tF+v\cdot\nabla_xF+E(t,x)\cdot\nabla_vF=\nu \big(\Delta_vF+\nabla_v\cdot(vF)\big),\\[1mm]
E(t,x)=-\nabla_x(-\Delta_x)^{-1}\Big(\varrho-\frac{1}{|\T^n|}\displaystyle\int_{\T^n}\varrho(t,x)dx\Big),\\[1mm]
\displaystyle\varrho(t,x)=\int_{\mathbb{R}^n}F(t,x,v)dv,\\[1mm]
F(t=0, x, v)=F_{\mathrm{in}}(x,v).
\end{cases}
\ee
The unknown is the distribution function $F (t, x, v)$, which gives the number density of (say) electrons at location $x$ moving with velocity $v$. The electrostatic force is responsible for the self-consistent force term $E(t,x)\cdot\nb_vF$. The non-local interaction between $E(t,x)$ and the density function $\varrho(t, x)$ of electrons would be through Coulomb electrostatic interactions. The Brownian force is modeled by the Fokker-Planck term $\nu \big(\Dl_vF+\nb_v\cdot(vF)\big)$ with the small coefficient $0< \nu\ll 1$ which represents the collision frequency. 

A direct calculation gives the  conservation of mass
\begin{align*}
    \frac{d}{dt}\int_{\mathbb{T}^n}\varrho(t,x)dx=\frac{d}{dt}\int_{\mathbb{T}^n\times\mathbb{R}^n}F(t, x,v)dvdx=0.
\end{align*}
Therefore, the average density can be normalized to 1 by assuming that the initial distribution $F_{\mathrm{in}}$ satisfies
\be\label{initial-F}
\displaystyle
\frac{1}{(2\pi)^n}\int_{\mathbb{T}^n\times\mathbb{R}^n}F_{\mathrm{in}}(x,v)dxdv=1.
\ee
Clearly, the VPFP equation \eqref{VFP} admits a steady solution, namely the global Maxwellian, 
\be\label{gM}
\mu(v)=\frac{1}{(2\pi )^{\frac{n}{2}}}e^{-\frac{|v|^2}{2}}.
\ee

In this paper, we study the asymptotic stability of the global Maxwellian $\mu(v)$. It is natural to introduce the perturbations
\begin{gather}
\label{pertur-1}F(t,x,v)=\mu(v)+g(t,x,v), \quad \varrho(t,x)=1+\rho(t,x).
\end{gather}
Then $\rho(t,x)=\D\int_{\R^3}g(t,x,v)dv$, and by using the conservation of mass and the initial assumption \eqref{initial-F}, one deduces that
\begin{gather}\label{0M}
\int_{\mathbb{T}^n}\rho(t,x)dx=0.
\end{gather}

Now we rewrite the system \eqref{VFP} in terms of  $g$ and $\rho$:
\begin{subequations}\label{pVPFP}
\be\label{VPFP-perturbation}
\begin{cases}
\pr_tg+v\cdot\nb_xg-\nu { L}[g]+E\cdot\nb_v\mu=-E\cdot\nb_vg,\\
E(t,x)=-\nb_x(-\Dl_x)^{-1}\rho,\quad \rho(t,x)=\D\int_{\R^3}g(t,x,v)dv,\\
g(t=0, x, v)=g_{\mathrm{in}}(x,v),
\end{cases}
\ee
with $\rho$ satisfying \eqref{0M}, and  
\begin{align}
    &{ L}[g]:=\Dl_v g+\nb_v\cdot(vg).
\end{align}
\end{subequations}

In many plasma physics settings, the collision frequency is very small ($\nu\ll1$), and it is common to neglect collisions and consider the Vlasov-Poisson equation ($\nu=0$). However, the interaction between weak collisions, collisionless effects, and nonlinear dynamics has been a central topic in plasma physics for decades \cite{malmberg1964collisionless,o1968effect}. A fundamental collisionless phenomenon is {\bf Landau damping (L.D.)}, discovered by Landau \cite{landau196561}, which describes the rapid decay of the electric field in the absence of dissipation. This effect is closely related to phase-space mixing generated by the transport operator. In the nonlinear regime, the dynamics are strongly influenced by the plasma echo mechanism \cite{malmberg1968plasma}, where nonlinear interactions excite modes that unmix in phase space and generate transient growth, leading to a cascade and loss of regularity. The first mathematical proof of nonlinear Landau damping for the Vlasov-Poisson equation was given by Mouhot and Villani \cite{MouhotVillani2011}, requiring analytic regularity. This was later simplified by Bedrossian, Masmoudi, and Mouhot \cite{BMM2016}, where the threshold was lowered to Gevrey regularity $s>\frac13$; see also \cite{grenier2021landau}. The sharp Gevrey result was obtained in \cite{ionescu2024nonlinearsharp}. In lower regularity settings, BGK waves and plasma echoes may appear \cite{bedrossian2020nonlinear, lin2011small}. A related phenomenon in fluid mechanics is {\bf inviscid damping (I.D.)} for the 2D Euler equation near Couette flow, first observed by Orr \cite{Orr1907} and proved nonlinearly in \cite{BM2015, IonescuJia2020cmp, ChenWeiZhangZhang2023}, with further developments for monotone shear flows in \cite{IJ2020, MasmoudiZhao2020, zhao2023inviscid}. Both Landau damping and inviscid damping originate from transport-induced mixing mechanisms.
When collisions are weak but nonzero, the interaction between transport and dissipation leads to the {\bf enhanced dissipation (E.D.)} effect: mixing transfers energy to high frequencies where dissipation becomes stronger, yielding a dissipation time scale $\nu^{-1/3}$, much shorter than the heat time scale $\nu^{-1}$. This mechanism can suppress plasma echoes generated by nonlinear interactions. Bedrossian \cite{bedrossian2017suppression} proved the asymptotic stability of the global Maxwellian for the Vlasov-Fokker-Planck equation when the perturbation is of size $\nu^{1/3}$ in Sobolev spaces. 

In this paper, we study the asymptotic stability in even lower Sobolev spaces, and solve the open problem raised by Bedrossian in  \cite{bedrossian2022brief}:\par

{\it Can one obtain stability thresholds in much lower regularity, for example $L^1\cap L^{\infty}$ or $L^1\cap L^p$ initial data (along with velocity moments), at least for linear Fokker-Planck collisions? Could one prove these were sharp?}

Our first main result states as follows:
\begin{theorem}\label{thm:main1}
    Let $n=1,2,3$. For any $N>0$, there exist $\epsilon_0,\nu_0>0, m_0>3$ such that if the initial data $g_{\mathrm{in}}(x,v)$ satisfies 
    \begin{align}\label{eq: normalize}
        \int_{\mathbb{T}^n\times \mathbb{R}^n}g_{\mathrm{in}}(x,v)dxdv=0,
    \end{align}
    and for any $0< \epsilon\leq \epsilon_0$ and $m\geq m_0$, it holds that
\begin{subequations}\label{initial-sta}
    \begin{align}
        &\|\ln(e+|D_x|)\langle v\rangle^mg_{\mathrm{in}}\|_{L^2_{x,v}(\mathbb{T}^3\times \mathbb{R}^3)}\leq \epsilon \nu^{\fr{1}{2}},\quad \text{for}\quad n=3,\\
        &\|\langle v\rangle^mg_{\mathrm{in}}\|_{L^2_{x,v}(\mathbb{T}^n\times \mathbb{R}^n)}\leq \epsilon \nu^{\fr{1}{2}},\quad \text{for}\quad n=1,2,
    \end{align}
\end{subequations}
    then the equation \eqref{pVPFP} with $0< \nu\leq \nu_0$ has a unique global solution $(g, E)$. Moreover, the following estimates hold for some $C>1$ independent of $\nu, t$: 
    \begin{itemize}
        \item Long-time estimates: For $0\leq t\leq \nu^{-1}$, we have
        \begin{itemize}
            \item a bounded estimate of the distribution
            \begin{subequations}\label{bd:distribution}
                \begin{align}
                &\sup_{t\in[0,\nu^{-1}]}\|\ln(e+|\nb_x|)\langle v\rangle^mg(t)\|_{L^2_{x,v}(\mathbb{T}^3\times \mathbb{R}^3)}\leq C\epsilon \nu^{\fr{1}{2}},\quad \text{for}\quad n=3,\\
        &\sup_{t\in[0,\nu^{-1}]}\|\langle v\rangle^mg(t)\|_{L^2_{x,v}(\mathbb{T}^n\times \mathbb{R}^n)}\leq C\epsilon \nu^{\fr{1}{2}},\quad \text{for}\quad n=1,2.
            \end{align}
            \end{subequations}
            \item a space-time type of Landau damping 
            \begin{align}\label{bd:E}
        \||\nb_x|^{\fr32}E(t)\|_{L^2_{t,x}}\leq C\epsilon\nu^{\fr12}.
            \end{align}
            \item an improved regularity estimate for any $s>0$
            \begin{align}\label{bd:improve}
                \sup_{t\in [\nu^{-1}/2,\nu^{-1}]}\|\langle v\rangle^mg(t)\|_{H^{s}_{x,v}}\leq C\epsilon\nu^{\fr12}. 
            \end{align}
        \end{itemize}
        \item Global-in-time estimates: For $t\geq \nu^{-1}$, a point-wise Landau damping estimate holds, namely, 
        \begin{align}\label{Lan-dam}
            \|E(t)\|_{L^2}\leq \frac{C\epsilon\nu^{\fr12}}{\langle t\rangle^N}. 
        \end{align}
        \item Enhanced dissipation estimate: for all $t\geq 0$, it holds that
        \begin{subequations}\label{eq: enhanced-dissipation}
            \begin{align}
            \|\la v\ra^m\mathbb{P}_{\neq}g(t)\|_{L^2}\leq& C \|\la v\ra^m\mathbb{P}_{\neq}g_{\rm in}\|_{L^2}\langle \nu^{1/3}t\rangle^{-s},\quad{\rm for}\quad n=1,2,\\
            \|\la v\ra^m\ln(e+|\nb_x|)\mathbb{P}_{\neq}g(t)\|_{L^2}\leq& C \|\la v\ra^m\ln(e+|\nb_x|)\mathbb{P}_{\neq}g_{\rm in}\|_{L^2}\langle \nu^{1/3}t\rangle^{-s},\quad{\rm for}\quad n=3,
        \end{align}
        \end{subequations}
        where $\mathbb{P}_{\neq}$ is the projection to the non-zero modes. 
    \end{itemize}
\end{theorem}

The second result is to show that the threshold $\nu^{1/2}$ for the perturbations in the critical spaces is sharp. 

\begin{theorem}\label{Thm: main2}
    Let $n=1$. For any $\epsilon_0\in (0,\frac{1}{100})$, there is $g_{\rm in}(x,v)$ satisfying \eqref{eq: normalize} and 
    \begin{align*}
        \|\langle v\rangle^m g_{\rm in}\|_{L^2}\ge C\nu^{\frac12-\frac32\epsilon_0}
    \end{align*}
    for some $C>1$ independent of $\nu$, such that the enhanced dissipation estimate \eqref{eq: enhanced-dissipation} does not hold. More precisely, there are $C_0>c_0>0$ and $\delta_0>0$ independent of $\nu$, such that for all $t\in [0,T_0]$ with $T_0=\delta_0\nu^{-\fr13+\epsilon_0}\ln \nu^{-1}\ll \nu^{-\fr13}$, it holds that
    \begin{align*}
        c_0e^{c_0\nu^{\fr13-\epsilon_0}t} \|\la v\ra^m\mathbb{P}_{\neq}g_{\rm in}\|_{L^2}\leq \|\la v\ra^m\mathbb{P}_{\neq}g(t)\|_{L^2}\leq C_0e^{C_0\nu^{\fr13-\epsilon_0}t} \|\la v\ra^m\mathbb{P}_{\neq}g_{\rm in}\|_{L^2}.
    \end{align*}
    In particular, a $\nu$-dependent amplification estimate holds at $T_0$, namely, 
    \begin{align*}
        \|\la v\ra^m\mathbb{P}_{\neq}g(T_0)\|_{L^2}\geq c_0\|\la v\ra^m\mathbb{P}_{\neq}g_{\rm in}\|_{L^2}\nu^{-c_0\delta_0}. 
    \end{align*}
\end{theorem}
\begin{rem}
    To simplify the proof, we state and prove Theorem \ref{Thm: main2} for $n=1$. For $n=2,3$, by modifying our proof (keeping homogeneous in the other dimensions in choosing the zero-mode perturbation, see \eqref{5.1}), the instability result holds. 
\end{rem}

\begin{rem}
    Both results in this paper are related to the stability threshold problem, namely, 

{\it 
 Given a norm $\|\cdot\|_X$, find a $\gamma=\gamma(X)$ so that
\begin{align}\label{eq:STP}
\begin{aligned}
  &\|g_{\mathrm{in}}\|_{X}\leq \nu^{\gamma} \Rightarrow \text{stability, L.D. (or I.D.), and E.D.},\\
&\|g_{\mathrm{in}}\|_{X}\gg \nu^{\gamma} \Rightarrow \text{instability}.
\end{aligned}
\end{align}
}

    This kind of `stability threshold' result is entirely analogous to the line of research on the two-dimensional Navier-Stokes equations near shear flows. Here we briefly compare our results with the Navier-Stokes case in the following: 
    \begin{itemize}
        \item Stability in the Sobolev 
        \begin{itemize}
            \item 2D NS near Couette flow \cite{MasmoudiZhao2019,wei2023nonlinear}: 
            $\|u_{\mathrm{in}}\|_{H^3}\leq \epsilon\nu^{\frac{1}{3}}\ \Rightarrow$ I.D. and E.D.
            \item VPFP near global Maxwellian \cite{bedrossian2017suppression}: 
            $\|\langle v\rangle^m g_{\mathrm{in}}\|_{H^{N}}\leq \epsilon\nu^{\frac{1}{3}}\ \Rightarrow$ L.D. and E.D.
            \item VPL near global Maxwellian \cite{chaturvedi2023vlasov}:
            $\|\langle v\rangle^m g_{\mathrm{in}}\|_{H^{N}}\leq \epsilon\nu^{\frac{1}{3}}\ \Rightarrow$ L.D. and E.D.
        \end{itemize}
        \item Stability in Gevrey class
        \begin{itemize}
            \item 2D NS near Couette flow \cite{liMasmoudiZhao2022asymptotic}: 
            $\|u_{\mathrm{in}}\|_{\mathcal{G}^{s_{\mathrm{f}}}}\leq \epsilon\nu^{\gamma}\  \Rightarrow$ I.D. and E.D.
            \item VPFP near global Maxwellian \cite{bedrossian2025landau, luo2024weak}:
            $\|\langle v\rangle^me^{2\nu|v|^2}g_{\mathrm{in}}\|_{\mathcal{G}^{s}}\leq \epsilon\nu^{\gamma} \ \Rightarrow$ L.D. and E.D.
        \end{itemize}
        \item Justification of the inviscid limit or collisionless limit 
        \begin{itemize}
            \item 2D NS near Couette flow \cite{BMV2016, liMasmoudiZhao2022asymptotic}: $u_{\mathrm{in}}\in \mathcal{G}^2$.
            \item VPFP near global Maxwellian \cite{bedrossian2025landau, luo2024weak}:
            $g_{\mathrm{in}}\in{\mathcal{G}^{s}_{w}}$ with $s<3$.
            \item VPL near global Maxwellian
            \cite{bzz-VPLandau}: $g_{\mathrm{in}}\in{\mathcal{G}^{s}_{w}}$ with $s<2$.
        \end{itemize}
        \item Sharp stability threshold in lower regularity spaces
        \begin{itemize}
            \item 2D NS near Couette flow \cite{MasmoudiZhao2020cpde, liMasmoudiZhao2023dynamical}: $\nu^{1/2}$ is the sharp stability threshold for 
            $\nabla u_{in}\in {H^{\log}_xL^2_y}$ 
            \item VPFP near global Maxwellian [this paper]: 
            $\nu^{1/2}$ is the sharp stability threshold for 
            $g_{\rm in}\in {H^{\log}_xL^2_v(\langle v\rangle^m)}$. 
        \end{itemize}
    \end{itemize}
\end{rem}

\subsection{Notation}\label{sec: notation}
For $f(x,v)\in L^2(\mathbb{T}^n\times\mathbb{R}^n)$, the Fourier transform of $f$ in $x$ variable is defined by
\begin{align*}
    f_k(v)=\fr{1}{(2\pi)^n}\int_{\T^n}e^{-ik\cdot x}f(x,v)dx,
\end{align*}
and the Fourier transform of $f$ in $(x,v)$ variables is defined by
\begin{align*}
    \hat{f}_k(\xi)=\mathcal{F}_v[f_k](\xi)=\fr{1}{(2\pi)^n}\int_{\T^n\times\R^n}e^{-ik\cdot x}e^{-i\xi\cdot v}f(x,v)dxdv.
\end{align*}

For $f: [0,\infty)\rightarrow \mathbb{C}$ satisfying $e^{-\underline{\lm} t}f(t)\in L^1[0,\infty)$ for some $\underline{\lm}\in\R$, let us define the Fourier-Laplace transform
\[
\mathcal{L}[f](\lm)=\int_0^\infty e^{-\lm t}f(t)dt,\quad {\rm for}\quad \mathfrak{Re}\lm\ge \underline{\lm}.
\]
In particular, if $\underline{\lm}\le0$, $\mathcal{L}[f](iz)=\mathcal{F}[{\bf1}_{t\ge0}f](z)$ for all $z\in\R$.

For $k\in\Z^3,s\ge0$, let us denote 
\[
{\rm A}_{k,s}(t):=\ln(e+|k|)\la C_s\nu^{\fr13}|k|^{\fr23}t\ra^{\fr{s}{2}},\quad
{\rm B}_{k,s}(t):=|k|^{\fr12}{\rm A}_{k,s}(t),
\]
where $C_s$ is the same as that appearing in \eqref{def-En-al}.
We also use the notation
\[
{\bf A}_s(t):=\ln(e+|\pr_x|)\la C_s\nu^{\fr13}|\pr_x|^{\fr23}t\ra^{\fr{s}{2}},\quad
{\bf B}_{s}(t):=|\pr_x|^{\fr12}{\bf A}_{s}(t).
\]

\section{Main idea}
In this section, we highlight the main ideas in this paper. 
\subsection{Treatment of nonlocal term and the wave operator method}
We first consider the linearized Vlasov equation, 
\begin{align*}
    \left\{
    \begin{aligned}
        &\partial_tg+v\cdot \nabla g=0,\quad
        \rho=\int_{\mathbb{R}^3}g(t, x, v)dv,\\
        &g(0,x,v)=g_{\mathrm{in}}(x,v).
    \end{aligned}\right.
\end{align*}
It is easy to obtain that the Fourier transform of the density $\rho$ satisfies
\begin{align*}
    \hat{\rho}_k(t)=(\widehat{g_{\mathrm{in}}})_k(kt). 
\end{align*}
Due to the very low regularity assumption of $g_{\rm{in}}$, we can only obtain a space-time estimate before the enhanced dissipation happens, namely, for $t\leq \nu^{-\fr13}$
\begin{align*}
    \big\||k|^{1/2}\rho_k(t)\big\|_{L_t^2l^2_k}\lesssim \|\langle v\rangle^{3}g_{\rm in}\|_{L^2_xL^2_v}.
\end{align*}
Thus, the classical treatment of the linear nonlocal term $E\cdot \nb_v \mu$ will cause a $\min\{\sqrt{t},\nu^{-\fr16}\}$ size loss even for the linear estimate. To obtain a uniform in time control of the distribution, even at the linear level for very low regularity initial data, the wave operator seems to be necessary. We note that if the initial data has a better Sobolev regularity, one may expect $\|E\|_{L^2_k}\in L^1_t$, then there is no such loss of size. 

To treat this term, we apply the wave operator method to absorb this term, namely, we construct a linear operator $\mathbf{D}$, such that 
\begin{align*}
    \mathbf{D}(\partial_tg+v\cdot \nabla_x g+E\cdot\nb_v \mu)=(\partial_t +v\cdot \nabla_x)\mathbf{D}[g]. 
\end{align*}

\begin{rem}
After we completed this paper, it was pointed out to us that the wave operator for the Vlasov-Poisson equation had first been obtained by Morrison and Pfirsch in 1992 \cite{MorrisonPfirsch1992}, who used it to derive formulas for linear Landau damping. This operator, called the $G$-transform there, was later studied in detail in \cite{Morrison}. A similar idea has also been used in more recent kinetic theory results \cite{despres2014symmetrization, despres2019scattering, JeffreyMorrison2018}. We are grateful for the opportunity to include these references and for the constructive feedback that helped improve this work.
\end{rem}

Wave operators (also called distorted Fourier transform) appear in the study of dispersive equations with potential,  see \cite{DelortMasmoudi, GP20,  LLS21, GPR18, Schlag, Yajima} for more details. Let us also mention that recently, the wave operator was successfully used to solve important problems in fluid mechanics.
In \cite{LWZ}, the authors use the wave operator method to solve Gallay's conjecture on pseudospectral and spectral bounds of the Oseen vortices operator. In \cite{WeiZhangZhao2020}, the wave operator method was used to solve Beck and Wayne's conjecture. In \cite{MasmoudiZhao2020, zhao2023inviscid}, the authors use the wave operator method to prove the nonlinear inviscid damping for stable monotone shear flows for both homogeneous and inhomogeneous Euler equations. In \cite{li2024asymptotic}, the time-dependent wave operator is constructed to solve the stability threshold problem of the stable monotonic shear flows. 

%As far as we know, this is the first time the wave operator is applied to the kinetic equation. 
Now, let us highlight the main idea in the construction for $n=3$, which is more direct than that in the fluid equations. Let us focus on a special case, namely, $ k=|k|(\frac{k}{|k|},0,0)$, which can be obtained by an easy rotation, see \eqref{Orthogonal}. The aim is to construct a time-independent operator $\mathbf{D}_k$ such that 
\begin{align}\label{eq:waveop1}
    \mathbf{D}_k\bigg[|k|v_1 g+\frac{|k|v_1 e^{-\fr{|v|^2}{2}}}{(2\pi)^{\fr32}|k|^2}\int_{\mathbb{R}^3}g(t,v)dv\bigg]=|k|v_1  \mathbf{D}_k[g]. 
\end{align}
We assume that the wave operator has a local part and a nonlocal part, namely, 
\begin{align}\label{eq: waveop2}
    \mathbf{D}_k[g]=g+\int_{\mathbb{R}^3}K(v,v')g(v')dv'.
\end{align}
Then the kernel $K(v,v')$ can be obtained directly by substituting \eqref{eq: waveop2} into \eqref{eq:waveop1}. We refer to \eqref{WO} and \eqref{Def-D} for the expression of the wave operator. The main difficulty is to prove the invertibility of the wave operator and the estimates of the commutator with derivatives. We refer to sections \ref{sec:inverseWO} and \ref{sec:prop-WO}, respectively. It worth to pointing out that \underline{the invertibility of the wave operator} is equivalent to \underline{the Penrose stability condition} \cite{Morrison}. 

\begin{rem}
    In the recent paper \cite{ionescu2024nonlinearsharp}, the existence of the wave operator associated with the nonlinear Vlasov-Poisson equation was obtained; see also \cite{flynn2023scattering,huang2026scattering} for the studies of scattering operators for the Vlasov-Poisson equation. 
\end{rem}

\subsection{Improved regularity and time-weighted norm}
To obtain a point-wise Landau damping estimate, we need to show that due to the Fokker-Planck collision, the distribution $g$ gains regularity and has a uniform-in-$\nu$ Sobolev estimate for ${t\geq \nu^{-1}}$. We thus introduce the following energy functional for $s,m\in\N$:
\begin{align}\label{def-En}
\mathcal{E}^{s}_{m}(h(t))
:=\sum_{\al\in \N^n:|\al|\le m}\mathfrak{E}_{\al}^{s}(h(t)),
\end{align}
with the $(s,\al)$ level energy functional
\begin{align}\label{def-En-al}
\mathfrak{E}^{s}_{\al}(h(t))
:=\sum_{0\le\ell\le s}\Big[(\kappa\nu t)^\ell\big\||\nb_v|^\ell\mathcal{M}(v^\al h(t))\big\|_{L^2_{x,v}}^2+(C_{s}\kappa\nu^{\fr13} t)^\ell\big\||\nb_x|^{\fr{\ell}{3}}\mathcal{M}(v^\al h(t))\big\|_{L^2_{x,v}}^2\Big],
\end{align}
where $0<\kappa\le1\le C_s$ are two constants independent of $\nu$, determined in \eqref{Csig} and \eqref{kappa}, respectively.
Moreover,  $\mathcal{M}$ is a  Fourier multiplier defined in  \ref{def-M} which enables us to obtain the enhanced dissipation estimate. 
We also define the corresponding $(s,\al)$ level dissipation 
\begin{align}\label{def-Dis-al}
\frak{D}^{s}_{\al}(h(t))
:=\nn&\sum_{0\le\ell\le s}(\kappa\nu t)^{\ell}\Big(\nu\big\||\nb_v|^{\ell+1}\mathcal{M}(v^\al {h})\big\|_{L^2_{x,v}}^2+\fr{1}{4}\nu^{\fr13}\big\||\nb_x|^{\fr13}|\nb_v|^\ell \mathcal{M}(v^\al {h}_{\ne})\big\|_{L^2_{x,v}}^2\Big)\\
    &+\sum_{0\le\ell\le s}(C_s\kappa\nu^{\fr13} t)^{\ell}\Big(\nu\big\||\nb_v||\nb_x|^{\fr\ell3}\mathcal{M}(v^\al {h})\big\|_{L^2_{x,v}}^2+\fr{1}{4}\nu^{\fr13}\big\||\nb_x|^{\fr{\ell+1}{3}} \mathcal{M}(v^\al {h}_{\ne})\big\|_{L^2_{x,v}}^2\Big),
\end{align}
and the total dissipation is defined by
\begin{align}\label{def-Dis}
    \mathcal{D}^s_m(h(t)):=\sum_{\al\in \N^n:|\al|\le m}\frak{D}^{s}_{\al}(h(t)).
\end{align}

Notice that at $t=0$, it holds that $\mathcal{E}^{s}_{m}(h_{\ne}(0))=\|\langle v\rangle^{m}(h_{\rm in})_{\ne}\|_{L^2_{x,v}}^2$ and at $t=\nu^{-1}$, it holds that 
\begin{align*}
    \mathcal{E}^{s}_{m}(h_{\ne}(\nu^{-1}))\approx \|\langle\nabla_v\rangle^{s}(\langle v\rangle^mh_{\neq}(\nu^{-1}))\|_{L^2_{x,v}}^2+\|\langle\nabla_x\rangle^{s/3}(\langle v\rangle^mh_{\neq}(\nu^{-1}))\|_{L^2_{x,v}}^2.
\end{align*} 

\subsection{Secondary instability}\label{sec:sec-insta}
In this section, we investigate the optimality of the stability threshold, that is, we study possible instabilities triggered by perturbations. Our instability analysis is based on the concept of secondary instability.
Roughly speaking, secondary instability refers to the instability of a nonlinear state generated by a primary (linear) instability. In kinetic plasma models, an initially unstable equilibrium may first produce coherent nonlinear structures, such as traveling waves or BGK-type states. These coherent states then serve as new background profiles, and their stability properties determine whether the system undergoes further growth or disturbances. The emergence of secondary instability therefore provides a natural mechanism for the breakdown of nonlinear coherent structures and the onset of more complex, possibly turbulent, dynamics.
In our stability analysis, we have proved that the linearized Vlasov-Poisson-Fokker-Planck (VPFP) operator around a Maxwellian equilibrium is stable under sufficiently small perturbations. More precisely, we obtained the semigroup estimate
\begin{align*}
    \|S_{\mu}(t)\|_{L^2(\langle v\rangle^m)\to L^2(\langle v\rangle^m)}\lesssim e^{-c\nu^{1/3}t}.
\end{align*}
Here, $S_{\mu}(t)$ denotes the solution operator associated with the linearized equation
\begin{align}\label{eq:linearize}
\partial_t f+\mathbb{VP}_{\mu}[f]-\nu L[f]=0,
\end{align}
where
\begin{align*}
    \mathbb{VP}_{\mu}[f]=v\partial_xf+\partial_x\Delta^{-1}_x\int_{\mathbb{R}}f(x,v)dv\cdot \partial_v\mu
\end{align*}

To investigate the sharpness of this stability result, Theorem \ref{thm:main1}, we first construct a spatially homogeneous solution $\tl{f}^{0}(t,v)$ of the VPFP system in a neighborhood of the Maxwellian, see \eqref{eq:tlf}. We then study the dynamics of perturbations around $\tl{f}^0(t,v)$ by introducing a new linearized operator around this time-dependent background, and analyze whether secondary instabilities may arise. Since $\tl{f}_0(t,v)$ varies slowly throughout $t\le \nu^{-1}$, we consider the linearized operator around the time-independent background $\tl{f}^0(v)$ at $t=0$, see \eqref{linear-f}. 
To create an unstable eigenvalue, we perturb the background Maxwellian $\mu(v)$ by a homogeneous distribution  of the form 
\begin{align}\label{def:tlmu}
    \tilde{\mu}(v)=M\gamma \sigma(\frac{v}{\gamma}),
\end{align}
where $\sigma(\cdot)$ is a smooth even function, $M$ and $\gamma$ are positive constants,   all of which are to be determined in section \ref{sec:Insta}.
Notice that $\|\partial_v^k\tilde{\mu}\|_{L^2(\langle v\rangle^m)}\lesssim \gamma^{3/2-k}\|\sigma\|_{H^k(\langle v\rangle^m)}$ for all $k,m=0,1,2,...$ We will take $\gamma=\nu^{\frac{1}{3}-\epsilon_0}$. 
 
We first show that the linearized operator around the new background $f^0(v)$ (see \eqref{def-f0}) with $\nu=0$ (the collisionless case)
\begin{align*}
 \mathbb{VP}_{f^0}[f]=v\partial_xf+\pr_vf^0(v)\partial_x\Delta^{-1}_x\int_{\mathbb{R}}f(x,v)dv
\end{align*}
has an unstable eigenvalue $\lm=\lm_{\rm r}+i\lm_{\rm i}$ with $\lm_{\rm r}\approx \gamma>0$. Let $e_{\gamma, \lm}$ be the associated eigenfunction. Then it holds that (see Lemma \ref{lem-lm} and Lemma \ref{lem:f*})
\begin{align*}
    &\|e^{-t{\mathbb{VP}_{f^0}}}(e^{ix}e_{\gamma, \lm})\|_{L^2(\langle v\rangle^m)}\gtrsim e^{\lm_{\rm r}t}\|e_{\gamma,\lm}\|_{L^2(\langle v\rangle^m)},\\
    &\|\partial_v^{j}(e^{-t{\mathbb{VP}_{f^0}}}(e^{ix}e_{\gamma, \lm}))\|_{L^2(\langle v\rangle^m)}\lesssim \gamma^{-j}e^{C\gamma t}\|e_{\gamma,\lm}\|_{L^2(\langle v\rangle^m)}, \quad j=0,1,2. 
\end{align*}
Once we turn on the collision in the linearized equation $\nu\neq 0$, the error between the solution to the linearized VPFP and VP is under control. Indeed, we have 
\begin{align*}
    \|\nu L[e^{-t{\mathbb{VP}_{f^0}}}(e^{ix}e_{\gamma, \lm})]\|_{L^2(\langle v\rangle^m)}\lesssim \nu \gamma^{-2}e^{\lm_{\rm r} t}\|e_{\gamma,\lm}\|_{L^2(\langle v\rangle^{(m+1)})},
\end{align*}
which gives the lower bounds of the semi-group to the associated collision equation, namely
\begin{align*}
    \|\mathbb{S}(t)\|_{L^2(\langle v\rangle^m)\to L^2(\langle v\rangle^m)}\gtrsim e^{\lm_{\rm r}t} \quad \text{for}\quad t\leq \delta_0\gamma^{-1} \ln\gamma^{-1}.
\end{align*}
See Proposition \ref{prop:lb-S} for more details. 

\subsection{Upper bounds of the perturbed semigroup}
The upper bounds estimate of the perturbed semigroup associated with $f^0$ is also challenging. In the estimates of the distribution function $f(t,x,v)$, we also apply the wave operator to remove the nonlocal part generated by the Maxwellian. To obtain the estimate of density, we study the Volterra equation \eqref{eq:theta_k} and obtain a $\nu,\gamma$-independent Penrose-type stability estimate (see Lemma \ref{lem:Penrose}).

\section{The wave operator}
In this section, we first construct the wave operator for the case of spatial dimension $n=3$, followed by a presentation of its key properties, including boundedness results and commutator estimates. For the lower-dimensional cases of $n=1$ and $n=2$, the construction of the wave operator and derivation of its corresponding properties can be carried out analogously, and with greater simplicity. Finally, as an application of the wave operator, we rewrite the perturbed Vlasov-Poisson-Fokker-Planck equations \eqref{pVPFP} and obtain a representation of the density without solving the Volterra equation. 
\subsection{Construction of the wave operator}
To construct the wave operator, let us consider the homogeneous linearized Vlasov-Possion equation obtained from \eqref{pVPFP}:
\begin{align}\label{Linear-h}
    \begin{cases}
    \pr_th+v\cdot\nb_xh-E(t,x)\cdot v\mu=0,\\[2mm]
    \displaystyle E(t,x)=-\nb_x(-\Dl_x)^{-1}\rho,\quad \rho=\int_{\R^3}h(t,x,v)dv.
    \end{cases}
\end{align}
Taking the Fourier transform with respect to the $x$ variable, we are led to
\begin{align}\label{Linear-h_k}
    \pr_th_k+ik\cdot vh_k+\fr{ik\cdot v}{|k|^2}\rho_k\mu=0,\quad \rho_k=\int_{\R^3}h_k(t,v)dv.
\end{align}
Let us introduce the change of variable
\begin{align}\label{Orthogonal}
    \tl{v}O=v,\quad{\rm such\  that}\quad k\cdot v=|k|\tl{v}_1,
\end{align}
where $O$ is a orthogonal matrix, such that $\det (O)=1$. More precisely, the orthogonal matrix $O$ can be chose as $O=\begin{pmatrix}\fr{k}{|k|}\\v_*\\ w_*\end{pmatrix}$, where $v_*$ and $w_*$ are unit row vectors in $\R^3$ such that $v_*\bot k$, $w_*\bot k$, and $w_*\bot v_*$.

Under this change of variable, we give the unknown $h_k(t,v)$ a new notation:
\begin{align*}
    g_k(t,\tl{v}):=h_k(t,v)=h_k(t,\tl{v}O).
\end{align*}
Moreover, now $\rho_k$ can be given in terms of $g_k(t,\tl{v})$:
\begin{align*}
    \rho_k(t)=\int_{\R^3}h_k(t,\tl{v}O)d\tl{v}=\int_{\R^3}g_k(t,\tl{v})d\tl{v}.
\end{align*}
Noting that $\mu(\tl{v}O)=\mu(\tl{v})$, now \eqref{Linear-h_k} can be rewritten as 
\begin{align}\label{Li-gk}
\pr_tg_k(t,\tl{v})+i|k|\tl{v}_1g_k(t,\tl{v})+\fr{i|k|\tl{v}_1e^{-\frac{|\tl{v}|^2}{2}}}{(2\pi)^{\frac{3}{2}}|k|^2}\int_{\R^3}g_k(t,\tl{v})d\tl{v}=0.
\end{align}

Let us denote 
\[
d_k(v):=\fr{v_1e^{-\frac{|v|^2}{2}}}{(2\pi)^{\fr32}|k|^2}=\fr{v_1\mu(v)}{|k|^2},\quad \mathrm{and}\quad P_k(u_1):=1-\pi\mathcal{H}\circ\mathcal{A}[d_k](u_1),
\]
where $\mathcal{A}[f]$ is the integral of $f$ over $\R^2$ with respect to the $v'=(v_2,v_3)$ variables, and $\mathcal{H}[f]$ is the Hilbert transform of $f$ with respect to the $v_1$ variable, namely,
\[
\mathcal{A}[d_k](v_1):=\int_{\R^2}d_k(v)dv_2dv_3=\fr{v_1e^{-\fr{v_1^2}{2}}}{(2\pi)^{\fr12}|k|^2},\quad \mathcal{H}[f](v_1):=\frac{1}{\pi}{\rm P.V.}\int_{-\infty}^{\infty}\frac{f(u_1)}{v_1-u_1}du_1.
\]
One can compute $P_k(u_1)$ by virtue of the Fourier inversion formula:
\begin{align}\label{exp-P_k}
   P_k(u_1)=\nn&1 -\fr{1}{2}\int_{\R}\widehat{(\mathcal{H}\circ\mathcal{A}[d_k])}(\xi_1)e^{iu_1\xi_1}d\xi_1\\
=&1+\fr12\int_0^\infty\fr{\xi_1e^{-\fr{\xi_1^2}{2}}}{|k|^2}(e^{iu_1\xi_1}+e^{-iu_1\xi_1})d\xi_1.
\end{align}
It is easy to see that $P_k(u_1)\ge0$ for all $|k|\ge1$ due to the fact
\[
\left|\fr12\int_0^\infty\fr{\xi_1e^{-\fr{\xi_1^2}{2}}}{|k|^2}(e^{iu_1\xi_1}+e^{-iu_1\xi_1})d\xi_1\right|\le \int_0^\infty \xi_1e^{-\fr{\xi_1^2}{2}}d\xi_1=1.
\]
Furthermore, noting that the Penrose stability condition holds for the global Maxwellian $\mu$, we thus have  
\[
\inf_{k\in \Z^3\setminus\{0\}}\inf_{u_1\in\R}\left|1+\int_0^\infty \fr{\xi_1e^{-\fr{\xi_1^2}{2}}}{|k|^2}e^{\pm iu_1\xi_1}d\xi_1\right|>0. 
\]
Consequently,
\begin{align}\label{inf-Pk}
    \inf_{k\in \Z^3\setminus\{0\}}\inf_{u_1\in\R}P_k(u_1)>0.
\end{align}

Now we are in a position to define the wave operator for $k\in\Z^3\setminus\{0\}$:
\begin{align}\label{WO}
    \mathbb{D}_k[\mathfrak{f}](u):=\frak{f}(u)+\pi\fr{d_k(u)}{P_k(u_1)} \mathcal{H}\circ\mathcal{A}[\frak{f}](u_1).
\end{align}
Then it is not difficult to verify that
\begin{align}\label{def-WO}
    \mathbb{D}_k\left[i|k|v_1\frak{f}(v)+i|k|d_k(v)\int_{\R^3}\frak{f}(v)dv\right](u)=i|k|u_1\mathbb{D}_k[\frak{f}](u).
\end{align}
In view of \eqref{Li-gk} and \eqref{def-WO},  we find that the equation for $\mathbb{D}_k[g_k]$ takes the form of
\begin{align*}
    \pr_t\mathbb{D}_k[g_k](\tl{v})+i|k|\tl{v}_1\mathbb{D}_k[g_k](\tl{v})=0.
\end{align*}
Recalling \eqref{Orthogonal}, we rewrite this equation for $\mathbb{D}_k[g_k]$ in the original $v$ variable:
\begin{align}\label{eq:g_k}
    \pr_t\mathbb{D}_k[g_k](vO^{-1})+ik\cdot v\mathbb{D}_k[g_k](vO^{-1})=0.
\end{align}
To obtain the wave operator in the original $v$ variable, let us introduce the notation:
\[
\mathbb{O}[\frak{f}](v):=\frak{f}(vO),\quad \mathbb{O}^{-1}[\frak{f}](v):=\frak{f}(vO^{-1}).
\]
Then we define
\begin{align}\label{Def-D}
    {\bf D}_k[\frak{f}](v):=\mathbb{O}^{-1}\circ\mathbb{D}_k\circ\mathbb{O}[\frak{f}](v).
\end{align}
Recalling that $g_k(t,\tl{v})=h_k(t,\tl{v}O)$, one can see from \eqref{eq:g_k} that the equation for ${\bf D}_k[h_k](v)$ takes the form of
\begin{align}\label{eq-Dh}
    \pr_t{\bf D}_k[h_k](v)+ik\cdot v{\bf D}_k[h_k](v)=0.
\end{align}

 The precise representation of ${\bf D}_k[\frak{f}](v)$ is now given explicitly. Noting that $(vO^{-1})_1=\fr{k\cdot v}{|k|}$, from \eqref{WO} and \eqref{Def-D}, we have
\begin{align}\label{exp-D_k}
    {\bf D}_k[\frak{f}](v)=\frak{f}(v)+\pi\fr{d_k(vO^{-1})}{P_k(\fr{k\cdot v}{|k|})}\mathcal{H}\circ\mathcal{A}\circ\mathbb{O}[\frak{f}]\big(\fr{k\cdot v}{|k|}\big)
    =\frak{f}(v)+{\bf T}_k[\frak{f}](v),
\end{align}
where 
\begin{align}\label{def-T_k}
    {\bf T}_k[\frak{f}](v):=\fr{k\cdot v\mu(v)}{|k|^3P_k(\fr{k\cdot v}{|k|})}{\rm P.V.}\int_{-\infty}^{\infty}\fr{\int_{\R^2}\frak{f}(uO)du_2du_3}{\fr{k\cdot v}{|k|}-{u}_1}d{u}_1.
\end{align}

We finally define 
\begin{align}
    {\bf D}[f](t, x, v)=\frac{1}{(2\pi)^3}\int_{\mathbb{T}^3}f(t,x, v)dx+\frac{1}{(2\pi)^3}\sum_{k\in \mathbb{Z}^3\setminus \{0\}}{\bf D}_k[f_k(t)](v)e^{ikx}.
\end{align}
\subsection{Inverse of the wave operator}\label{sec:inverseWO}
We first solve $\frak{f}$ from \eqref{WO} in terms of $\mathbb{D}_k[\frak{f}]$. To begin with, let us denote 
\begin{align}\label{def-QW}
Q_k(u_1):=\pi \mathcal{A}[d_k](u_1)=\sqrt{\frac{\pi}{2}}\fr{u_1e^{-\fr{u_1^2}{2}}}{|k|^2},\quad W_k(u_1):=P_k^2(u_1)+Q_k^2(u_1). 
\end{align}
Applying $\mathcal{A}$ to \eqref{WO} yields
\begin{align}\label{AD}
    \mathcal{A}\circ\mathbb{D}_k[\frak{f}]=\mathcal{A}[\frak{f}]+ \fr{Q_k}{P_k}\mathcal{H}\circ\mathcal{A}[\frak{f}].
\end{align}
Note that $P_k(u_1)+iQ_k(u_1)=1-\pi\mathcal{H}\circ\mathcal{A}[d_k](u_1)+i\pi \mathcal{A}[d_k](u_1)$ is the trace of an analytic function, and so is $\mathcal{A}[f]+i\mathcal{H}\circ\mathcal{A}[f]$. Then
\begin{align*}
\fr{\mathcal{A}[\frak{f}]+i\mathcal{H}\circ\mathcal{A}[\frak{f}]}{P_k+iQ_k}=\fr{P_k\mathcal{A}\circ\mathbb{D}_k[\frak{f}]}{W_k}+i\fr{P_k\mathcal{H}\circ\mathcal{A}[\frak{f}]-Q_k\mathcal{A}[\frak{f}]}{W_k}
\end{align*}
is the trace of an analytic function as well, provided that the analytic extension of $P_k(u_1)+iQ_k(u_1)$ 
\begin{align*}
    P_k(z)+iQ_k(z)
\end{align*}
has no zeros for $\frak{Im}\, z\geq 0$,  which is equivalent to the Penrose's stability condition
\begin{align*}
\inf_{k\in \Z^3\setminus\{0\}}\inf_{\frak{Im}\, z\geq 0}\left|1+\int_0^\infty \fr{\xi_1e^{-\fr{\xi_1^2}{2}}}{|k|^2}e^{iz\xi_1}d\xi_1\right|>0. 
\end{align*} 
 Then
\begin{align}\label{HPAD}
    \mathcal{H}\left[\fr{P_k\mathcal{A}\circ\mathbb{D}_k[\frak{f}]}{W_k}\right]=\fr{P_k\mathcal{H}\circ\mathcal{A}[\frak{f}]-Q_k\mathcal{A}[\frak{f}]}{W_k}.
\end{align}
Combining \eqref{AD} and \eqref{HPAD}, we have
\begin{align*}
\begin{cases}
    \mathcal{H}\circ\mathcal{A}[\frak{f}]=\fr{P_k}{W_k}\left(W_k\mathcal{H}\left[\fr{P_k\mathcal{A}\circ\mathbb{D}_k[\frak{f}]}{W_k}\right]+Q_k\mathcal{A}\circ\mathbb{D}_k[\frak{f}]\right),\\[3mm]
    \mathcal{A}[\frak{f}]=\fr{-Q_kW_k\mathcal{H}\left[\fr{P_k\mathcal{A}\circ\mathbb{D}_k[\frak{f}]}{W_k}\right]+P_k^2\mathcal{A}\circ\mathbb{D}_k[\frak{f}]}{W_k}.
    \end{cases}
\end{align*}
Substituting this into \eqref{WO}, we are led to
\begin{align}\label{inverseWO}
    \frak{f}=\mathbb{D}_k[\frak{f}]-\pi d_k\left(\mathcal{H}\left[\fr{P_k\mathcal{A}\circ\mathbb{D}_k[\frak{f}]}{W_k}\right]+\fr{Q_k}{W_k}\mathcal{A}\circ\mathbb{D}_k[\frak{f}]\right).
\end{align}

Now we rewrite \eqref{inverseWO}  so that $\frak{f}$ can be solved in terms of ${\bf D}_k[\frak{f}]$. In fact, applying $\mathbb{O}^{-1}$ to \eqref{inverseWO} with $\frak{f}$ replaced by $\mathbb{O}[\frak{f}]$, and using \eqref{Def-D}, we find that
\begin{align}\label{exp-D_kinverse}
\frak{f}(v)=\nn&\mathbb{O}^{-1}\circ\mathbb{D}_k\circ\mathbb{O}[\frak{f}]-\pi \mathbb{O}^{-1}\Bigg[d_k \Bigg(\mathcal{H}\left[\fr{P_k\mathcal{A}\circ\mathbb{O}\big[{\bf D}_k[\frak{f}]\big]}{W_k}\right]+\fr{Q_k}{W_k}\mathcal{A}\circ \mathbb{O}\big[{\bf D}_k[\frak{f}]\big]\Bigg)\Bigg](v)\\
=&{\bf D}_k[\frak{f}](v)+\tl{{\bf T}}_k\big[{\bf D}_k[\frak{f}]\big](v),
\end{align}
where
\begin{align}\label{def-tlT_k}
    \tl{{\bf T}}_k\big[{\bf D}_k[\frak{f}]\big](v):=\nn&-\pi\fr{k\cdot v}{|k|^3}\mu(v) \Bigg[\fr{1}{\pi}{\rm P.V.} \int_{-\infty}^{\infty}\fr{\fr{P_k(u_1)}{W_k(u_1)}\int_{\R^2}{\bf D}_k[\frak{f}](uO)du_2du_3}{\fr{k\cdot v}{|k|}-u_1}du_1\\
    &+\fr{Q_k(\fr{k\cdot v}{|k|})}{W_k(\fr{k\cdot v}{|k|})}\int_{\R^2}{\bf D}_k[\frak{f}]\left(\left(\fr{k\cdot v}{|k|}, u_2, u_3\right)O\right) du_2du_3\Bigg].
\end{align}

\subsection{Properties of the wave operator}\label{sec:prop-WO}
In this section, we first establish the following continuity results for the operators ${\bf D}_k$ and ${\bf D}_k^{-1}$ for all $k\in\Z^3\setminus\{0\}$.
\begin{prop}[boundedness of ${\bf D}_k$ and ${\bf D}_k^{-1}$]\label{prop-continuity}
   For any multi-indices  $\alpha, \beta$, the following hold:
    \begin{align}
   \label{offD} &\|\pr^\beta_v(v^\al {\bf D}_k[f])\|_{L^2_v}\le \|\pr^\beta_v(v^\al f)\|_{L^2_v}+C_{\al,\beta}\sum_{\beta'\le\beta}\|\pr_{v}^{\beta'}(\la v\ra^{1+}f)\|_{L^2_{v}},\\
   \label{onD}&\|\pr_v^\beta(v^\al f)\|_{L^2_v}\le \|\pr_v^{\beta}(v^\al{\bf D}_k[f])\|_{L^2_v}+C_{\al,\beta}\sum_{\beta'\le\beta}\|\pr_v^{\beta'}(\la v\ra^{1+}{\bf D}_k[f])\|_{L^2_v},
    \end{align}
here $C$ is a constant independent of $k$. 
\end{prop}
\begin{proof}
To simplify the presentation, we use the notation introduced in Appendix \ref{app:calcu}. 
Recalling \eqref{def-T_k}, we write 
\begin{align}\label{Tk'}
    {\bf T}_k[f](v)=\fr{\pi}{|k|^2}A_k(v)\Phi_{2,k}(\frac{k\cdot v}{|k|})\mathbb{T}_k[f](v),
\end{align}
    where 
    \begin{align*}
        \mathbb{T}_k[f](v):=\mathbb{O}^{-1}\circ \mathcal{H}\circ\mathcal{A}\circ\mathbb{O}[f](v).
    \end{align*}
Since $O$ is orthogonal, we have
\[
(\pr_{v_j}f)(uO)=\sum_{i=1}^3O_{ij}\pr_{u_i}[f(uO)].
\]
Consequently,
\begin{align}\label{pr-AO}
\mathcal{A}\circ\mathbb{O}[\pr_{v_j}f](u_1)=O_{1j}\fr{d}{du_1}\mathcal{A}\circ\mathbb{O}[f](u_1)=\fr{k_j}{|k|}\fr{d}{du_1}\mathcal{A}\circ\mathbb{O}[f](u_1),
\end{align}
and
\begin{align*}
    \mathbb{T}_k[\pr_{v_j}f](v)=\fr{k_j}{|k|}\mathcal{H}\big[(\mathcal{A}\circ\mathbb{O}[f])'\big](\fr{k\cdot v}{|k|})=\pr_{v_j}\mathbb{T}_k[f](v).
\end{align*}
Therefore, for any $\beta\in\N^3$, there holds
\begin{align}\label{com0}
    [\pr_v^\beta,\mathbb{T}]=0.
\end{align}
Combining this with \eqref{bd-w}, we have
\begin{align*}
    |\pr_v^\beta (v^\al {\bf T}_k[f])|\le\nn \fr{C_{\al,\beta}}{|k|^2}\sum_{\beta'\le\beta}e^{-c_{\al,\beta}|v|^2}\bigg|\mathcal{H}\circ\mathcal{A}\circ\mathbb{O}\big[\pr_{v}^{\beta'}f\big](\fr{k\cdot v}{|k|})\bigg|.
\end{align*}
Then applying the change of variable \eqref{Orthogonal}, we are led to
\begin{align}\label{est:Tkalbeta}
\|\pr_v^\beta (v^\al {\bf T}_k[f])\|_{L^2_v}
\le\nn&\fr{C_{\al,\beta}}{|k|^2}\sum_{\beta'\le\beta}\bigg\|e^{-c_{\al,\beta}|\tl{v}|^2}\mathcal{H}\circ\mathcal{A}\circ\mathbb{O}\big[\pr_{v}^{\beta'}f\big](\tl{v}_1)\bigg\|_{L^2_{\tl{v}}}\\
\le\nn&\fr{C_{\al,\beta}}{|k|^2}\sum_{\beta'\le\beta}\big\|\mathcal{A}\circ\mathbb{O}\big[\pr_{v}^{\beta'}f\big]\big\|_{L^2_{\tl v_1}}\\
\le&\fr{C_{\al,\beta}}{|k|^2}\sum_{\beta'\le\beta}\|\la v\ra^{1+}\pr_{v}^{\beta'}f\|_{L^2_{v}}\le \fr{C_{\al,\beta}}{|k|^2}\sum_{\beta'\le\beta}\|\pr_{v}^{\beta'}(\la v\ra^{1+}f)\|_{L^2_{v}},
\end{align}
where we have used
\begin{align*}
    &\|\mathcal{A}\circ\mathbb{O}[\pr_{v}^{\beta'}f]\|_{L^2_{\tl v_1}}^2=\int_{\R}\Big|\int_{\R^2}\pr_{v}^{\beta'}f(uO)du'\Big|^2du_1\\
    \les&\int_{\R}\int_{\R^2}\big|\la uO\ra^{1+}\pr_{v}^{\beta'}f(uO)\big|^2du'du_1=\|\la v\ra^{1+}\pr_{v}^{\beta'}f\|_{L^2_{v}}^2.
\end{align*}
Then \eqref{offD} follows immediately.

Now we turn to prove \eqref{onD}. Denote ${\bf f}_k={\bf D}_k[f]$, i.e., $f={\bf D}_k^{-1}[{\bf f}_k]$. Recalling \eqref{def-tlT_k}, we write
\begin{align}
    \tl{{\bf T}}_k[{\bf f}_k](v)=\nn&-\pi A_k(v)\mathcal{H}\big[\Phi_{4,k}\mathcal{A}\circ\mathbb{O}[{\bf f}_k]\big] (\fr{k\cdot v}{|k|})-\pi w_{3,k}(v)\mathcal{A}\circ\mathbb{O}[{\bf f_k}](\fr{k\cdot v}{|k|}).
\end{align}
Using \eqref{pr-AO}, one deduces  that
\begin{align*}
    \pr_v^\beta(v^\al\tl{\bf T}_k[{\bf f}_k])(v)
=I_1^{\al,\beta}(v)+I_2^{\al,\beta}(v),
\end{align*}
where 
\begin{align*}
    I_1^{\al,\beta}(v):=&-\fr{\pi}{|k|^2}\sum_{ \beta'\le \beta,\beta''\le\beta'}C_\beta^{\beta'}C_{\beta'}^{\beta''} w_{1,k}^{\al,\beta-\beta'}(v)\Big(\frac{k}{|k|}\Big)^{\beta'-\beta''}\mathcal{H}\big[\Phi_{4,k}^{(|\beta'-\beta''|)}\mathcal{A}\circ\mathbb{O}[\pr_v^{\beta''}{\bf f}_k]\big](\fr{k\cdot v}{|k|}),
\end{align*}
and
\begin{align*}
    I_2^{\al,\beta}(v):=&-\fr{\pi}{|k|^2} \sum_{\beta'\le \beta}C_\beta^{\beta'}w^{\al,\beta-\beta'}_{3,k}(v)\mathcal{A}\circ\mathbb{O}\big[\pr_{v}^{\beta'} {\bf f}_k\big]
     (\frac{k \cdot v}{|k|}).
\end{align*}
Using \eqref{bd-w} and \eqref{est:Phi}, 
similar to \eqref{est:Tkalbeta}, we arrive at
\begin{align}\label{est:I1}
    \|I_1^{\al,\beta}\|_{L^2_v}
    \le\nn&\frac{C_{\al,\beta}}{|k|^2}\sum_{\beta'\le\beta}\sum_{\beta''\le\beta'}\left\|\Phi_{4,k}^{|\beta'-\beta''|}\mathcal{A}\circ\mathbb{O}\big[\pr_{v}^{\beta''}{\bf f}_k\big]\right\|_{L^2_{\tl{v}_1}}\\
    \le&\frac{C_{\al,\beta}}{|k|^2}\sum_{\beta''\le\beta}\left\|\la v\ra^{1+}\pr_v^{\beta''}{\bf f}_k\right\|_{L^2_{v}}\le \fr{C_{\al,\beta}}{|k|^2}\sum_{\beta'\le\beta}\|\pr_{v}^{\beta'}(\la v\ra^{1+}{\bf f}_k)\|_{L^2_{v}}.
\end{align}
Moreover, using \eqref{bd-w} again, we find that the upper bound on the right-hand side of \eqref{est:I1} still holds for $I_2^{\al,\beta}$.
We thus get \eqref{onD}, and the proof of Proposition \ref{prop-continuity} is completed.
\end{proof}

The aim of the next lemma is to establish the  commutator estimates of $[{\bf D}_k,L]$ for $k\in\Z^3_*$.

\begin{prop}[commutator estimates]\label{prop-com}
For $k\in\Z^3_*$, the following commutator estimates hold:
    \begin{align}\label{bd:com1}
    \|\pr^\beta_v(v^\al [\nb_{v},{\bf D}_k]f)\|_{L^2_v}\le \fr{C_{\al,\beta}}{|k|^2}\sum_{\beta'\le\beta}\|\pr_{v}^{\beta'}(\la v\ra^{1+}f)\|_{L^2_{v}},
    \end{align}

\begin{align}\label{bd:com2}
    \|\pr_v^\beta(v^{\alpha}[\Delta_v, \mathbf{D}_k][f])\|_{L^2_v}\le\fr{C_{\al,\beta}}{|k|^2}\sum_{|\beta'|\le|\beta|+1}\|\pr_{v}^{\beta'}(\la v\ra^{1+}f)\|_{L^2_{v}},
\end{align}

    \begin{align}\label{bd:com3}
    \|\partial_v^\beta(v^\al[\Lambda,\mathbf{D}_k]f)\|_{L^2_v}\le\fr{C_{\al,\beta}}{|k|^2}\sum_{\beta'\le\beta}\|\pr_{v}^{\beta'}(\la v\ra^{1+}f)\|_{L^2_{v}}.
    \end{align}
    where $\Lambda[f]:=\nb_v\cdot(vf)$.
\end{prop}
\begin{proof}
    By \eqref{Tk'} and \eqref{com0}, we write
    \begin{align*}
        &[\pr_{v_j},{\bf D}_k]f=[\pr_{v_j},{\bf T}_k]f
        =\fr{\pi}{|k|^2}\pr_{v_j}\left(A_k(v)\Phi_{2,k}(\frac{k\cdot v}{|k|})\right)\mathbb{T}_k[f](v),
    \end{align*}
and hence
\begin{align*}
    v^\al [\pr_{v_j},{\bf D}_k]f=\fr{\pi}{|k|^2}\left[\pr_{v_j}\left(A^\al_k(v)\Phi_{2,k}(\frac{k\cdot v}{|k|})\right)-\al_j \left(A^{\al-e_j}_k(v)\Phi_{2,k}(\frac{k\cdot v}{|k|})\right)\right]\mathbb{T}_k[f](v).
\end{align*}
Then by \eqref{bd-w}, similar to \eqref{est:Tkalbeta}, we have
\begin{align}
    \nn&\|\pr^\beta_v(v^\al [\pr_{v_j},{\bf D}_k]f)\|_{L^2_v}\\
    \le\nn&\frac{C_{\al,\beta}}{|k|^2}\sum_{\beta'\le\beta}\left(\left\| w_{2,k}^{\al,\beta-\beta'+e_j}\mathbb{T}_k[\pr_v^{\beta'}f]\right\|_{L^2_v}+\left\| w_{2,k}^{\al-e_j,\beta-\beta'}\mathbb{T}_k[\pr_v^{\beta'}f]\right\|_{L^2_v}\right)\\
    \le\nn& \fr{C_{\al,\beta}}{|k|^2}\sum_{\beta'\le\beta}\|\pr_{v}^{\beta'}(\la v\ra^{1+}f)\|_{L^2_{v}}.
\end{align}
This proves \eqref{bd:com1}.

A direct calculation gives us that 
    \begin{align*}
        [\Delta_v, \mathbf{D}_k][f]
        =\sum_{i=1}^n\partial_{v_i} \big([\partial_{v_i},{\bf T}_k][f]\big)+\sum_{i=1}^n[\partial_{v_i},{\bf T}_k][\pr_{v_i}f].
    \end{align*}
Thanks to \eqref{bd:com1}, we have
\begin{align}
    \nn&\|\pr_v^\beta(v^{\alpha}[\Delta_v, \mathbf{D}_k][f])\|_{L^2_v}\\
    \leq\nn&\sum_{i=1}^n\Big(\|\pr_v^{\beta+e_i}(v^\al [\pr_{v_i},{\bf T}_k][f])\|_{L^2_v}+\|\pr_v^{\beta}(\pr_{v_i}v^\al [\pr_{v_i},{\bf T}_k][f])\|_{L^2_v}\Big)\\
    \nn&+\sum_{i=1}^n\|\pr_v^\beta(v^{\alpha}[\pr_{v_i}, \mathbf{T}_k][\pr_{v_i}f])\|_{L^2_v}\\
    \le\nn&\frac{C_{\al,\beta}}{|k|^2}\sum_{|\beta'|\le|\beta|+1}\|\pr_{v}^{\beta'}(\la v\ra^{1+}f)\|_{L^2_{v}}.
\end{align}
Thus, \eqref{bd:com2} holds.

Finally, we turn to prove \eqref{bd:com3}.
Noting that $\Lambda \circ\mathbb{O}=\mathbb{O}\circ\Lambda$ and $\Lambda\circ\mathbb{O}^{-1}=\mathbb{O}^{-1}\circ\Lambda$, we have 
\begin{align*}
\mathbb{T}_k\circ \Lambda [f](v)- \Lambda\circ\mathbb{T}_k [f](v)
=\mathbb{O}^{-1}\circ [\mathcal{H}\circ\mathcal{A},\Lambda]\circ\mathbb{O}[f](v).
\end{align*}
So it suffices to compute $[\mathcal{H}\circ\mathcal{A},\Lambda]$.
For $g=\mathbb{O}[f]$,
\begin{align*}
    \mathcal{A}\circ\Lambda [g](u)=&\int_{\R^2}\nb_u\cdot(ug(u))du_2du_3=\fr{d}{du_1}\left(u_1\mathcal{A}[g](u_1)\right)=\Lambda_1\circ\mathcal{A}[g](u_1),
\end{align*}
where $\Lambda_1f=\fr{d}{du_1}(u_1f)$.
Then
\begin{align*}
    \mathcal{H}\circ\mathcal{A}\circ\Lambda[g]=&\mathcal{H}\circ\Lambda_1\circ\mathcal{A}[g].
\end{align*}
On the other hand, 
\begin{align*}
    \Lambda\circ\mathcal{H}\circ\mathcal{A}[g]=&\nb_u\cdot(u\mathcal{H}\circ\mathcal{A}[g])=2\mathcal{H}\circ\mathcal{A}[g]+\Lambda_1\circ\mathcal{H}\circ\mathcal{A}[g].
\end{align*}
It is easy to verify that $[\Lambda_1,\mathcal{H}]=0$. Then,
\begin{align}
    [\Lambda,\mathbb{T}_k]=2\mathbb{T}_k.
\end{align}
Accordingly,
\begin{align*}
    [\Lambda,{\bf T}_k]f=\fr{\pi}{|k|^2}v\cdot \nb_v\left(A_k(v)\Phi_{2,k}(\fr{k\cdot v}{|k|})\right)\mathbb{T}_k[f]+2\fr{\pi}{|k|^2}\left(A_k(v)\Phi_{2,k}(\fr{k\cdot v}{|k|})\right)\mathbb{T}_k[f].
\end{align*}
Then, similar to the proof of \eqref{bd:com1}, we find \eqref{bd:com3} holds.
\end{proof}

\subsection{Reformulation via wave operator}\label{sec-refor}

In this subsection, we first rewrite system \eqref{pVPFP} by applying the wave operator ${\bf D}_k$ mode by mode. Then we give a representation of the density $\rho$ without solving the Volterra equation. 

\subsubsection{Representation of the distribution function}
To begin with, taking the Fourier transform in $x$ of the equation for $g$ in \eqref{VPFP-perturbation}, we are led to
\begin{align}\label{eq-fk}
    \pr_tg_k+ik\cdot vg_k+i\fr{k\cdot v}{|k|^2}\rho_k\mu-\nu L[g_k]=N[g_k],
\end{align}
where 
\begin{align}\label{nonlinearity}
\rho_k(t)=\int_{\R^3}g_{k}(t,v)dv\quad{\rm and}\quad  N[g_k]=\sum_{l\in\Z^n_*}\rho_l\frac{il}{|l|^2}\cdot \nb_v{g}_{k-l}.
\end{align}
Let us denote $h_k:=\mathbf{D}_k[{g}_k]$. Then in view of \eqref{eq-Dh}, for any $k\in\Z^n_*$, applying the wave operator $\mathbf{D}_k$ to \eqref{eq-fk} yields
\begin{align}\label{eq-varfk}
    \pr_th_k+ik\cdot vh_k-\nu L[h_k]=\nu[\mathbf{D}_k,L][g_k]+\mathbf{D}_k\big[N[g_k]\big],
\end{align}
where  
\begin{align}\label{DNg}
    {\bf D}_k\big[N[g_k]\big](t,v)=\sum_{l\in\Z^n_*}\rho_l\fr{il}{|l|^2}\cdot {\bf D}_{ k}[\nb_vg_{ k-l}](t,v).
\end{align}

Next taking the Fourier transform of \eqref{eq-varfk} in $v$ variable, we have
\begin{align}\label{eq-hatf}
    \pr_t\hat{h}_k-k\cdot\nb_{\xi}\hat{h}_k+\nu\xi\cdot\nb_\xi\hat{h}_k +\nu |\xi|^2\hat{h}_k=\mathcal{F}_v\big[\mathbf{D}_k\big[N[g_k]\big]\big](\xi)+\nu\mathcal{F}_v\big[[\mathbf{D}_k,L][g_k]\big](\xi).
\end{align}
To absorb the  transport term $-k\cdot\nb_{\xi}\hat{h}_k$, we denote $\bar{\eta}(t;k,\eta)=\eta-kt$,  and define
\begin{align}\label{def-fw}
\hat{f}_k(t,\eta)=\hat{h}_k(t, \bar{\eta}(t; k,\eta)), \quad \mathrm{or} \quad \hat{h}_k(t, \xi)=\hat{f}_k\left(t, \xi+kt \right).
\end{align}
Then
$\hat{f}_k(t,\eta)$ solves
\begin{align}\label{eq-hk}
    \pr_t\hat{f}_k(t,\eta)+\nu|\bar{\eta}|^2\hat{f}_k(t,\eta)+\nn&\nu \bar{\eta}\cdot\nb_\eta \hat{f}_k(t,\eta)\\
    =&\mathcal{F}_v\big[\mathbf{D}_k\big[N[g_k]\big]\big](\bar{\eta}(t;k,\eta))+\nu\mathcal{F}_v\big[[\mathbf{D}_k,L][g_k]\big](\bar{\eta}(t;k,\eta)).
\end{align}

Let
\begin{align}\label{S1}
S_k(t,\tau; \eta):=\nn&\exp\left(-\nu\int_\tau^t|\bar{\eta}(s;k,\eta)|^2ds\right)\\
=&\exp\left(-\nu \Big[|\bar{\eta}(t;k,\eta)|^2(t-\tau)+\bar{\eta}(t;k,\eta)\cdot k(t-\tau)^2+\fr13|k|^2(t-\tau)^3\Big]\right).
\end{align}
It is not difficult to verify that
\begin{align}\label{semigroup-est}
    S_k(t,\tau;\eta)\le \exp\left(-\Big[\fr{\nu}{8}|\bar{\eta}(t;k,\eta)|^2(t-\tau)+\fr{\nu}{21}|k|^2(t-\tau)^3\Big]\right).
\end{align}
In the following, we denote $S_k(t,0;\eta)$ by $S_k(t;\eta)$ for simplicity. Then $S_k(t,\tau;\eta)=S_k(t-\tau;\eta)$. 
Now take $\eta=kt$ in \eqref{S1}, and denote 
\begin{align}\label{S2}
S_k(t-\tau):=&S_k\left(t-\tau;  kt\right)
=\exp\left(-\fr13\nu|k|^2(t-\tau)^3\right).
\end{align}
In particular,
\be\label{S3}
S_k(t)=\exp\left(-\fr{1}{3}\nu|k|^2t^3\right).
\ee

By Duhamel's principle, the solution $\hat{f}_k(t,\eta)$ to \eqref{eq-hk} can be given as
\begin{align}\label{ex-h}
\hat{f}_k(t,\eta)\nn=&S_k(t;\eta)(\widehat{f_{\mathrm{in}}})_k(\eta)+\nu\int_0^tS_k(t-\tau; \eta)\mathcal{F}_v\big[[\mathbf{D}_k,L][g_k]\big](\tau,\bar{\eta}(\tau;k,\eta))d\tau\\
\nn&-\nu\int_0^t S_k(t-\tau;\eta)\bar{\eta}(\tau;k,\eta)\cdot\nb_\eta\hat{f}_k(\tau,\eta)d\tau\\
&+\int_0^tS_k(t-\tau;\eta)\mathcal{F}_v\big[\mathbf{D}_k\big[N[g_k]\big]\big](\tau,\bar{\eta}(\tau;k,\eta))d\tau.
\end{align}

\subsubsection{Representation of  the density}
Now we derive the representation of the density $\rho$. Recalling that $P_k=1-\mathcal{H}[Q_k]$, from \eqref{def-tlT_k}, we find that
\begin{align*}
\int_{\R^3}\tl{{\bf T}}_k[h_k](v)dv
=&-\int_{\R}Q_k(u_1)\Big(\mathcal{H}\Big[\frac{P_k}{W_k}\mathcal{A}\circ \mathbb{O}[h_k]\Big]+\fr{Q_k}{W_k}\mathcal{A}\mathcal\circ\mathbb{O}[h_k]\Big)(u_1)du_1\\
=&\int_{\R}\Big(\mathcal{H}[Q_k]\frac{P_k}{W_k}-\fr{Q_k^2}{W_k}\Big)(u_1)(\mathcal{A}\mathcal\circ\mathbb{O}[h_k])(u_1)du_1\\
=&\int_{\R}\Big(\frac{P_k}{W_k}-1\Big)(u_1)(\mathcal{A}\mathcal\circ\mathbb{O}[h_k])(u_1)du_1.
\end{align*}
Then from \eqref{exp-D_kinverse},   using Planchrel's theorem and \eqref{def-fw}, we have
\begin{align}\label{def-Tk}
    \nn\rho_k(t)
    =\nn&\int_{\R^3}h_k(t,v)dv+\int_{\R^3}\tl{{\bf T}}_k[h_k](t,v)dv\\
    =\nn&\hat{f}_k(t,kt)+\int_{\R}\bar{\hat{K}}_k(\xi_1)\mathcal{F}\big[\mathcal{A}\circ\mathbb{O} [h_k]\big](t,\xi_1)d\xi_1\\
    =&\hat{f}_k(t,kt)+\int_{\R}\bar{\hat{K}}_k(\xi_1)\hat{f}_k\big(t,\fr{k}{|k|}\xi_1+kt\big)d\xi_1,
\end{align}
where we have used 
\begin{align*}
    \mathcal{F}\big[\mathcal{A}\circ\mathbb{O} [h_k]\big](t,\xi_1)=&\int_{\R}e^{-i\xi_1u_1}\int_{\R^2}h_k(t,uO)du'du_1=\int_{\R^3}e^{-i\xi_1\fr{k\cdot v}{|k|}}h_k(t,v)dv\\
    =&\hat{h}_k(t,\fr{k}{|k|}\xi_1)=\hat{f}_k(t,\fr{k}{|k|}\xi_1+kt),
\end{align*}
and
\begin{align}\label{def-K}
    K_k(u_1):=\Big(\fr{P_k}{W_k}-1\Big)(u_1).
\end{align}

Let us denote $\xi_*(t,k)=\fr{k}{|k|}\xi_1+kt$.
Note that 
\begin{gather}\label{bar-eta-tau}
\bar{\eta}(t;k,kt)=0,\\
\label{bar-eta-ktap}\bar{\eta}(\tau;k,kt)=k(t-\tau),\\
\label{bar-eta-t-xi}\bar{\eta}(t,k,\xi_*(t,k))=\xi_1\fr{k}{|k|},\\
\bar{\eta}(\tau,k,\xi_*(t,k))=\xi_*(t-\tau,k).
\end{gather}
Taking $\eta=kt$ in \eqref{ex-h}, we get
\begin{align}\label{ex-h-var1}
\hat{f}_k(t,kt)\nn=&S_k(t)(\widehat{f_{\mathrm{in}}})_k(kt)+\nu\int_0^tS_k(t-\tau)\mathcal{F}_v\big[[\mathbf{D}_k,L][g_k]\big](\tau,k(t-\tau))d\tau\\
\nn&-\nu\int_0^t S_k(t-\tau)k(t-\tau)\cdot\nb_\eta\hat{f}_k(\tau,kt)d\tau\\
&+\int_0^tS_k(t-\tau)\mathcal{F}_v\big[\mathbf{D}_k\big[N[g_k]\big]\big](\tau,k(t-\tau))d\tau.
\end{align}
In addition, taking $\eta=\xi_*(t,k)$ in \eqref{ex-h} yields
\begin{align}\label{ex-h-var2}
\hat{f}_k(t,\xi_*(t,k))\nn=&S_k(t;\xi_*(t,k))(\widehat{f_{\mathrm{in}}})_k(\xi_*(t,k))\\
\nn&-\nu\int_0^t S_k(t-\tau;\xi_*(t,k))\xi_*(t-\tau,k)\cdot\nb_\eta\hat{f}_k(\tau,\xi_*(t,k))d\tau\\
\nn&+\nu\int_0^tS_k(t-\tau;\xi_*(t,k))\mathcal{F}_v\big[[\mathbf{D}_k,L][g_k]\big](\tau,\xi_*(t-\tau,k))d\tau\\
&+\int_0^tS_k(t-\tau;\xi_*(t,k))\mathcal{F}_v\big[\mathbf{D}_k\big[N[g_k]\big]\big](\tau,\xi_*(t-\tau,k))d\tau.
\end{align}
Substituting \eqref{ex-h-var1} and \eqref{ex-h-var2} into \eqref{def-Tk}, we get the representation for $\rho_k$:
\begin{align}\label{rep-rhok}
\rho_k(t)=\nn&S_k(t)(\widehat{h_{\mathrm{in}}})_k(kt)+\nu\int_0^tS_k(t-\tau)\mathcal{F}_v\big[[\mathbf{D}_k,L][g_k]\big](\tau,k(t-\tau))d\tau\\
\nn&-\nu\int_0^t S_k(t-\tau)k(t-\tau)\cdot\nb_\xi\hat{h}_k(\tau,k(t-\tau))d\tau\\
&\nn+\int_0^tS_k(t-\tau)\mathcal{F}_v\big[\mathbf{D}_k\big[N[g_k]\big]\big](\tau,k(t-\tau))d\tau\\
\nn&+\int_{\R}\bar{\hat{K}}_k(\xi_1)S_k(t;\xi_*(t,k))(\widehat{h_{\mathrm{in}}})_k(\xi_*(t,k))d\xi_1\\
\nn&-\nu\int_0^t\int_{\R}\bar{\hat{K}}_k(\xi_1) S_k(t-\tau;\xi_*(t,k))\xi_*(t-\tau,k)\cdot\nb_\xi\hat{h}_k(\tau,\xi_*(t-\tau,k))d\xi_1d\tau\\
\nn&+\nu\int_0^t\int_{\R}\bar{\hat{K}}_k(\xi_1)S_k(t-\tau;\xi_*(t,k))\mathcal{F}_v\big[[\mathbf{D}_k,L][g_k]\big](\tau,\xi_*(t-\tau,k))d\xi_1d\tau\\
&+\int_0^t\int_{\R}\bar{\hat{K}}_k(\xi_1)S_k(t-\tau;\xi_*(t,k))\mathcal{F}_v\big[\mathbf{D}_k\big[N[g_k]\big]\big](\tau,\xi_*(t-\tau,k))d\xi_1d\tau.
\end{align}

\begin{remark}
    One can also estimate $\rho$ by solving the Volterra equation. Here we introduce a different approach, avoiding dealing with the Volterra equation and using the Laplace transform, which might be useful in other problems. 
\end{remark}

\section{Stability estimates}\label{sec:sta}
A standard argument gives the local well-posedness for the Vlasov-Poisson-Fokker-Planck system. Here, our goal is to prove by a continuity argument that $\mathcal{E}^s_m(h(t))$ (together with the $L^2_tL^2_x$ norm of the density $\rho$) is uniformly bounded for all time $t\le\nu^{-1}$ whenever  $\eps$ is sufficiently small. Let ${\rm C}_1$, ${\rm C}_2$ and ${\rm C}_3$  be positive constants independent of $t$ and $\nu$, determined by the proof below. We define the following controls referred to in the sequel as the {\it bootstrap hypotheses} for all $T\le \nu^{-1}$:
\begin{subequations}\label{hypo}
\begin{align}
\label{hypo1}\|{\bf B}_{s}\rho\|_{L^2_tL^2_x}\le& 4{\rm C}_1\eps \nu^{\fr12},\\
\label{hypo2}\sup_{0\le t\le T}\mathcal{E}^s_m(h_{\ne}(t))+\fr{1}{4}\int_0^{T}\mathcal{D}^s_m(h_{\ne}(\tau))d\tau\le& (4{\rm C}_{2}\eps\nu^{\fr12})^2,\\
\label{hypo3}\sup_{0\le t\le T}\mathcal{E}^s_m(h_0(t))+\fr{1}{4}\int_0^{T}\mathcal{D}^s_m(h_0(\tau))d\tau\le& (4{\rm C}_{3}\eps\nu^{\fr12})^2.
\end{align}
\end{subequations}

We shall establish the following proposition.
\begin{prop}\label{prop:stablity}
    Let $T^*$ be the maximal time such that the bootstrap hypotheses \eqref{hypo} hold on $[0,T^*]$. Then there exist $\eps_0$ and $\nu_0>0$, such that for any $0<\nu\le\nu_0$ and $0<\eps<\eps_0$, the same estimates in \eqref{hypo} hold with the occurrences of 4 on the right-hand side replaced by 2.
\end{prop}
In the following two subsections, we complete the proof of Proposition \ref{prop:stablity}. Since the proofs for the spatial dimensions $n=1,2$ are analogous and simpler, we restrict all subsequent proofs to the case $n=3$. 
\subsection{Estimates of the density}
In this subsection, we improve the bound in \eqref{hypo1}. Using the representation of the density \eqref{rep-rhok}, we estimate the terms individually.
\subsubsection{Initial contributions}\label{subsec:ini}
Recalling the definition of $S_k(t)$ in \eqref{S3},
taking the change of variable 
$t'=|k|t$, and then  using the Sobolev trace theorem, we find that
\begin{align*}
    \int_0^{T^*}\la C_s\nu^{\fr13}|k|^{\fr23}t\ra^\ell\big|S_k(t)(\widehat{h_{\rm in}})_k(kt)\big|^2dt\les& \fr{1}{|k|}\int_0^\infty \Big|(\widehat{h_{\rm in}})_k\Big(\fr{k}{|k|}t'\Big)\Big|^2 dt'\\
    \le&\fr{1}{|k|}\sup_{|\om|=1}\int_{-\infty}^{\infty}|(\widehat{h_{\rm in}})_k(\om t')|^2dt'\les \fr{1}{|k|}\|(\widehat{h_{\rm in}})_k(\eta)\|^2_{H^{1+}_\eta}.
\end{align*}
It follows that
\begin{align}
    \sum_{k\in\Z^3_*}\int_0^{T^*}{\rm B}_{k,s}^2(t)\big|S_k(t)(\widehat{h_{\rm in}})_k(kt)\big|^2dt\les\big\|\la v\ra^{1+}\ln(e+|\nb_x|)(h_{\rm in})_{\ne}\big\|^2_{L^2_{x,v}}.
\end{align}
Similarly, thanks to \eqref{semigroup-est} with $\tau=0$, we have
\begin{align*}
    &\int_0^{T^*} \la C_s\nu^{\fr13}|k|^{\fr23}t\ra^\ell\big|S_k(t;\xi_*(t,k))(\widehat{h_{\mathrm{in}}})_k(\xi_*(t,k))\big|^2dt\\
    \les&\fr{1}{|k|}\int_0^\infty \Big|(\widehat{h_{\mathrm{in}}})_k\Big(\xi_1\fr{k}{|k|}+\fr{k}{|k|}t'\Big)\Big|^2dt'
    \les\fr{1}{|k|}\|(\widehat{f_{\rm in}})_k(\eta)\|^2_{H^{1+}_\eta},
\end{align*}
and hence
\begin{align}
    \nn&\left\|\la C_s\nu^{\fr13}|k|^{\fr23}t\ra^\ell\int_{\R}\bar{\hat{K}}_k(\xi_1)S_k(t;\xi_*(t,k))(\widehat{h_{\mathrm{in}}})_k(\xi_*(t,k))d\xi_1\right\|_{L^2_t}\\
    \les\nn&\int_{\R}|\hat{K}_k(\xi_1)|\Big\|(\nu^{\fr13}|k|^{\fr23}t)^{\fr12}S_k(t;\xi_*(t,k))(\widehat{h_{\mathrm{in}}})_k(\xi_*(t,k))\Big\|_{L^2_t}d\xi_1\\
    \les\nn&\fr{1}{|k|^{\fr12}}\|(\widehat{h_{\rm in}})_k(\eta)\|_{H^{1+}_\eta}\|\hat{K}_k\|_{L^1_{\xi_1}}.
\end{align}
Combining this with \eqref{est:Wk}, we are led to
\begin{align}
    \nn&\left\|{\rm B}_{k,s}(t)\int_{\R}\bar{\hat{K}}_k(\xi_1)S_k(t;\xi_*(t,k))(\widehat{h_{\mathrm{in}}})_k(\xi_*(t,k))d\xi_1\right\|_{L^2_tL^2_k}\\
    \les& \big\|\la v\ra^{1+}\ln(e+|\nb_x|)(h_{\rm in})_{\ne}\big\|_{L^2_{x,v}}.
\end{align}

\subsubsection{Contributions from the linear term}\label{subsec:linear}
Recalling \eqref{S3}, we find that
\begin{align}\label{absorb}
\la C_s\nu^{\fr13}|k|^{\fr23}t\ra^{\fr{\ell}{2}}S_k(t-\tau)\les \la C_s\nu^{\fr13}|k|^{\fr23}\tau\ra^{\fr{\ell}{2}}S^{\fr12}_k(t-\tau),
\end{align}
and
\begin{align}\label{bd-semi1}
    \int_0^\infty S_k(t)(|k|t)^2dt\les \nu^{-1}.
\end{align}
Then, by Schur's test and  the Sobolev trace theorem, we have 
\begin{align}\label{es:rho-L}
    \nn&\nu^2\int_0^{T^*}\Big|{\rm B}_{k,s}(t)\int_0^t S_k(t-\tau)k(t-\tau)\cdot\nb_\xi\hat{h}_k(\tau,k(t-\tau))d\tau\Big|^2dt\\
   \le\nn& \nu^2|k|\int_0^{T^*}\int_0^tS_k(t-\tau)|k(t-\tau)|^2d\tau\int_0^t\big|{\rm A}_{k,s}(\tau)\nb_\xi\hat{h}_k(\tau,k(t-\tau))\big|^2d\tau dt\\
   \le\nn&\nu^2|k|\sup_t\int_0^tS_k(\tau)(|k|\tau)^2 d\tau \int_0^{T^*}\int_{\tau}^{T^*} \big|{\rm A}_{k,s}(\tau)\nb_\xi \hat{h}_k(\tau,k(t-\tau))\big|^2dtd\tau\\
   \les&\nu \int_0^{T^*}\sup_{|\om|=1}\int_{\R}\big|{\rm A}_{k,s}(\tau)\nb_\xi \hat{h}_k(\tau,\om t)\big|^2dtd\tau\les \nu\|{\rm A}_{k,s}\nb_\xi\hat{h}_k\|^2_{L^2_tH^{1+}_\eta}.
\end{align}
Thus,
\begin{align}\label{es:rho-L1}
    \nn&\nu\bigg\|{\rm B}_{k,s}(t)\int_0^t S_k(t-\tau)k(t-\tau)\cdot\nb_\xi\hat{h}_k(\tau,k(t-\tau))d\tau\bigg\|_{L^2_tL^2_k}\\
    \les&\nu^{\fr12}\|\la v\ra^3{\bf A}_s h_{\ne}\|_{L^2_tL^2_{x,v}}.
\end{align}

On the other hand, by using Lemma \ref{lem-LK}, we arrive at
\begin{align}\label{es:rho-L2}
    \nn&\nu\bigg\|{\rm B}_{k,s}(t)\int_0^t\int_{\R}\bar{\hat{K}}_k(\xi_1) S_k(t-\tau;\xi_*(t,k))\\
    \nn&\qquad\qquad\qquad\quad\times\xi_*(t-\tau,k)\cdot\nb_\xi\hat{h}_k(\tau,\xi_*(t-\tau,k))d\xi_1d\tau\bigg\|_{L^2_tL^2_k}\\
    \les&\nu^{\fr12}\|\la v\ra^3{\bf A}_s h_{\ne}\|_{L^2_tL^2_{x,v}}.
\end{align}

\subsubsection{Contributions from the commutator}\label{subsec:com}
By \eqref{exp-D_k}, we write
\begin{align}\label{com:DL}
    [{\bf D}_k,L][g_k]=[{\bf T}_k,\Dl_v][g_k]+\big({\bf T}_k[\nb_v\cdot(v g_k)]-\nb_v\cdot(v{\bf T}_k[g_k])\big)={\bf com}_{k,1}+{\bf com}_{k,2}.
\end{align} 
Recalling from Section \ref{sec-refor} that ${\bf D}_k^{-1}[h_k]=g_k$ with $g_k$ solving \eqref{eq-fk}, we infer from \eqref{bd:com2} and \eqref{onD} that
\begin{align*}
    \|\la v\ra^{2}{\bf com}_{k,1}(\tau)\|_{L^2_v}\les \frac{1}{|k|^2}\|\la v\ra^2{\bf D}_k^{-1}[h_k](\tau)\|_{H^1_v}\les \frac{1}{|k|^2}\|\la v\ra^2h_k(\tau)\|_{H^1_v}.
\end{align*}
Thus, by using \eqref{absorb} and the $L^1$ bound of $S_k(\cdot)$:
\begin{align}\label{bd:Sk1}
    \int_0^tS_k(t-\tau)d\tau\les \nu^{-\fr13}|k|^{-\fr23},
\end{align}
similar to \eqref{es:rho-L}, we obtain
\begin{align}\label{est:com1}
    \nn&\nu^2\int_0^{T^*}\Big|{\rm B}_{k,s}(t)\int_0^tS_k(t-\tau)\mathcal{F}_v[{\bf com}_{k,1}](\tau,k(t-\tau))d\tau\Big|^2dt\\
    \les\nn&\fr{\nu^2|k|}{\nu^{\fr13}|k|^{\fr23}}\int_0^{T^*}\int_\tau^{T^*}\big|{\rm A}_{k,s}(\tau)\mathcal{F}_v[{\bf com}_{k,1}](\tau,k(t-\tau))\big|^2dt d\tau \\
\les\nn&\fr{\nu^2}{\nu^{\fr13}|k|^{\fr23}}\int_0^{T^*}\|\la v\ra^{2}{\rm A}_{k,s}(\tau){\bf com}_{k,1}(\tau)\|^2_{L^2_v} d\tau\\
\les&\nu^{\fr53}\int_0^{T^*}\|\la v\ra^{2}{\rm A}_{k,s}(\tau)h_k(\tau)\|^2_{H^1_v} d\tau.
\end{align}
Consequently, 
\begin{align}\label{est:com1-1}
    \nn&\nu\Big\|{\rm B}_{k,s}(t)\int_0^tS_k(t-\tau)\mathcal{F}_v[{\bf com}_{k,1}](\tau,k(t-\tau))d\tau\Big\|_{L^2_tL^2_k}\\
\les&\nu^{\fr13}\Big(\nu^{\fr12}\big\|\nb_v{\bf A}_s(\la v\ra^{2}h_{\ne})\big\|_{L^2_tL^2_{x,v}} \Big)+\nu^{\fr23}\Big(\nu^{\fr16}\big\|{\bf A}_s(\la v\ra^{2}h_{\ne})\big\|_{L^2_tL^2_{x,v}}\Big).
\end{align}

On the other hand, thanks to \eqref{semigroup-est}, $S_k(t-\tau;\xi_*(t,k))$ can serve as $S_k(t-\tau)$. Moreover, the Sobolev trace theorem also applies to $\mathcal{F}_v[{\bf com}_{k,1}](\tau,\xi_*(t-\tau,k))$:
\begin{align}\label{trace}
    \nn&\int_{\R}|\mathcal{F}_v[{\bf com}_{k,1}](\tau,\xi_*(t-\tau,k))|^2dt\nn=\int_{\R}\Big|\mathcal{F}_v[{\bf com}_{k,1}](\tau,\xi_1\fr{k}{|k|}+k(t-\tau))\Big|^2dt\\
    =&\fr{1}{|k|}\int_{\R}\Big|\mathcal{F}_v[{\bf com}_{k,1}](\tau,\xi_1\fr{k}{|k|}+\fr{k}{|k|}t')\Big|^2dt'\les \fr{1}{|k|}\|\la v\ra^2{\bf com}_{k,1}(\tau)\|_{L^2_v}^2.
\end{align}
Therefore, similar to \eqref{est:com1} and \eqref{est:com1-1}, using Minkowski's inequality and \eqref{est:Wk}, we have
\begin{align}\label{est:com1-2}
    \nn&\nu\bigg\|{\rm B}_{k,s}(t)\int_0^t\int_{\R}\bar{\hat{K}}_k(\xi_1)S_k(t-\tau;\xi_*(t,k))\mathcal{F}_v[{\bf com}_{k,1}](\tau,\xi_*(t-\tau,k))d\xi_1d\tau\bigg\|_{L^2_tL^2_k}\\
    \les\nn&\nu\bigg\|\int_{\R}|\hat{K}_k(\xi_1)|\Big\|{\rm B}_{k,s}(t)\int_0^tS_k(t-\tau;\xi_*(t,k))\mathcal{F}_v[{\bf com}_{k,1}](\tau,\xi_*(t-\tau,k))d\tau\Big\|_{L^2_t}d\xi_1\bigg\|_{L^2_k}\\
    \les\nn&\nu\Big\|{\rm B}_{k,s}(t)\int_0^tS_k(t-\tau;\xi_*(t,k))\mathcal{F}_v[{\bf com}_{k,1}](\tau,\xi_*(t-\tau,k))d\tau\Big\|_{L^2_tL^2_k}\\
    \les&\nu^{\fr13}\Big(\nu^{\fr12}\big\|\nb_v{\bf A}_s(\la v\ra^{2}h_{\ne})\big\|_{L^2_tL^2_{x,v}} \Big)+\nu^{\fr23}\Big(\nu^{\fr16}\big\|{\bf A}_s(\la v\ra^{2}h_{\ne})\big\|_{L^2_tL^2_{x,v}}\Big).
\end{align}

For ${\bf com}_{k,2}$, by \eqref{bd:com3} and \eqref{onD}, we have
\begin{align*}
    \|\la v\ra^{2}{\bf com}_{k,2}(\tau)\|_{L^2_v}\les \frac{1}{|k|^2}\|\la v\ra^2{\bf D}_k^{-1}[h_k](\tau)\|_{L^2_v}\les \frac{1}{|k|^2}\|\la v\ra^2h_k(\tau)\|_{L^2_v}.
\end{align*}
Then similar to \eqref{est:com1-1} and \eqref{est:com1-2}, we arrive at
\begin{align}\label{est:com2-1}
    \nn&\nu\left\| {\rm B}_{k,s}(t)\int_0^tS_k(t-\tau)\mathcal{F}_v[{\bf com}_{k,2}](\tau,k(t-\tau))d\tau\right\|_{L^2_t}\\
    \les& \nu^{\fr23}\left(\nu^{\fr16}\big\|{\bf A}_s(\la v\ra^{2}h_{\ne})\big\|_{L^2_tL^2_{x,v}}\right),
\end{align}
and
\begin{align}\label{est:com2-2}
    \nn&\nu\Big\|{\rm B}_{k,s}(t)\int_0^t\int_{\R}\bar{\hat{K}}_k(\xi_1)S_k(t-\tau;\xi_*(t,k))\\
    &\qquad\qquad\quad\times\mathcal{F}_v[{\bf com}_{k,2}](\tau,\xi_*(t-\tau,k))d\xi_1d\tau\Big\|_{L^2_tL^2_k}
    \les\nu^{\fr23}\left(\nu^{\fr16}\big\|{\bf A}_s(\la v\ra^{2}h_{\ne})\big\|_{L^2_tL^2_{x,v}}\right).
\end{align}

\subsubsection{Contributions from the nonlinearity}\label{subsec:nol}
Recalling \eqref{nonlinearity},  we first split ${\bf D}_k\big[N[g_k]\big]$ into two parts:
\begin{align*}
{\bf D}_k\big[N[g_k]\big]=\sum_{\substack{l\in\Z^3_*}}\rho_l\frac{il}{|l|^2}\cdot \nb_v{\bf D}_k[{g}_{k-l}]+\sum_{\substack{l\in\Z^3_*}}\rho_l\frac{il}{|l|^2}\cdot [{\bf D}_k,\nb_v]{g}_{k-l}=N_{a,k}+N_{b,k}.
\end{align*}

\begin{align*}
    \mathcal{N}_a:=&\sum_{k\in\Z^3_*}\int_0^{T^*}\Big|{\rm B}_{k,s}(t)\int_0^tS_k(t-\tau)\mathcal{F}_v\big[N_{a,k}\big](\tau,k(t-\tau))d\tau\Big|^2dt\\
    =&\sum_{k\in\Z^3_*}\int_0^{T^*}\bigg|\Big(\sum_{\substack{l\in\Z^3_*, |l|\ge |k-l|}}+\sum_{\substack{l\in\Z^3_*, |l|< |k-l|}}\Big){\rm B}_{k,s}(t)\int_0^tS_k(t-\tau)\rho_l(\tau)\fr{l
    \cdot k(t-\tau)}{|l|^2}\\
    &\times \mathcal{F}_v\big[{\bf D}_k[g_{k-l}]\big](\tau,k(t-\tau))d\tau\bigg|^2dt\\
    \les&\mathcal{N}_a^{\rm HL}+\mathcal{N}_a^{\rm LH},
\end{align*}
and $\mathcal{N}_b$ can be further split into two parts analogously.

For $\mathcal{N}_a^{\rm HL}$, using the restriction $|l|\ge |k-l|$ with $|l|\ge1$ and \eqref{S2}, we have
\begin{align}
\fr{1}{|l|}\les \fr{1}{\la k-l\ra}\les |l|^{\fr12}\fr{\ln(e+|k-l|)}{\ln(e+|k-l|)\la k-l\ra^{\fr32}},
\end{align}
and
\begin{align}
    {\rm B}_{k,s}(t)S_k(t-\tau){\bf 1}_{|l|\ge|k-l|}\les_\ell |k|^{\fr12}{\rm A}_{l,s}(\tau)S_k^{\fr12}(t-\tau).
\end{align}
Then combining the above two inequalities with \eqref{bd-semi1}, \eqref{offD} and \eqref{onD}, we are led to
\begin{align}\label{est:NaHL}
    \mathcal{N}_a^{\rm HL}\les\nn&\sum_{k\in\Z^3_*}|k|\int_0^{T^*}\bigg[\sum_{\substack{l\in\Z^3_*, |l|\ge |k-l|}}\int_0^tS^{\fr12}_k(t-\tau)|k|(t-\tau)|{\rm B}_{l,s}(\tau)\rho_l(\tau)|\\
    \nn&\times \fr{\ln(e+|k-l|)}{\ln(e+|k-l|)\la k-l\ra^{\fr32}}|\mathcal{F}_v\big[{\bf D}_k[g_{k-l}]\big](\tau,k(t-\tau))|d\tau\bigg]^2dt\\
    \les\nn&\sum_{k\in\Z^3_*}|k|\int_0^{T^*}\sum_{l\in\Z^3_*}\fr{1}{\big(\ln(e+|k-l|)\big)^2\la k-l\ra^3}\int_0^tS_k(t-\tau)|k|^2(t-\tau)^2d\tau \\
    \nn&\times \sum_{l\in\Z^3_*}\int_0^t|{\rm B}_{l,s}(\tau)\rho_l(\tau)|^2\big|{\rm A}_{k-l,0}(\tau)\mathcal{F}_v\big[{\bf D}_k[g_{k-l}]\big](\tau,k(t-\tau))\big|^2d\tau dt\\
    \les\nn&\nu^{-1}\sum_{k\in\Z^3_*}|k|\sum_{l\in\Z^n_*}\int_0^{T^*}|{\rm B}_{l,s}(\tau)\rho_l(\tau)|^2\int_\tau^{T^*}\big|{\rm A}_{k-l,0}(\tau)\mathcal{F}_v\big[{\bf D}_k[g_{k-l}]\big](\tau,k(t-\tau))\big|^2 dtd\tau\\
    \les\nn&\nu^{-1}\sum_{k,l\in\Z^3_*}\int_0^{T^*}|{\rm B}_{l,s}(\tau)\rho_l(\tau)|^2\big\|\la v\ra^2{\rm A}_{k-l,0}(\tau){\bf D}_k[g_{k-l}](\tau)\big\|_{L^2_v}^2 d\tau\\
    \les\nn&\nu^{-1}\sum_{k,l\in\Z^3_*}\int_0^{T^*}|{\rm B}_{l,s}(\tau)\rho_l(\tau)|^2\big\|\la v\ra^2{\rm A}_{k-l,0}(\tau)h_{k-l}(\tau)\big\|_{L^2_v}^2 d\tau\\
    \les&\nu^{-1}\|{\bf B}_{s}\rho\|_{L^2_{t,x}}^2\|\la v\ra^2{\bf A}_{0}h\|_{L^\infty_tL^2_{x,v}}^2.
\end{align}

For $\mathcal{N}_a^{\rm LH}$, using the restriction $|l|< |k-l|$ with $|l|\ge1$ and \eqref{S2}, we find that
\begin{align}
    {\rm B}_{k,s}(t)S_k(t-\tau){\bf 1}_{|l|<|k-l|}\les |k|^{\fr12}{\rm A}_{k-l,s}(\tau)S_k^{\fr12}(t-\tau).
\end{align}
Then similar to \eqref{est:NaHL}, we have
\begin{align}\label{est:NaLH}
    \mathcal{N}_a^{\rm LH}\les\nn&\sum_{k\in\Z^3_*}|k|\int_0^{T^*}\bigg[\sum_{\substack{l\in\Z^3_*, |l|< |k-l|}}\int_0^tS^{\fr12}_k(t-\tau)|k|(t-\tau)\fr{\ln(e+|l|)|l|^{\fr12}|\rho_l(\tau)|}{\ln(e+|l|)|l|^{\fr32}}\\
    \nn&\times {\rm A}_{k-l,\ell}(\tau)|\mathcal{F}_v\big[{\bf D}_k[g_{k-l}]\big](\tau,k(t-\tau))|d\tau\bigg]^2dt\\
    \les\nn&\sum_{k\in\Z^3_*}|k|\int_0^{T^*}\sum_{l\in \Z^3_{*}}\fr{1}{\big(\ln(e+|l|)\big)^2|l|^{3}}\int_0^tS_k(t-\tau)|k|^2(t-\tau)^2d\tau\\
    \nn&\times\sum_{l\in \Z^3_{*}}\int_0^t|{\rm B}_{l,0}(\tau)\rho_l(\tau)|^2\big|{\rm A}_{k-l,s}(\tau)\mathcal{F}_v\big[{\bf D}_k[g_{k-l}]\big](\tau,k(t-\tau))\big|^2d\tau dt\\
    \les\nn&\nu^{-1}\sum_{k,l\in\Z^3_*}\int_0^{T^*}|{\rm B}_{l,0}(\tau)\rho_l(\tau)|^2\big\|\la v\ra^2{\rm A}_{k-l,s}(\tau)h_{k-l}(\tau)\big\|^2_{L^2_v}d\tau\\
    \les&\nu^{-1}\|{\bf B}_{0}\rho\|_{L^2_{t,x}}^2\big\|\la v\ra^2{\bf A}_{s}h_{\ne}(\tau)\big\|^2_{L^\infty_tL^2_{x,v}}.
\end{align}

As for $\mathcal{N}_b$, by \eqref{bd:com1}, \eqref{onD} and \eqref{bd:Sk1},
similar to \eqref{est:NaHL} and \eqref{est:NaLH}, we find that
    \begin{align}\label{est:NbHL}
    \mathcal{N}_b^{\rm HL}
    \les\nn&\nu^{-\fr13}\sum_{k,l\in\Z^3_*}|k|^{-\fr23}\int_0^{T^*}|{\rm B}_{l,s}(\tau)\rho_l(\tau)|^2\big\|\la v\ra^2{\rm A}_{k-l,0}(\tau)[{\bf D}_k,\nb_v]{g}_{k-l}(\tau)\big\|_{L^2_v}^2 d\tau\\
    \les&\nu^{-\fr13}\|{\bf B}_{s}\rho\|_{L^2_{t,x}}^2\|\la v\ra^2{\bf A}_{0}h\|_{L^\infty_tL^2_{x,v}}^2,
\end{align}
and
\begin{align}\label{est:NbLH}
    \mathcal{N}_b^{\rm LH}
    \les\nn&\nu^{-\fr13}\sum_{k,l\in\Z^3_*}|k|^{-\fr23}\int_0^{T^*}|{\rm B}_{l,0}(\tau)\rho_l(\tau)|^2\big\|\la v\ra^2{\rm A}_{k-l,s}(\tau)[{\bf D}_k,\nb_v]{g}_{k-l}(\tau)\big\|_{L^2_v}^2 d\tau\\
    \les&\nu^{-\fr13}\|{\bf B}_{0}\rho\|_{L^2_{t,x}}^2\|\la v\ra^2{\bf A}_{s}h_{\ne}\|_{L^\infty_tL^2_{x,v}}^2.
\end{align}

On the other hand, thanks to Corollary \ref{coro-NLK}, we have
\begin{align}\label{es:KNa}
\nn&\bigg\|{\rm B}_{k,\ell}(t)\int_0^t\int_{\R}\bar{\hat{K}}_k(\xi_1)S_k(t-\tau;\xi_*(t,k))\mathcal{F}_v\big[N_{a,k}\big](\tau,\xi_*(t-\tau,k))d\xi_1d\tau\bigg\|_{L^2_tL^2_k}\\
\les&\nu^{-\fr12}\Big(\|{\bf B}_{s}\rho\|_{L^2_{t,x}}\|\la v\ra^2{\bf A}_{0}h\|_{L^\infty_tL^2_{x,v}}+\|{\bf B}_{0}\rho\|_{L^2_{t,x}}\big\|\la v\ra^2{\bf A}_{s}h_{\ne}\big\|_{L^\infty_tL^2_{x,v}}\Big).
\end{align}
Finally, since \eqref{semigroup-est} ensures that $S_k(t-\tau;\xi_*(t,k))$ can serve as $S_k(t-\tau)$, by Minkowski's inequality and \eqref{est:Wk}, similar to \eqref{est:NbHL} and \eqref{est:NbLH}, one deduces that
\begin{align}\label{es:KNb}
\nn&\bigg\|{\rm B}_{k,s}(t)\int_0^t\int_{\R}\hat{K}_k(\xi_1)S_k(t-\tau;\xi_*(t,k))\mathcal{F}_v\big[N_{b,k}\big](\tau,\xi_*(t-\tau,k))d\xi_1d\tau\bigg\|_{L^2_tL^2_k}\\
\le\nn&\bigg\|\int_{\R}|\hat{K}_k(\xi_1)|\Big\|{\rm B}_{k,s}(t)\int_0^tS_k(t-\tau;\xi_*(t,k))\mathcal{F}_v\big[N_{b,k}\big](\tau,\xi_*(t-\tau,k))d\tau\Big\|_{L^2_t} d\xi_1\bigg\|_{L^2_k}\\
\les&\nu^{-\fr16}\left(\|{\bf B}_{s}\rho\|_{L^2_{t,x}}\|\la v\ra^2{\bf A}_{0}h\|_{L^\infty_tL^2_{x,v}}+\|{\bf B}_{0}\rho\|_{L^2_{t,x}}\|\la v\ra^2{\bf A}_{s}h_{\ne}\|_{L^\infty_tL^2_{x,v}}\right).
\end{align}

\noindent\underline{\bf Improvement of \eqref{hypo1}}: Collecting all the estimates from Section \ref{subsec:ini} to Section \ref{subsec:nol}, and noting that 
\begin{align*}
    \nu^{\fr13}\|\la v\ra^m{\bf A}_sh_{\ne}\|_{L^2_tL^2_{x,v}}^2+\nu^{\fr12}\|\nb_v(\la v\ra^m{\bf A}_sh_{\ne})\|_{L^2_tL^2_{x,v}}^2\les\int_0^{T^*}\mathcal{D}^s_{m}(h_{\ne}(t)) dt,\\
    \quad \|\la v\ra^m{\bf A}_sh\|_{L^\infty_tL^2_{x,v}}^2\les\sup_{0\le t\le T^*}\mathcal{E}^s_m(h(t)),
\end{align*}
we find that there exists a positive constant $C_*$, depending on $s$ and $m$ but independent of $\nu$, such that for all $s\ge0$ and $m\in\N$ with $m\ge3$, there holds
\begin{align}\label{bd:Bsrho}
    \|{\bf B}_s\rho\|_{L^2_tL^2_x}\le\nn& C_*\|\la v\ra^m\ln(e+|\nb_x|)(g_{\rm in})_{\ne}\|_{L^2_{x,v}}+C_*\nu^{\fr13}\Big(\int_0^{T^*}\mathcal{D}^s_{m}(h_{\ne}(t)) dt\Big)^{\fr12}\\
    &+C_*\nu^{-\fr12}\|{\bf B}_s\rho\|_{L^2_tL^2_x}\Big( \sup_{0\le t\le T^*}\mathcal{E}^s_m(h_{\ne}(t))+\sup_{0\le t\le T^*}\mathcal{E}^s_m(h_{0}(t))\Big)^{\fr12},
\end{align}
where we have used the fact
\begin{align}\label{bd:initial}
    \|\la v\ra^m\ln(e+|\nb_x|)(h_{\rm in})_{\ne}\|_{L^2_{x,v}}\les \|\la v\ra^m\ln(e+|\nb_x|)(g_{\rm in})_{\ne}\|_{L^2_{x,v}}
\end{align}
due to \eqref{offD}. By \eqref{initial-sta} and the bootstrap hypotheses \eqref{hypo}, we infer from \eqref{bd:Bsrho} that
\begin{align*}
    \|{\bf B}_s\rho\|_{L^2_tL^2_x}\le C_*\eps\nu^{\fr12}+8\nu^{\fr13}{\rm C_2}C_*(\eps\nu^{\fr12})+16\eps{\rm C_1(C_2+C_3)}C_*(\eps\nu^{\fr12}),
\end{align*}
which is sufficient to improve \eqref{hypo1}, provide we take
\begin{align}\label{det:C1}
    {\rm C}_1:=2C_*,\quad \nu\le \nu_1:=\fr{1}{(8{\rm C}_2)^3},\quad \eps\le\eps_1:=\fr{1}{16{\rm C_1(C_2+C_3)}},
\end{align}
where ${\rm C_2}$ and ${\rm C_3}$ will be determined in \eqref{det:C2} and \eqref{det:C3}, respectively.

\subsection{Estimates of the distribution}
This subsection aims to improve the bounds in \eqref{hypo2} and \eqref{hypo3} by employing time-weighted energy estimates and the Fourier multiplier method.
\subsubsection{Multiplier}\label{sec:multi} Inspirited by \cite{DengWuZhang2021}, let us define the multiplier that captures the enhanced dissipation as follows.
Assume that $\varphi\in C^\infty(\R)$ be a real-valued function such that 
\begin{align}\label{prop-varphi}
0\le\varphi\le1,\quad 0\le\varphi'\le\fr14,\quad{\rm and}\quad \varphi'=\fr14\ \ {\rm on}\ \ [-1,1].
\end{align}
Define
\[
\mathcal{M}_1(k,\xi)=1+{\bf 1}_{|k|\ne0}\varphi(\nu^{\fr13}|k|^{-\fr43}k\cdot \xi).
\]
Then 
\begin{subequations}\label{prop-M1}
\begin{gather}
1\le \mathcal{M}_1(k,\xi)\le2,\\
k\cdot\nb_{\xi}\mathcal{M}_1(k,\xi)=\nu^{\fr13}|k|^{\fr23}\varphi'(\nu^{\fr13}|k|^{-\fr43}k\cdot\xi),\\
\label{enhan-dis}\nu|\xi|^2+k\cdot\nb_{\xi}\mathcal{M}_1(k,\xi)\ge\fr14\nu^{\fr13}|k|^{\fr23}.
\end{gather}
\end{subequations}
Now we define the multiplier 
\begin{align}\label{def-M}
\mathcal{M}(k,\xi)=\begin{cases}
{\mathcal{M}_1(k,\xi)},\ \, \quad \quad\quad\  \ \,\quad{\rm if}\  n=1,2;\\
\ln(e+|k|){\mathcal{M}_1(k,\xi)}, \quad{\rm if}\ n=3.
\end{cases}
\end{align}
\subsubsection{Gains of Sobolev regularity}
Now we investigate the evolution of $\mathcal{E}^{s}_{m}(h(t))$. As mentioned above, here we only focus on the case $n=3$. To begin with, let us rewrite \eqref{eq-hatf} as follows
\begin{align}\label{eq-homo-h}
    \pr_t\hat{h}_k(t,\xi)-k\cdot\nb_\xi\hat{h}_k(t,\xi)+\nu|\xi|^2\hat{h}_k(t,\xi)=R_k(t,\xi),
\end{align}
where
\begin{align}\label{def-R_k}
R_k(t,\xi):=\mathcal{F}_v\big[\mathbf{D}_k\big[N[g_k]\big]\big](\xi)+\nu\mathcal{F}_v\big[[\mathbf{D}_k,L][g_k]\big](\xi)-\nu\xi\cdot\nb_\xi\hat{h}_k(t,\xi).
\end{align}
Then 
\begin{align*}
    &\pr_t\big(\mathcal{M}(k,\xi)\pr_\xi^\al\hat{h}_k\big)-k\cdot\nb_\xi\big(\mathcal{M}(k,\xi)\pr_\xi^\al\hat{h}_k\big)+k\cdot\nb_\xi\mathcal{M}(k,\xi)\pr_{\xi}^\al\hat{h}_k+\nu|\xi|^2\mathcal{M}(k,\xi)\pr_\xi^\al\hat{h}_k\\
    =&\nu\mathcal{M}(k,\xi)[|\xi|^2,\pr_\xi^\al]\hat{h}_k+\mathcal{M}(k,\xi)\pr_\xi^\al R_k(t,\xi).
\end{align*}
By \eqref{prop-varphi} and \eqref{prop-M1}, we infer that
\begin{align*}
    &k\cdot\nb_\xi\mathcal{M}(k,\xi)\pr_\xi^\al\hat{h}_k\mathcal{M}(k,\xi)\pr_\xi^\al\bar{\hat{h}}_k+\nu|\xi|^2\big|\mathcal{M}(k,\xi)\pr_\xi^\al\hat{h}_k\big|^2\\
    =&\left(\frac{k\cdot\nb_\xi\mathcal{M}_1(k,\xi)}{\mathcal{M}_1(k,\xi)}+\nu|\xi|^2\right)\big|\mathcal{M}(k,\xi)\pr_\xi^\al\hat{h}_k\big|^2\\
    \ge&\fr18\nu^{\fr13}|k|^{\fr23}\big|\mathcal{M}(k,\xi)\pr_\xi^\al\hat{h}_k\big|^2+\fr12\nu|\xi|^2\big|\mathcal{M}(k,\xi)\pr_\xi^\al\hat{h}_k\big|^2.
\end{align*}
Thus, we have
\begin{align*}
    &\fr{d}{dt}\big\||\xi|^\ell \mathcal{M}(\pr_\xi^\al \hat{h}_k)\big\|_{L^2_\xi}^2+\nu\big\||\xi|^{\ell+1}\mathcal{M}(\pr_\xi^\al \hat{h}_k)\big\|_{L^2_\xi}^2+\fr{1}{4}\nu^{\fr13}\big\||k|^{\fr13}|\xi|^\ell \mathcal{M}(\pr_\xi^\al \hat{h}_k)\big\|_{L^2_\xi}^2\\
    \le&2\nu\frak{Re}\int_\xi |\xi|^\ell\mathcal{M}(k,\xi)[|\xi|^2,\pr_\xi^\al]\hat{h}_k\big(|\xi|^\ell\mathcal{M}(k,\xi)\pr_\xi^\al\bar{\hat{h}}_k\big)d\xi\\
    &+2\frak{Re}\int_\xi |\xi|^\ell\mathcal{M}(k,\xi)\pr_\xi^\al R_k(t,\xi)\big(|\xi|^\ell\mathcal{M}(k,\xi)\pr_\xi^\al\bar{\hat{h}}_k\big)d\xi\\
     &-2\ell\frak{Re}\int_{\xi}k\cdot\xi |\xi|^{2\ell-2}\big|\mathcal{M}(k,\xi)\pr_\xi^\al\hat{h}_k\big|^2d\xi\\
    =:&G_{k,1}^{\ell,\al}(t)+G_{k,2}^{\ell,\al}(t)+G_{k,3}^{\ell,\al}(t),
\end{align*}
and
\begin{align*}
    &\fr{d}{dt}\big\||k|^{\fr\ell3} \mathcal{M}(\pr_\xi^\al \hat{h}_k)\big\|_{L^2_\xi}^2+\nu\big\||\xi||k|^{\fr{\ell}{3}}\mathcal{M}(\pr_\xi^\al \hat{h}_k)\big\|_{L^2_\xi}^2+\fr{1}{4}\nu^{\fr13}\big\||k|^{\fr{\ell+1}{3}} \mathcal{M}(\pr_\xi^\al \hat{h}_k)\big\|_{L^2_\xi}^2\\
    \le&2\nu\frak{Re}\int_\xi |k|^{\fr\ell3}\mathcal{M}(k,\xi)[|\xi|^2,\pr_\xi^\al]\hat{h}_k\big(|k|^{\fr\ell3}\mathcal{M}(k,\xi)\pr_\xi^\al\bar{\hat{h}}_k\big)d\xi\\
    &+2\frak{Re}\int_\xi |k|^{\fr\ell3}\mathcal{M}(k,\xi)\pr_\xi^\al R_k(t,\xi)\big(|k|^{\fr\ell3}\mathcal{M}(k,\xi)\pr_\xi^\al\bar{\hat{h}}_k\big)d\xi\\
    =:&H_{k,1}^{\ell,\al}(t)+H_{k,2}^{\ell,\al}(t).
\end{align*}
It follows from the above two inequalities that
\begin{align}\label{es:en-s-al1}
    \fr{d}{dt}\frak{E}^s_{\al}(h(t))+\frak{D}^{s}_{\al}(h(t))
    =\nn&\sum_{1\le\ell\le s}\sum_{k\in\Z^3}\Big({\ell\kappa \nu}(\kappa\nu t)^{\ell-1}\big\||\nb_v|^\ell\mathcal{M}(v^\al h_{k}(t))\big\|_{L^2_v}^2\\
    \nn&\qquad\qquad\quad+C_s\ell\kappa \nu^{\fr13}(C_s\kappa\nu^{\fr13} t)^{\ell-1}\big\||\nb_x|^{\fr{\ell}{3}}\mathcal{M}(v^\al h_{k}(t))\big\|_{L^2_v}^2\Big)\\
    &+\sum_{0\le\ell\le s}\sum_{k\in\Z^3}\sum_{1\le j\le3}(\kappa \nu t)^\ell G^{\ell,\al}_{k,j}+\sum_{0\le\ell\le s}\sum_{k\in\Z^3}\sum_{1\le j\le2}(C_s\kappa \nu^{\fr13} t)^\ell H^{\ell,\al}_{k,j}.
\end{align}

We first bound the sum involving $G^{\ell,\al}_{k,3}$. Clearly, $G^{\ell,\al}_{k,3}=0$ if $k=0$. For $k\ne0$,  interpolating between the two enhanced dissipation terms in \eqref{def-Dis-al}:
\[
|k|^{\fr23}|\xi|^{\ell-1}=\big(|k|^{\fr13}|\xi|^\ell\big)^{1-\fr{1}{\ell}}\big(|k|^{\fr{\ell+1}{3}}\big)^{\fr{1}{\ell}},
\]
and choosing $C_s$ so large that
\begin{align}\label{Csig}
\fr{2s}{C_s}\le \fr{1}{32},
\end{align}
we find that
\begin{align*}
    \sum_{\substack{0\le\ell\le s\\ k\in\Z^3}}(\kappa\nu t)^\ell G^{\ell,\al}_{k,3}
    \le&2\sum_{\substack{1\le\ell\le s\\ k\in\Z^3_*}}\ell (\kappa\nu t)^\ell\big\||k|^{\fr23}|\xi|^{\ell-1}\mathcal{M}(k,\xi)\pr_\xi^\al\hat{h}_k\big\|_{L^2_\xi}\big\||k|^{\fr13}|\xi|^\ell\mathcal{M}(k,\xi)\pr_\xi^\al{\hat{h}}_k\big\|_{L^2_\xi}\\
    \le&\sum_{\substack{1\le\ell\le s\\ k\in\Z^3_*}}\fr{2\ell}{\sqrt{C_s}} \Big(\nu^\fr{1}{6}(C_s\kappa\nu^{\fr13}t)^{\fr{\ell}{2}}\big\||k|^{\fr{\ell+1}{3}}\mathcal{M}(k,\xi)\pr_\xi^\al\hat{h}_k\big\|_{L^2_\xi}\Big)^{\fr{1}{\ell}}\\
    &\times\Big(\nu^{\fr16}(\kappa\nu t)^{\fr{\ell}{2}}\big\||k|^{\fr13}|\xi|^{\ell}\mathcal{M}(k,\xi)\pr_\xi^\al\hat{h}_k\big\|_{L^2_\xi}\Big)^{\fr{\ell-1}{\ell}+1}\\
    \le&\fr{2s}{\sqrt{C_s}}\sum_{\substack{1\le\ell\le s\\ k\in\Z^3_*}} \Big[\nu^\fr{1}{3}(C_s\kappa\nu^{\fr13}t)^{\ell}\big\||k|^{\fr{\ell+1}{3}}\mathcal{M}(k,\xi)\pr_\xi^\al\hat{h}_k\big\|_{L^2_\xi}^2\\
    &+\nu^{\fr13}(\kappa\nu t)^{\ell}\big\||k|^{\fr13}|\xi|^{\ell}\mathcal{M}(k,\xi)\pr_\xi^\al\hat{h}_k\big\|_{L^2_\xi}^2\Big]
    \le \fr18\frak{D}^s_{\al}(h_{\ne}(t)).
\end{align*}

Next, choosing $\kappa$ so small that
\begin{align}\label{kappa}
    C_ss\kappa\le\fr{1}{32},
\end{align}
one easily deduces that
\begin{align*}
    &\sum_{1\le\ell\le s}\sum_{k\in\Z^3}\Big({\ell\kappa \nu}(\kappa\nu t)^{\ell-1}\big\||\nb_v|^\ell\mathcal{M}(v^\al h_{k}(t))\big\|_{L^2_v}^2\\
    \nn&+C_s\ell\kappa \nu^{\fr13}(C_s\kappa\nu^{\fr13} t)^{\ell-1}\big\||\nb_x|^{\fr{\ell}{3}}\mathcal{M}(v^\al h_{k}(t))\big\|_{L^2_v}^2\Big)\\
    \le&s\kappa\sum_{0\le\ell\le s-1}\sum_{k\in\Z^3} \nu(\kappa\nu t)^{\ell}\big\||\nb_v|^{\ell+1}\mathcal{M}(v^\al h_{k}(t))\big\|_{L^2_v}^2\\
    &+C_ss\kappa\sum_{0\le\ell\le s-1}\sum_{k\in\Z^3} \nu^{\fr13}(C_s\kappa\nu^{\fr13} t)^{\ell}\big\||\nb_x|^{\fr{\ell+1}{3}}\mathcal{M}(v^\al h_{k}(t))\big\|_{L^2_v}^2\Big)
    \le \fr18\frak{D}^s_{\al}(h(t)).
\end{align*}
Substituting the above two estimates into \eqref{es:en-s-al1}, we are led to
\begin{align}\label{es:en-s-al2}
    \nn&\fr{d}{dt}\frak{E}^s_{\al}(h_{\ne}(t))+\fr34\frak{D}^{s}_{\al}(h_{\ne}(t))\\
    \le&\sum_{0\le\ell\le s}\sum_{k\in\Z^3_*}\sum_{1\le j\le2}(\kappa \nu t)^\ell G^{\ell,\al}_{k,j}+\sum_{0\le\ell\le s}\sum_{k\in\Z^3_*}\sum_{1\le j\le2}(C_s\kappa \nu^{\fr13} t)^\ell H^{\ell,\al}_{k,j}.
\end{align}
and
\begin{align}\label{es:en-s-al2-0}
    \nn&\fr{d}{dt}\frak{E}^s_{\al}(h_0(t))+\fr34\frak{D}^{s}_{\al}(h_0(t))\\
    \le&\sum_{0\le\ell\le s}\sum_{1\le j\le2}(\kappa \nu t)^\ell G^{\ell,\al}_{0,j}+\sum_{0\le\ell\le s}\sum_{1\le j\le2}(C_s\kappa \nu^{\fr13} t)^\ell H^{\ell,\al}_{0,j}.
\end{align}
In the following, we bound the right-hand side of \eqref{es:en-s-al2}.

\noindent $\diamond$ \underline{Treatments of $G_{k,1}^{\ell,\al}$ and $H_{k,1}^{\ell,\al}$}. 
Clearly, 
\[
[\pr_\xi^\al,|\xi|^2]\hat{h}_k=2\sum_{i=1}^3\al_i\xi_i\pr_\xi^{\al-e_i}\hat{h}_k+\sum_{i=1}^3\al_i(\al_i-1)\pr_\xi^{\al-2e_i}\hat{h}_k.
\]
The non-zero mode $h_{\ne}$ and the zero mode $h_0$ can be treated in different ways, since the former exhibits an enhanced dissipation effect. More precisely, for the non-zero mode $h_{\ne}$,
\begin{align*}
    \nn&\sum_{0\le\ell\le s}\sum_{k\in\Z^3_*}(\kappa\nu t)^\ell G^{\ell,\al}_{k,1}\\
    \le\nn&4\nu\sum_{i=1}^3\al_i\sum_{0\le\ell\le s}(\kappa\nu t)^\ell\big\||\nb_v|^{\ell}\mathcal{M}(v^{\al-e_i}h_{\ne})\big\|_{L^2_{x,v}}\big\||\nb_v|^{\ell+1}\mathcal{M}(v^\al h_{\ne})\big\|_{L^2_{x,v}}\\
    &+2\nu\sum_{i=1}^3\al_i(\al_i-1)\sum_{0\le\ell\le s}(\kappa\nu t)^\ell\big\||\nb_v|^{\ell}\mathcal{M}(v^{\al-2e_i}h_{\ne})\big\|_{L^2_{x,v}}\big\||\nb_v|^{\ell}\mathcal{M}(v^\al h_{\ne})\big\|_{L^2_{x,v}}\\
    \le\nn&\fr{1}{16}\sum_{0\le\ell\le s}(\kappa\nu t)^\ell\Big(\nu\big\||\nb_v|^{\ell+1}\mathcal{M}(v^\al h_{\ne})\big\|_{L^2_{x,v}}^2\Big)+C_{|\al|}\nu\sum_{i=1}^3\sum_{0\le\ell\le s}(\kappa\nu t)^{\ell}\big\||\nb_v|^\ell\mathcal{M}(v^{\al-e_i}h_{\ne})\big\|_{L^2_{x,v}}^2\\
    &+C_{|\al|}\nu \sum_{i=1}^3\sum_{0\le\ell\le s}(\kappa\nu t)^{\ell}\big\||\nb_v|^\ell\mathcal{M}(v^{\al-2e_i}h_{\ne})\big\|_{L^2_{x,v}}\big\||\nb_v|^\ell\mathcal{M}(v^{\al}h_{\ne})\big\|_{L^2_{x,v}},
\end{align*}
and similarly,
\begin{align*}
    \nn&\sum_{0\le\ell\le s}\sum_{k\in\Z^3_*}(C_s\kappa\nu^{\fr13} t)^\ell H^{\ell,\al}_{k,1}
    \le\fr{1}{64}\sum_{0\le\ell\le s}(C_s\kappa\nu^{\fr13} t)^\ell\Big(\nu\big\||\nb_v||\nb_x|^{\fr{\ell}{3}}\mathcal{M}(v^\al h_{\ne})\big\|_{L^2_{x,v}}^2\Big)\\
    \nn&+C_{|\al|}\nu\sum_{i=1}^3\sum_{0\le\ell \le s}(C_s\kappa\nu^{\fr13} t)^{\ell}\big\||\nb_x|^{\fr{\ell}{3}}\mathcal{M}(v^{\al-e_i}h_{\ne})\big\|_{L^2_{x,v}}^2\\
    &+C_{|\al|}\nu \sum_{i=1}^3\sum_{0\le\ell \le s}(C_s\kappa\nu^{\fr13} t)^{\ell}\big\||\nb_x|^{\fr{\ell}{3}}\mathcal{M}(v^{\al-2e_i}h_{\ne})\big\|_{L^2_{x,v}}\big\||\nb_x|^{\fr{\ell}{3}}\mathcal{M}(v^{\al}h_{\ne})\big\|_{L^2_{x,v}}.
\end{align*}
It follows from the above two estimates that
\begin{align}\label{es:com-v1}
    \nn&\sum_{0\le\ell\le s}\sum_{k\in\Z^3_*}(\kappa\nu t)^\ell G^{\ell,\al}_{k,1}+\sum_{0\le\ell\le s}\sum_{k\in\Z^3_*}(C_s\kappa\nu^{\fr13} t)^\ell H^{\ell,\al}_{k,1}\\
    \le\nn&\fr{1}{8}\frak{D}_\al^s(h_{\ne}(t))+C_{|\al|}\nu^{\fr23}\sum_{i=1}^3\frak{D}_{\al-e_i}^s(h_{\ne}(t))\\
    &+C_{|\al|}\nu^{\fr23}\sum_{i=1}^3\big(\frak{D}_{\al-2e_i}^s(h_{\ne}(t))\big)^{\fr12}\big(\frak{D}_{\al}^s(h_{\ne}(t))\big)^{\fr12}.
\end{align}
However, for the zero mode $h_0$, enhanced dissipation is unavailable. Now \eqref{es:com-v1} is replaced by  
\begin{align}\label{es:com-v2}
    \nn&\sum_{0\le\ell\le s}(\kappa\nu t)^\ell G^{\ell,\al}_{0,1}+\sum_{0\le\ell\le s}(C_s\kappa\nu^{\fr13} t)^\ell H^{\ell,\al}_{0,1}\\
    \le&\fr{1}{8}\frak{D}_\al^s(h_{0}(t))+C_{|\al|}\nu\sum_{i=1}^3\frak{E}_{\al-e_i}^s(h_{0}(t))+C_{|\al|}\nu\sum_{i=1}^3\big(\frak{E}_{\al-2e_i}^s(h_{0}(t))\big)^{\fr12}\big(\frak{E}_{\al}^s(h_{0}(t))\big)^{\fr12}.
\end{align}

\noindent $\diamond$ \underline{Treatments of $G_{k,2}^{\ell,\al}$ and $H_{k,2}^{\ell,\al}$}. 
According to \eqref{def-R_k} , we  write
\begin{align*}
G_{k,2}^{\ell,\al}:=G_{k,2;1)}^{\ell,\al}+G_{k,2;2)}^{\ell,\al}+G_{k,2;3)}^{\ell,\al},
\end{align*}
and
\[
H_{k,2}^{\ell,\al}:=H_{k,2;1)}^{\ell,\al}+H_{k,2;2)}^{\ell,\al}+H_{k,2;3)}^{\ell,\al}.
\]
We first deal with the nonlinear estimates stemming from $G_{k,2;1)}^{\ell,\al}$ and $H_{k,2;1)}^{\ell,\al}$. Noting that the wave operator ${\bf D}_k$ reduces to the identity operator when $k=0$, we split the estimates into the cases $k\ne0$ and $k=0$.

{\it Case 1: $k\ne0$.} By \eqref{offD}, we have
\begin{align}\label{est:Dnbg}
    \big\||\nb_v|^\ell(v^{\al}{\bf D}_k[\nb_vg_{k-l}])\big\|_{L^2_v}\les\nn& \big\||\nb_v|^\ell(v^{\al}\nb_vg_{k-l})\big\|_{L^2_v}+\sum_{|\beta|\le\ell}\big\|\pr_v^{\beta}(\la v \ra^{2}\nb_vg_{k-l})\big\|_{L^2_v}\\
    \les&\big\||\nb_v|^{\ell+1}\big(\la v \ra^{\max\{2,|\al|\}} g_{k-l}\big)\big\|_{L^2_v}+\sum_{|\beta|\le\ell}\big\|\pr_v^\beta(\la v \ra^{\max\{2,|\al|-1\}} g_{k-l})\big\|_{L^2_v}.
\end{align}
This, together with \eqref{log-produ} with $s=0$, implies that
\begin{align}
    \nn&\sum_{0\le\ell\le s}\sum_{k\in\Z^3_*}(\kappa \nu t)^\ell G_{k,2;1)}^{\ell,\al}\\
    \les\nn&\sum_{0\le\ell\le s}(\kappa \nu t)^\ell\sum_{k\in\Z^3_*} \sum_{l\in\Z^3_*}{\rm A}_{k,0}\fr{|\rho_l|}{|l|}\big\||\nb_v|^{\ell}\big(v^\al \mathbf{D}_k[\nb_vg_{k-l}]\big)\big\|_{L^2_v}\big\||\nb_v|^{\ell}\mathcal{M}(v^\al {h}_k)\big\|_{L^2_v}\\
    \les\nn&\sum_{0\le\ell\le s}(\kappa \nu t)^\ell\sum_{k\in\Z^3_*} \sum_{l\in\Z^3_*}{\rm A}_{k,0}\fr{|\rho_l|}{|l|}\Big(\big\||\nb_v|^{\ell+1}\big(\la v \ra^{\max\{2,|\al|\}} g_{k-l}\big)\big\|_{L^2_v}\\
    \nn&+\sum_{|\beta|\le\ell}\big\|\pr_v^\beta(\la v \ra^{\max\{2,|\al|-1\}} g_{k-l})\big\|_{L^2_v}\Big)\big\||\nb_v|^{\ell}\mathcal{M}(v^\al {h}_k)\big\|_{L^2_v}\\
    \les\nn&\sum_{0\le\ell\le s}(\kappa \nu t)^\ell\|{\bf B}_0\rho\|_{L^2_x}\Big(\big\|{\bf A}_0|\nb_v|^{\ell+1}\big(\la v \ra^{\max\{2,|\al|\}} g\big)\big\|_{L^2_{x,v}}\\
    \nn&+\sum_{|\beta|\le\ell}\big\|{\bf A}_0\pr_v^\beta(\la v \ra^{\max\{2,|\al|-1\}} g)\big\|_{L^2_{x,v}}\Big)\big\||\nb_v|^{\ell}\mathcal{M}(v^\al {h}_{\ne})\big\|_{L^2_{x,v}}.
\end{align}
Furthermore, in view of \eqref{onD}, using the facts $\nu t\le1$ and $g_0=h_0$, it holds that
\begin{align}
    \nn&(\kappa\nu t)^{\fr{\ell}{2}}\Big(\big\|{\bf A}_0|\nb_v|^{\ell+1}\big(\la v \ra^{\max\{2,|\al|\}} g\big)\big\|_{L^2_{x,v}}
    +\sum_{|\beta|\le\ell}\big\|{\bf A}_0\pr_v^\beta(\la v \ra^{\max\{2,|\al|-1\}} g)\big\|_{L^2_{x,v}}\Big)\\
    \les \nn&(\kappa\nu t)^{\fr{\ell}{2}}\big\||\nb_v|^{\ell+1}\big(\la v \ra^{\max\{2,|\al|\}} g_0\big)\big\|_{L^2_{v}}
    +\sum_{|\beta|\le\ell}(\kappa\nu t)^{\fr{|\beta|}{2}}\big\|\pr_v^\beta(\la v \ra^{\max\{2,|\al|-1\}} g_0)\big\|_{L^2_{v}}\\
    \nn&+(\kappa\nu t)^{\fr{\ell}{2}}\big\||\nb_v|^{\ell+1}\mathcal{M}\big(\la v \ra^{\max\{2,|\al|\}} h_{\ne}\big)\big\|_{L^2_{x,v}}
    +\sum_{|\beta|\le\ell}(\kappa\nu t)^{\fr{|\beta|}{2}}\big\|\pr_v^\beta\mathcal{M}(\la v \ra^{\max\{2,|\al|\}} h_{\ne})\big\|_{L^2_{x,v}}.
\end{align}
Moreover,
\begin{align*}
    \sum_{0\le\ell\le s}(\kappa\nu t)^\ell \big\||\nb_v|^{\ell}\mathcal{M}(v^\al {h}_{\ne})\big\|_{L^2_{x,v}}^2\le \frak{E}^s_\al(h_{\ne}(t))\le 4\nu^{-\fr13}\frak{D}^s_\al (h_{\ne}(t)).
\end{align*}
It follows from the above three estimates that
\begin{align}\label{es:G21-ne}
\sum_{0\le\ell\le s}\sum_{k\in\Z^3_*}(\kappa \nu t)^\ell G_{k,2;1)}^{\ell,\al}
\les\nn&\nu^{-\fr12}\big(\mathcal{D}^s_{\max\{2,|\al|\}}(h_0(t))\big)^{\fr12}\big(\frak{E}^s_\al(h_{\ne}(t))\big)^{\fr12}\|{\bf B}_0\rho\|_{L^2_x}\\
\nn&+\nu^{-\fr16}\big(\mathcal{E}^s_{\max\{2,|\al|-1\}}(h_0(t))\big)^{\fr12}\big(\frak{D}^s_\al(h_{\ne}(t))\big)^{\fr12}\|{\bf B}_0\rho\|_{L^2_x}\\
\nn&+\nu^{-\fr12}\big(\mathcal{D}^s_{\max\{2,|\al|\}}(h_{\ne}(t))\big)^{\fr12}\big(\frak{E}^s_\al(h_{\ne}(t))\big)^{\fr12}\|{\bf B}_0\rho\|_{L^2_x}\\
&+\nu^{-\fr16}\big(\mathcal{E}^s_{\max\{2,|\al|\}}(h_{\ne}(t))\big)^{\fr12}\big(\frak{D}^s_\al(h_{\ne}(t))\big)^{\fr12}\|{\bf B}_0\rho\|_{L^2_x}.
\end{align}

The nonlinearity in $H_{k,2;1)}^{\ell,\al}$ can be treated in a similar manner. In fact,
using \eqref{offD}, \eqref{est:Dnbg} with $\ell=0$, \eqref{log-produ}, we find that
\begin{align*}
    &\sum_{0\le\ell\le s}\sum_{k\in\Z^3_*}(C_s\kappa \nu^{\fr13} t)^\ell H_{k,2;1)}^{\ell,\al}\\
    \les&\sum_{k\in\Z^3_*}\sum_{l\in\Z^3_*}{\rm A}_{k,s}(t) \fr{|\rho_l|}{|l|}\Big\|\big(v^\al \mathbf{D}_k[\nb_vg_{k-l}]\big)\Big\|_{L^2_v}\big\|{\rm A}_{k,s}(t)(v^\al {h}_k)\big\|_{L^2_v}\\
    \les &\|{\bf B}_0(t)\rho\|_{L^2_x}\Big(\big\||\nb_v|{\bf A}_s(t)\big(\la v \ra^{\max\{2,|\al|\}} g_{\ne}\big)\big\|_{L^2_{x,v}}\\
    &\qquad\qquad\qquad\qquad\qquad+\big\|{\bf A}_s(\la v \ra^{\max\{2,|\al|-1\}} g_{\ne })\big\|_{L^2_{x,v}}\Big)\|{\bf A}_s(t)(v^\al h_{\ne})\|_{L^2_{x,v}}\\
    &+\|{\bf B}_s(t)\rho\|_{L^2_x}\Big(\big\||\nb_v|{\bf A}_0(t)\big(\la v \ra^{\max\{2,|\al|\}} g\big)\big\|_{L^2_{x,v}}\\
    &\qquad\qquad\qquad\qquad\qquad+\big\|{\bf A}_0(t)(\la v \ra^{\max\{2,|\al|-1\}} g)\big\|_{L^2_{x,v}}\Big)\|{\bf A}_s(t)(v^\al h_{\ne})\|_{L^2_{x,v}}.
\end{align*}
Then noting that
\begin{align*}
    \|{\bf A}_s(v^\al h_{\ne})\|_{L^2_{x,v}}\le \big(\frak{E}^s_\al(h_{\ne}(t))\big)^{\fr12}\le 2\nu^{-\fr16}\big(\frak{D}^s_\al (h_{\ne}(t))\big)^{\fr12},
\end{align*}
using \eqref{onD}, similar to \eqref{es:G21-ne}, we have
\begin{align}\label{es:H21-ne}
    \sum_{0\le\ell\le s}\sum_{k\in\Z^3_*}(C_s\kappa \nu^{\fr13} t)^\ell H_{k,2;1)}^{\ell,\al}
    \les \nn&\nu^{-\fr12}\big(\mathcal{D}^s_{\max\{2,|\al|\}}(h_{\ne}(t))\big)^{\fr12}\big(\frak{E}^s_\al(h_{\ne}(t))\big)^{\fr12}\|{\bf B}_0\rho\|_{L^2_x}\\
\nn&+\nu^{-\fr16}\big(\mathcal{E}^s_{\max\{2,|\al|\}}(h_{\ne}(t))\big)^{\fr12}\big(\frak{D}^s_\al(h_{\ne}(t))\big)^{\fr12}\|{\bf B}_0\rho\|_{L^2_x}\\
    \nn&+\nu^{-\fr12}\big(\mathcal{D}^0_{\max\{2,|\al|\}}(h_{\ne}(t))\big)^{\fr12}\big(\frak{E}^s_\al(h_{\ne}(t))\big)^{\fr12}\|{\bf B}_s(t)\rho\|_{L^2_x}\\
\nn&+\nu^{-\fr16}\big(\mathcal{E}^0_{\max\{2,|\al|\}}(h_{\ne}(t))\big)^{\fr12}\big(\frak{D}^s_\al(h_{\ne}(t))\big)^{\fr12}\|{\bf B}_s(t)\rho\|_{L^2_x}\\
    \nn&+\nu^{-\fr12}\big(\mathcal{D}^0_{\max\{2,|\al|\}}(h_{0}(t))\big)^{\fr12}\big(\frak{E}^s_\al(h_{\ne}(t))\big)^{\fr12}\|{\bf B}_s(t)\rho\|_{L^2_x}\\
&+\nu^{-\fr16}\big(\mathcal{E}^0_{\max\{2,|\al|\}}(h_{0}(t))\big)^{\fr12}\big(\frak{D}^s_\al(h_{\ne}(t))\big)^{\fr12}\|{\bf B}_s(t)\rho\|_{L^2_x}.
\end{align}

{\it Case 2: $k=0$.} This case is much easier to handle, since the nonlinear term ${\bf D}_k[N[g_k]]$ now reduces to $\sum_{l\in\Z^3_{*}}\rho_l\fr{il}{|l|^2}\cdot\nb_vg_{-l}$.  We skip the details and state the result here:
\begin{align}\label{es:GH21-0}
    \nn&\sum_{0\le\ell\le s}(\kappa \nu t)^\ell G_{0,2;1)}^{\ell,\al}+\sum_{0\le\ell\le s}(C_s\kappa \nu^{\fr13} t)^\ell H_{0,2;1)}^{\ell,\al}\\
    \les\nn&\nu^{-\fr12}\big(\mathcal{D}^s_{\max\{2,|\al|\}}(h_{\ne}(t))\big)^{\fr12}\big(\frak{E}^s_\al(h_{0}(t))\big)^{\fr12}\|\rho\|_{L^2_x}\\
&+\nu^{-\fr16}\big(\mathcal{D}^s_{\max\{2,|\al|-1\}}(h_{\ne}(t))\big)^{\fr12}\big(\frak{E}^s_\al(h_{0}(t))\big)^{\fr12}\|\rho\|_{L^2_x}.
\end{align}

Next we investigate $G^{\ell,\al}_{k,2;2)}$ and $H^{\ell,\al}_{k,2;2)}$. Note first that $G^{\ell,\al}_{k,2;2)}$ and $H^{\ell,\al}_{k,2;2)}$ vanish when $k=0$. Here we use the notation introduced in \eqref{com:DL}.   Thanks to \eqref{bd:com2}, \eqref{bd:com3} and \eqref{onD}, using the fact $\nu t\le1$, we have
\begin{align}\label{est:Gk22}
\nn&\sum_{0\le\ell\le s}\sum_{k\in\Z^3_*}(\kappa \nu t)^\ell G^{\ell,\al}_{k,2;2)}\\
=\nn&2\nu\sum_{i=1}^2\sum_{0\le\ell\le s}\sum_{k\in\Z^3_*}(\kappa \nu t)^\ell\frak{Re}\int_{\xi}|\xi|^\ell\mathcal{M}(k,\xi)\pr_{\xi}^\al \mathcal{F}_v[{\bf com}_{k,i}]\big(|\xi|^\ell\mathcal{M}(k,\xi)\pr_\xi^\al\bar{\hat{h}}_k\big)d\xi\\
\les\nn&\nu \sum_{i=1}^2\sum_{0\le\ell\le s}\sum_{k\in\Z^3_*}(\kappa \nu t)^\ell\ln(e+|k|)\big\||\nb_v|^\ell (v^\al {\bf com}_{k,i})\big\|_{L^2_v}\big\||\nb_v|^\ell\mathcal{M}(v^\al h_k)\big\|_{L^2_v}\\
\les\nn&\nu \sum_{0\le\ell\le s}\sum_{|\beta|\le\ell+1}(\kappa \nu t)^\ell\big\|\mathcal{M}\pr_v^\beta (\la v\ra^2h_{\ne})\big\|_{L^2_{x,v}}\big\||\nb_v|^\ell\mathcal{M}(v^\al h_{\ne})\big\|_{L^2_{x,v}}\\
\les\nn&\nu^{\fr13} \sum_{0\le\ell\le s}\Big(\nu^{\fr12}(\kappa \nu t)^{\fr{\ell}{2}}\big\||\nb_v|^{\ell+1}\mathcal{M}(\la v\ra^2h_{\ne})\big\|_{L^2_{x,v}}\Big)\Big(\nu^{\fr16}(\kappa \nu t)^{\fr{\ell}{2}}\big\||\nb_v|^\ell\mathcal{M}(v^\al h_{\ne})\big\|_{L^2_{x,v}}\Big)\\
\nn&+\nu^{\fr23} \sum_{0\le\ell\le s}\sum_{|\beta|\le \ell}\Big(\nu^{\fr16}(\kappa \nu t)^{\fr{|\beta|}{2}}\big\|\mathcal{M}\pr_v^\beta (\la v\ra^2h_{\ne})\big\|_{L^2_{x,v}}\Big)\Big(\nu^{\fr16}(\kappa \nu t)^{\fr{\ell}{2}}\big\||\nb_v|^\ell\mathcal{M}(v^\al h_{\ne})\big\|_{L^2_{x,v}}\Big)\\
\les&\nu^{\fr13}\big(\mathcal{D}^s_{2}(h_{\ne}(t))\big)^{\fr12}\big(\frak{D}^s_{\al}(h_{\ne}(t))\big)^{\fr12}
+\nu^{\fr23}\big(\mathcal{D}^s_{2}(h_{\ne}(t))\big)^{\fr12}\big(\frak{D}^s_{\al}(h_{\ne}(t))\big)^{\fr12}.
\end{align}
Similarly, we have
\begin{align}\label{est:Hk22}
\nn&\sum_{0\le\ell\le s}\sum_{k\in\Z^3_*}(C_s\kappa \nu^{\fr13} t)^\ell H^{\ell,\al}_{k,2;2)}\\
\les\nn&\nu \sum_{i=1}^2\sum_{0\le\ell\le s}\sum_{k\in\Z^3_*}(C_s\kappa \nu^{\fr13} t)^\ell\ln(e+|k|)\big\||k|^{\fr{\ell}{3}} (v^\al {\bf com}_{k,i})\big\|_{L^2_v}\big\||k|^{\fr{\ell}{3}}\mathcal{M}(v^\al h_k)\big\|_{L^2_v}\\
\les\nn&\nu \sum_{0\le\ell\le s}\sum_{|\beta|\le1}(C_s\kappa \nu^{\fr13} t)^\ell\big\||\nb_x|^{\fr{\ell}{3}}\mathcal{M}\pr_v^\beta (\la v\ra^2h_{\ne})\big\|_{L^2_{x,v}}\big\||\nb_x|^{\fr{\ell}{3}}\mathcal{M}(v^\al h_{\ne})\big\|_{L^2_{x,v}}\\
\les\nn&\nu^{\fr13} \sum_{0\le\ell\le s}\Big(\nu^{\fr12}(C_s\kappa \nu^{\fr13} t)^{\fr{\ell}{2}}\big\||\nb_v||\nb_x|^{\fr{\ell}{3}}\mathcal{M} (\la v\ra^2h_{\ne})\big\|_{L^2_{x,v}}\Big)\\
\nn&\qquad\ \, \quad\times\Big(\nu^{\fr16}(C_s\kappa \nu^{\fr13} t)^{\fr{\ell}{2}}\big\||\nb_x|^{\fr{\ell}{3}}\mathcal{M}(v^\al h_{\ne})\big\|_{L^2_{x,v}}\Big)\\
\nn&+\nu^{\fr23} \sum_{0\le\ell\le s}\Big(\nu^{\fr16}(C_s\kappa \nu^{\fr13} t)^{\fr{\ell}{2}}\big\||\nb_x|^{\fr{\ell}{3}}\mathcal{M} (\la v\ra^2h_{\ne})\big\|_{L^2_{x,v}}\Big)\\
\nn&\qquad\ \ \quad\times\Big(\nu^{\fr16}(C_s\kappa \nu^{\fr13} t)^{\fr{\ell}{2}}\big\||\nb_x|^{\fr{\ell}{3}}\mathcal{M}(v^\al h_{\ne})\big\|_{L^2_{x,v}}\Big)\\
\les&\nu^{\fr13}\big(\mathcal{D}_2^s(h_{\ne}(t))\big)^{\fr12}\big(\frak{D}_{\alpha}^s(h_{\ne}(t))\big)^{\fr12}.
\end{align}

Finally, we turn to estimate $G^{\ell,\al}_{k,2;3)} $ and $H^{\ell,\al}_{k,2;3)}$. We split the estimates into the cases $k\ne$ and $k=0$ as above. 

{\it Case 1: $k\ne0$.}
Noting that
\[
\pr_\xi^\al(\xi\cdot\nb_\xi\hat{h}_k)=\xi\cdot\nb_\xi\pr_\xi^\al\hat{h}_k+|\al|\pr_\xi^\al\hat{h}_k,
\]
then
\begin{align*}
    &\sum_{0\le\ell\le s}\sum_{k\in\Z^3_*}(\kappa \nu t)^\ell G_{k,2;3)}^{\ell,\al}\\
    =&-2\nu\sum_{0\le\ell\le s}\sum_{k\in\Z^3_*}(\kappa \nu t)^\ell\frak{Re}\int_{\xi}\xi\cdot\nb_\xi\big(|\xi|^{\ell}\mathcal{M}(k,\xi)\pr_{\xi}^\al\hat{h}_k\big)\big(|\xi|^\ell\mathcal{M}(k,\xi)\pr_\xi^\al\bar{\hat{h}}_k\big)d\xi\\
    &+2\nu \sum_{0\le\ell\le s}\sum_{k\in\Z^3_*}(\kappa \nu t)^\ell\frak{Re}\int_{\xi}\xi\cdot\nb_\xi\big(|\xi|^{\ell}\mathcal{M}(k,\xi)\big)\pr_{\xi}^\al\hat{h}_k\big(|\xi|^\ell\mathcal{M}(k,\xi)\pr_\xi^\al\bar{\hat{h}}_k\big)d\xi\\
    &-2|\al|\nu \sum_{0\le\ell\le s}(\kappa \nu t)^\ell\big\||\nb_v|^{\ell}\mathcal{M}(v^\al h_{\ne})\big\|^2_{L^2_{x,v}}\\
    =&{ \nu\sum_{0\le\ell\le s}(3+2\ell-2|\al|)(\kappa \nu t)^\ell\big\||\nb_v|^{\ell}\mathcal{M}(v^\al h_{\ne})\big\|^2_{L^2_{x,v}}}\\
    &+2\nu\sum_{0\le\ell\le s}\sum_{k\in\Z^3_*}(\kappa \nu t)^\ell\frak{Re}\int_{\xi}|\xi|^{\ell}\xi\cdot\nb_\xi \mathcal{M}(k,\xi)\big)\pr_{\xi}^\al\hat{h}_k\big(|\xi|^\ell\mathcal{M}(k,\xi)\pr_\xi^\al\bar{\hat{h}}_k\big)d\xi.
\end{align*}
Recalling the definition of $\mathcal{M}$, we have
\[
|\nb_\xi \mathcal{M}(k,\xi)|\le\fr14\nu^{\fr13}|k|^{-\fr13}\ln(e+|k|){\bf 1}_{|k|\ne0}.
\]
Then
\begin{align*}
    &2\nu\sum_{0\le\ell\le s}\sum_{k\in\Z^3_*}(\kappa \nu t)^\ell\frak{Re}\int_{\xi}|\xi|^{\ell}\xi\cdot\nb_\xi \mathcal{M}(k,\xi)\big)\pr_{\xi}^\al\hat{h}_k\big(|\xi|^\ell\mathcal{M}(k,\xi)\pr_\xi^\al\bar{\hat{h}}_k\big)d\xi\\
    \le&\fr12\nu^{\fr43}\sum_{0\le\ell\le s}(\kappa \nu t)^\ell\big\||\nb_v|^{\ell}\mathcal{M}(v^\al{h}_{\ne})\big\|_{L^2_{x,v}}\big\||\nb_v|^{\ell+1}\mathcal{M}(v^\al{h}_{\ne})\big\|_{L^2_{x,v}}.
\end{align*}
Similarly,
\begin{align*}
    \nn&\sum_{0\le\ell\le s}\sum_{k\in\Z^3_*}(C_s\kappa \nu^{\fr13} t)^\ell H_{k,2;3)}^{\ell,\al}\\
    \le& \nu (3-2|\al|)\sum_{0\le\ell\le s}(C_s\kappa \nu^{\fr13}t)^{\ell}\big\||\nb_x|^{\fr{\ell}{3}}\mathcal{M}(v^\al h_{\ne})\big\|^2_{L^2_{x,v}}\\
    &+\fr12\nu^{\fr43}\sum_{0\le\ell\le s}(C_s\kappa \nu^{\fr13} t)^\ell\big\||\nb_x|^{\fr{\ell}{3}}\mathcal{M}(v^\al{h}_{\ne})\big\|_{L^2_{x,v}}\big\||\nb_v||\nb_x|^{\fr{\ell}{3}}\mathcal{M}(v^\al{h}_{\ne})\big\|_{L^2_{x,v}}.
\end{align*}
It follows that
\begin{align}\label{es: G23-ne}
    \nn&\sum_{0\le\ell\le s}\sum_{k\in\Z^3_*}(\kappa \nu t)^\ell G_{k,2;3)}^{\ell,\al}+\sum_{0\le\ell\le s}\sum_{k\in\Z^3_*}(C_s\kappa \nu^{\fr13} t)^\ell H_{k,2;3)}^{\ell,\al}\\
    \le&  (3+2s)\nu \frak{E}^s_\al(h_{\ne}(t))+\nu^{\fr23}\frak{D}^s_{\al}(h_{\ne}(t))\les \big(4(3+2s)+1\big)\nu^{\fr23}\frak{D}^s_{\al}(h_{\ne}(t)).
\end{align}

{\it Case 2: $k=0$.} Now $M(0,\xi)\equiv1$, and we cannot use $\frak{D}^s_{\al}(h_0(t))$ to bound $\frak{E}^s_{\al}(h_0(t))$. Thus \eqref{es: G23-ne} reduces to
\begin{align}\label{es: G23-0}
    \sum_{0\le\ell\le s}(\kappa \nu t)^\ell G_{0,2;3)}^{\ell,\al}+\sum_{0\le\ell\le s}(C_s\kappa \nu^{\fr13} t)^\ell H_{0,2;3)}^{\ell,\al}
    \le { (3+2s)\nu \frak{E}^s_\al(h_{0}(t))}.
\end{align}

\noindent\underline{\bf Improvement of \eqref{hypo2}}:
Substituting \eqref{es:com-v1}, \eqref{es:G21-ne}, \eqref{es:H21-ne}, and \eqref{est:Gk22}--\eqref{es: G23-ne} into \eqref{es:en-s-al2}, summing over $|\al|\le m$, then integrating the resulting inequality w.r.t the time variable $t$ over $[0,T^*]$ and using \eqref{bd:initial}, we find that there exist two positive constants $\nu_{s,m}<1<C_{s,m}$ depending on $s$ and $m$ but independent of $\nu$, such that for all $0<\nu\le \nu_{s,m}$,
\begin{align}\label{es:h-ne}
    \nn&\mathcal{E}^s_{m}(h_{\ne}(t))+\fr14\int_0^{T^*}\mathcal{D}^{s}_{m}(h_{\ne}(t))dt\\
    \le\nn&C_{s,m}\|\la v\ra^m \ln(e+|\nb_x|)(g_{\rm in})_{\ne}\|_{L^2_{x,v}}^2\\
    \nn&+C_{s,m}\nu^{-\fr12}\int_0^{T^*}\big(\mathcal{D}^s_m(h_{\ne}(t))\big)^\fr12\big(\mathcal{E}^s_m(h_{\ne}(t))\big)^\fr12\|{\bf B}_s(t)\rho\|_{L^2_x}dt\\
    \nn&+C_{s,m}\nu^{-\fr12}\int_0^{T^*}\big(\mathcal{D}^s_m(h_{0}(t))\big)^\fr12\big(\mathcal{E}^s_m(h_{\ne}(t))\big)^\fr12\|{\bf B}_s(t)\rho\|_{L^2_x}dt\\
    &+C_{s,m}\nu^{-\fr16}\int_0^{T^*}\big(\mathcal{E}^s_m(h_{0}(t))\big)^\fr12\big(\mathcal{D}^s_m(h_{\ne}(t))\big)^\fr12\|{\bf B}_s(t)\rho\|_{L^2_x}dt.
\end{align}
Using the bootstrap hypotheses \eqref{hypo}, for the initial data $g_{\rm in}$ satisfying \eqref{initial-sta}, one deduces from the above inequality that
\begin{align}
    \nn&\mathcal{E}^s_{m}(h_{\ne}(t))+\fr14\int_0^{T^*}\mathcal{D}^{s}_{m}(h_{\ne}(t))dt\\
    \le\nn&C_{s,m}(\eps\nu^{\fr12})^2+128 C_{s,m}{\rm C}_1{\rm C}_2^2\eps(\eps\nu^{\fr12})^2+256C_{s,m}{\rm C}_1{\rm C}_2{\rm C}_3\eps(\eps\nu^{\fr12})^2,
\end{align}
which is sufficient to improve \eqref{hypo2}, provided we  take
\begin{align}\label{det:C2}
{\rm C}_2:=C_{s,m}\quad {\rm and }\quad \eps\le\eps_2:=\fr{1}{256{\rm C}_1({\rm C}_2+{\rm C}_3)}.
\end{align}

\noindent\underline{\bf Improvement of \eqref{hypo3}}:
Substituting \eqref{es:com-v2}, \eqref{es:GH21-0} and \eqref{es: G23-0} into \eqref{es:en-s-al2-0}, and summing over $|\al|\le m$, we are led to
\begin{align}
    \fr{d}{dt}\mathcal{E}^s_{m}(h_0(t))+\fr14\mathcal{D}^{s}_{m}(h_0(t))\le\nn& C_{s,m}\nu\mathcal{E}^s_m(h_0(t))
    +C_{s,m}\nu^{-\fr12}\big(\mathcal{D}_{m}^s(h_{\ne}(t))\big)^{\fr12}\big(\mathcal{E}^s_m(h_0(t))\big)^{\fr12}\|\rho\|_{L^2_x}.
\end{align}
Noting that $\nu t\le1$, by using Gronwall's inequality, in view of the bootstrap hypotheses \eqref{hypo} and \eqref{initial-sta}, we have 
\begin{align}
    \mathcal{E}^s_{m}(h_0(t))+\fr14\int_0^{T^*}\mathcal{D}^{s}_{m}(h_0(t))dt\le e^{C_{s,m}}\left(2(\eps\nu^{\fr12})^2+128C_{s,m}{\rm C_1C_2C_3}\eps(\eps\nu^{\fr12})^2\right),
\end{align}
which is sufficient to improve \eqref{hypo3}, provide we take
\begin{align}\label{det:C3}
    {\rm C}_3:=2e^{C_{s,m}},\quad{\rm and}\quad \eps\le\eps_3:=\fr{1}{64C_{s,m}{\rm C_1C_2}}.
\end{align}
The proof of Proposition \ref{prop:stablity} is complete.

We are now in a position to prove Theorem \ref{thm:main1}.

\noindent\underline{Proof of Theorem \ref{thm:main1}}: Recalling \eqref{det:C1}, \eqref{det:C2} and \eqref{det:C3}, now let us choose 
\begin{align*}
    \nu_0:=\min\{\nu_1,\nu_{s,m}\},\quad{\rm and}\quad \eps_0:=\min\{\eps_1,\eps_2,\eps_3\}.
\end{align*}
Then the long-time bounds \eqref{bd:distribution} and \eqref{bd:E} follow from Proposition \ref{prop:stablity} with $s=0$ immediately. For the improved bound \eqref{bd:improve}, we first consider the case $s\in\N\setminus\{0\}$. Indeed, recalling the definition of the energy functional $\mathcal{E}^s_m(h(t))$ in \eqref{def-En}, using the bounds established for $h_{\ne}$ and $h_0$ in Proposition \ref{prop:stablity} with  $s$ replaced by $3s$, together with the boundedness of the wave operator ${\bf D}_k$ \eqref{offD}, we arrive at \eqref{bd:improve}. The general case $s>0$  can be obtained by interpolation immediately. 

Once \eqref{bd:improve} holds, one can use $g(\frac{1}{2}\nu^{-1},x,v)$ as a new initial data, which satisfies both the regularity and smallness requirements in Theorem 1 of \cite{bedrossian2017suppression}. Thus it holds that
\begin{align*}
\|E(t)\|_{L^2}\lesssim& \frac{\nu^{\frac12}}{\langle t-\frac12\nu^{-1}\rangle}, \\
\|\la v\ra^m\mathbb{P}_{\neq}g(t)\|_{L^2}\leq& C \|\la v\ra^m\mathbb{P}_{\neq}g_{\rm in}\|_{L^2}\langle \nu^{1/3}(t-\frac12\nu^{-1})\rangle^{-s},\quad{\rm for}\quad n=1,2,\\
            \|\la v\ra^m\ln(e+|\nb_x|)\mathbb{P}_{\neq}g(t)\|_{L^2}\leq& C \|\la v\ra^m\ln(e+|\nb_x|)\mathbb{P}_{\neq}g_{\rm in}\|_{L^2}\langle \nu^{1/3}(t-\frac12\nu^{-1})\rangle^{-s},\quad{\rm for}\quad n=3.
\end{align*} 
As a result, the Landau damping estimate \eqref{Lan-dam} and the enhanced dissipation estimate \eqref{eq: enhanced-dissipation} hold for all $t\ge\nu^{-1}$. It remains to prove \eqref{eq: enhanced-dissipation} for all $t\le\nu^{-1}$. The key point is that we can extract a factor on the right-hand side of \eqref{eq: enhanced-dissipation} that depends only on the non-zero mode $(g_{\rm in})_{\ne}$  of the initial data (one can further reduce $\nu_0$ and $\eps_0$ if necessary).
More precisely, by \eqref{bd:Bsrho} and Proposition \ref{prop:stablity}, we have
\begin{align}
    \nn\|{\bf B}_s\rho\|_{L^2_{t,x}}\le C\|\la v\ra^m\ln(e+|\nb_x|)(g_{\rm in})_{\ne}\|_{L^2_{x,v}}+C\nu^{\fr13}\Big(\int_0^{T^*}\mathcal{D}^s_{m}(h_{\ne}(t)) dt\Big)^{\fr12}.
\end{align}
Combining this with \eqref{es:h-ne} and Proposition \ref{prop:stablity},  one deduces that
\begin{align*}
    \nn&\mathcal{E}^s_{m}(h_{\ne}(t))+\fr{1}{16}\int_0^{T^*}\mathcal{D}^{s}_{m}(h_{\ne}(t))dt\\
    \le\nn& C\|\la v\ra^m\ln(e+|\nb_x|)(g_{\rm in})_{\ne}\|_{L^2_{x,v}}+C\nu^{-1}\int_0^{T^*}\mathcal{D}^s_m(h_{\ne}(t))\mathcal{E}^s_m(h_{\ne}(t))dt\\
    \nn&+C\nu^{-1}\int_0^{T^*}\mathcal{D}^s_m(h_{0}(t))\mathcal{E}^s_m(h_{\ne}(t))dt.
\end{align*}
Then \eqref{eq: enhanced-dissipation} follows from Gronwall's inequality and Proposition \ref{prop:stablity} immediately. This completes the proof of Theorem \ref{thm:main1}.

\section{Instability estimates}\label{sec:Insta}
We construct the initial perturbation as follows
\begin{align}\label{5.1}
    g_{\rm in}(x,v)=\underbrace{G_{n-1}(f^0(v_1)-\mu(v_1))}_{\text{zero mode with size $\nu^{1/2-}$}}
    +\underbrace{G_{n-1}\Big(\gamma^{\frac12}\nu^{\beta}\mathfrak{Re}\big(\frac{e_{\gamma,\lambda}(v_1)}{\hat{e}_{\gamma, \lambda}(0)}e^{i x_1}\big)\Big)}_{\text{non-zero modes with size $\nu^{1/2+}$}}, \quad \text{with}\quad \beta>\frac12
\end{align}
where $G_{n-1}=\left\{\begin{aligned}
    &\frac{1}{2\pi}e^{-(v_2^2+v_3^2)/2},\quad n=3\\
    &\frac{1}{(2\pi)^{\fr12}}e^{-v_2^2/2},\quad n=2\\
    &1, \quad n=1.
\end{aligned}\right.$ and the perturbed new homogeneous background distribution $f^0(v)$ (see \eqref{def-f0}) creates an unstable eigenvalue $\lambda$ for the linearized collisionless operator $\mathbb{VP}_{f^0}$ with the associated eigenfunction $G_{n-1}e_{\gamma,\lambda}(v_1)e^{i x_1}$. The rest of the paper will focus on $n=1$. We will first prove, see Lemma \ref{lem-lm}, that by choosing suitable $M$ and $\gamma$, the operator $\mathbb{VP}_{f^0}$ has an unstable eigenvalue $\lambda$ whose real part $\lambda_{\rm r}\geq c\gamma$ with $c>0$ independent of $\gamma$. 
\subsection{Reformulation} As mentioned in section \ref{sec:sec-insta}, to prove the instability result, we need to construct a spatially homogeneous solution $\tl{f}^0(t,v)$ to the VPFP equations near the global Maxwellian $\mu(v)$. Recalling the definition of $\tl{\mu}(v)$  in \eqref{def:tlmu}, we first define a spatially homogeneous perturbation $f^{\rm e}(t,v)$ around the global Maxwellian $\mu(v)$ as follows:
\begin{align}\label{eq:fe}
    \partial_tf^{\rm e}-\nu L[f^{\rm e}]=0,\quad {\rm with}\quad f^{\rm e}|_{t=0}=\fr{\tilde{\mu}(v)}{1+M\gamma^2m_{\sigma}},
\end{align}
where $\displaystyle m_\sig:=\int_{\R}\sig(v)dv$.
Then we introduce
\begin{align}\label{eq:tlf}
\tl{f}^0(t,v):=\frac{1}{1+M\gamma^2m_\sigma}\fr{1}{(2\pi)^{\fr12}}e^{-\fr{|v|^2}{2}}+f^{\rm e}(t,v),
\end{align}
which solves the Vlasov-Poisson-Fokker-Planck equation \eqref{VFP} with initial data
\begin{align}\label{def-f0}
\tl{f}^{0}|_{t=0}=f^0(v):=\fr{1}{1+M\gamma^2m_\sig}\left[\fr{1}{(2\pi)^{\fr12}}e^{-\fr{v^2}{2}}+\tl{\mu}(v)\right].
\end{align}
By conservation of mass, we have
\begin{align*}
\int_{\R}\tl{f}^0(t,v)dv=\int_{\R}f^0(v)dv=1.
\end{align*}
Now we introduce the perturbations around $\tl{f}^0(t,v)$ by
\[
\tl f(t,x,v):=F(t,x,v)-\tl{f}^0(t,v),
\]
and
\[
\tl{\rho}(t,x):=\int_{\R}\tl{f}(t,x,v)dv=\int_{\R}F(t,x,v)dv-1.
\]
Then the initial condition \eqref{initial-F} ensures that
\begin{align*}
\int_{\T}\tl{\rho}(t,x)dx=\int_{\T\times\R}F_{\rm in}(x,v)dvdx-2\pi=0.
\end{align*}
We thus find that $\tl{f}(t,x,v)$ solves
\begin{align}\label{eq-tlf}
\begin{cases}
\pr_t\tl{f}+v\pr_x\tl{f}+\tl{E}\pr_v \mu_1+\tl{E}\pr_vf^{\rm e}-\nu L[\tl{f}]=-\tl{E}\pr_v\tl{f},\\
\tl{E}=-\pr_x(-\pr_{xx})^{-1}\tl\rho,
\end{cases}
\end{align}
where 
\[
\mu_1(v):=\fr{1}{1+M\gamma^2m_\sig}\fr{1}{(2\pi)^{\fr12}}e^{-\fr{|v|^2}{2}}.
\]
In the following, we rewrite \eqref{eq-tlf} as
\begin{align}\label{eq-tlf'}
\begin{cases}
\pr_t\tl{f}+v\pr_x\tl{f}+\tl{E}\pr_v f^0-\nu L[\tl{f}]=\frak{N},\\
\tl{E}=-\pr_x(-\pr_{xx})^{-1}\tl\rho,
\end{cases}
\end{align}
where
\begin{align}\label{def:NL-N}
\frak{N}(t,x,v)=-\tl{E}\pr_v\tl{f}-\tl{E}\pr_v(f^{\rm e}-\mu_2),
\end{align}
and
\begin{align}\label{def:mu2}
\mu_2(v):=f^{\rm e}(0,v)=\fr{\tl{\mu}(v)}{1+M\gamma^2m_\sigma}=\fr{M\gamma \sigma (\fr{v}{\gamma})}{1+M\gamma^2m_\sigma}.
\end{align}

The next strategy is to show that the linearized system of \eqref{eq-tlf'} admits solutions with exponential growth. At first, instead of considering the linearized system of \eqref{eq-tlf'} directly, we investigate the linearized system without collision:
\begin{align}\label{LVP}
\pr_tf+v\pr_xf-\pr_x(-\pr_{xx})^{-1}\frak{I}[f]\pr_vf^0=0.
\end{align}
where 
\begin{align}\label{def-I}
    \frak{I}[f](t,x):=\int_{\R}f(t,x,v)dv,\quad{\rm with}\quad \int_{\mathbb{T}}\frak{I}[f](t,x)dx=0.
\end{align}
On the Fourier side, \eqref{LVP} can be rewritten as
\begin{align}\label{eq-f_k}
\pr_tf_k+ikvf_k-ik^{-1}\frak{I}[f_k]\pr_vf^0=0.
\end{align}
Let us denote the linearized Vlasov-Poisson operator around $f^0(v)$ at mode 1 by
\[
\mathbb{L}_1f:=-ivf+i\int_{\R}f(v)dv\pr_vf^0.
\]
\begin{lem}\label{lem: eigenvalue and Wronskian}
    Let $\lambda=\lm_{\rm r}+i\lm_{\rm i}$ be an unstable eigenvalue of $\mathbb{L}_1$, with the associated eigenfunction $e_{\gamma, \lm}$, namely, $\lm_{\rm r}>0$, and 
    \begin{align}\label{eigen}
\mathbb{L}_1e_{\gamma,\lm}=\lm e_{\gamma,\lm}.
\end{align} 
Then, $\lm$ is a root of 
\begin{align}\label{Psi=0}
    \Psi_\gamma (\lm,M)=0,
\end{align}
and the Fourier transform of the eigenfunction has the following expression
\begin{align}\label{exp-eigen}
    \hat{e}_{\gamma,\lm}(\xi)=\begin{cases}
    \displaystyle e^{\lm\xi}\int_{+\infty}^\xi\zeta e^{-\lm\zeta}\widehat{f^0}(\zeta)d\zeta \hat{e}_{\gamma,\lm}(0),\quad \xi\ge0\\[2mm]
    \displaystyle e^{\lm\xi}\hat{e}_{\gamma,\lm}(0)+e^{\lm\xi}\int_0^\xi\zeta e^{-\lm\zeta}\widehat{f^0}(\zeta)d\zeta \hat{e}_{\gamma,\lm}(0),\quad \xi\le0.
    \end{cases}
\end{align}
Here, 
\begin{align*}
\Psi_\gamma(\lm,M)
:=&1+\int_0^\infty\zeta e^{-\lm \zeta}\widehat{f^0}(\zeta)d\zeta\\
=&1+\fr{1}{1+M\gamma^2m_\sig}\left[\int_0^\infty e^{-\lm t}te^{-\fr{t^2}{2}}dt+M\int_0^\infty e^{-\fr{\lm t}{\gamma}}t \hat{\sig}( t) dt\right].
\end{align*}
\end{lem}
\begin{proof}
    Taking the Fourier transform of \eqref{eigen} in $v$, we arrive at
\begin{align}\label{eq:hat-e}
    \fr{d}{d\xi}\hat{e}_{\gamma,\lm}(\xi)+i\hat{e}_{\gamma,\lm}(0)\widehat{\pr_vf^0}(\xi)=\lm\hat{e}_{\gamma,\lm}(\xi),
\end{align}
We obtain \eqref{exp-eigen} directly by solving the ODE with vanishing condition at infinity. To ensure the continuity of $\hat{e}_{\gamma,\lm}(\xi)$ at $\xi=0$, we require $\Psi_\gamma(\lm,M)=0$. 
\end{proof}

To construct solutions to \eqref{LVP} with exponential growth, we want to adjust the parameter $M$ and prove the existence of a root $\lm$ of \eqref{Psi=0} with positive real part.
This will be the goal of our next section.

\subsection{Generation of unstable eigenvalue}
To achieve \eqref{Psi=0}, we first  seek special zeros of the function $\Psi_\gamma(\lm,M)$ with $\lm=0$.
Before proceeding any further, to simplify the presentation, we introduce the following notation. Let
\[
B_1:=\int_0^\infty  t \hat{\sigma}(t) dt,\quad B_2:=\int_0^\infty  t^2 \hat{\sigma}(t) dt,
\]
and  $\gamma_0$ be a small positive constant such that
\begin{align}\label{small-gamma0}
    0<\gamma_0\le\min\left\{1,\sqrt{\fr{|B_1|}{4|m_\sig|}}, \sqrt{\fr{3|B_1|}{(3|B_1|+8)|m_\sig|}}, \fr35\sqrt{\fr{2}{\pi}}\fr{|B_2|}{|B_1|}\right\}.
\end{align}
Assume that  the function $\sig$ is chosen such that 
\begin{align}\label{H-sig1}
    B_1<0, \quad{\rm and}\quad B_2<0.
\end{align}
For each $\gamma\in(0,\gamma_0]$, this enables us to define a positive constant
\begin{align}\label{M_0}
M_0:=-\fr{2}{B_1+\gamma^2m_\sig},
\end{align}
such that
\begin{align}\label{Psi0}
    \Psi_\gamma(0,M_0)=0.
\end{align}

To obtain a suitable unstable eigenvalue (with positive real part) starting from \eqref{Psi0} and satisfying \eqref{Psi=0}, we treat $\lm$ as a function of $M$ and solve the ODE governing the evolution of $\lm(M)$ with respect to $M$. To this end, we introduce two functions
\begin{align*}
    N_\gamma(\lm,M)=&\int_0^\infty e^{-\frac{\lambda t}{\gamma} } t \hat{\sigma}(t) dt - \gamma^2 m_\sigma \int_0^\infty e^{-\lambda t} t e^{-\frac{t^2}{2}} dt,\\
    D_\gamma(\lm,M)=&(1+M\gamma^2 m_\sigma) \left[ \int_0^\infty e^{-\lambda t}t^2  e^{-\frac{t^2}{2}} dt + \frac{M}{\gamma} \int_0^\infty e^{-\frac{\lambda t}{\gamma} } t^2  \hat{\sigma}(t) dt \right].
\end{align*}
Noting that
\begin{align*}
   \frac{\partial \Psi_\gamma}{\partial \lambda}(\lm,M) = -\frac{D_\gamma(\lm,M)}{(1 + M\gamma^2 m_\sigma)^2},\quad {\rm and}\quad \frac{\partial \Psi_\gamma}{\partial M}  = \frac{ N_\gamma(\lm,M) }{(1 + M\gamma^2 m_\sigma)^2},
\end{align*}
by the chain rule, and in view of \eqref{Psi0}, it is natural to introduce the following ODE:
\begin{align}\label{ODE}
\begin{cases}
    \fr{d\lm}{dM}=\fr{N_\gamma(\lm,M)}{D_\gamma(\lm,M)}, \quad (\lm_{\rm r},\lm_{\rm i},M)\in\mathsf{R}_\dl,\\[2mm]
    \lm(M_0)
    =0,
\end{cases}
\end{align}
where the rectangle $\mathsf{R}_\dl$ is defined by 
\begin{align}\label{def:dom}
\mathsf{R}_\dl:=\left\{(\lm_{\rm r},\lm_{\rm i},M): 0\le \lm_{\rm r}\le \dl\gamma, |\lm_{\rm i}|\le \dl\gamma, M_0\le M\le M_0+1\right\},
\end{align}
with $\dl>0$  a small constant independent of $\gamma$ to be determined later.

We shall establish the following lemma, which asserts that there exist unstable eigenvalues $\lm(M)$, together with some positive constants $M$, satisfying \eqref{Psi=0}. The key point is that the results hold uniformly for all $\gamma\in(0,\gamma_0]$.
\begin{lem}\label{lem-lm}
Assume that $\gamma_0$ satisfies \eqref{small-gamma0} and $\gamma\in(0,\gamma_0]$. The following conclusions hold:
\begin{enumerate}
    \item[(i)] There exists a  constant $\delta\in(0,1)$ independent of $\gamma$, such that $\fr{N_\gamma(\lm,M)}{D_\gamma(\lm,M)}$ is continuous on $\mathsf{R}_\dl$, and Lipchitz with respect to $\lm$.
    \item[(ii)] There exists a  constant $\varepsilon\in(0,1)$ independent of $\gamma$, such that the equation \eqref{ODE} admits a unique solution $\lm=\lm(M)$ for $M\in[M_0,M_0+\varepsilon]$. Moreover, there exist two positive constants $\underline{c}$ and $\bar{C}$ independent of $\gamma$, such that for $M\in[M_0, M_0+\eps]$, there hold
    \begin{align}\label{bd:lm_r}
    \underline{c}(M-M_0)\gamma \le\lm_{\rm r}(M)\le\bar{C}(M-M_0)\gamma,
    \end{align}
and
\begin{align}\label{bd:lm_i}
    |\lm_{\rm i}(M)|\le \bar{C}(M-M_0)\gamma.
\end{align}

\item[(iii)] 
The solution $\lm=\lm(M)$ satisfies 
\begin{align}\label{equiv}
    \Psi_\gamma(\lm(M),M)\equiv0,\quad {\rm for} \ \ M\in[M_0,M_0+\varepsilon].
\end{align}
\end{enumerate}
\end{lem}

\begin{proof}
To prove (i), we write
\begin{align*}
    N_\gamma(\lm,M)=&B_1 - \gamma^2 m_\sigma+{\rm F}(\lm,M),\\
    D_\gamma(\lm,M)=&(1+M\gamma^2 m_\sigma) \left[ \sqrt{\fr{\pi}{2}} + \frac{M}{\gamma} B_2 \right]+{\rm G}(\lm, M)
\end{align*}
where
\[
{\rm F}(\lm,M):=\int_0^\infty (e^{-\frac{\lambda t}{\gamma} }-1) t \hat{\sigma}(t) dt - \gamma^2 m_\sigma \int_0^\infty (e^{-\lambda t}-1) t e^{-\frac{t^2}{2}} dt,
\]
and
\[
{\rm G}(\lm, M):=(1+M\gamma^2 m_\sigma) \left[ \displaystyle \int_0^\infty(e^{-\lambda t}-1) t^2  e^{-\frac{t^2}{2}} dt + \frac{M}{\gamma} \int_0^\infty (e^{-\frac{\lambda t}{\gamma} }-1)t^2  \hat{\sigma}(t) dt \right].
\]
Note that for all $\lm\in\mathbb{C}$ and $t\ge0$, there hold
\begin{align*}
    |1-e^{-\frac{\lambda t}{\gamma}}|=&\left|1-e^{-\frac{\lambda_{\rm r}t}{\gamma}}\Big(\cos\frac{\lambda_{\rm i}t}{\gamma}-i\sin \frac{\lambda_{\rm i}t}{\gamma} \Big) \right|\\
    \le&\Big|1-e^{-\frac{\lambda_{\rm r}t}{\gamma}}\cos\frac{\lambda_{\rm i}t}{\gamma}\Big|+e^{-\frac{\lambda_{\rm r}t}{\gamma}}\Big|\sin \frac{\lambda_{\rm i}t}{\gamma}\Big|\\
    \le& \frac{|\lm_{\rm r}|t}{\gamma}e^{-\fr{\min\{\lm_{\rm r},0\}t}{\gamma}}+\fr12e^{-\fr{\lm_{\rm r}t}{\gamma}}\left(\fr{\lm_{\rm i}t}{\gamma}\right)^2+e^{-\fr{\lm_{\rm r}t}{\gamma}}\fr{|\lm_{\rm i}|t}{\gamma},
\end{align*}
and 
\begin{align*}
    |1-e^{-\lambda t}|
    \le& |\lm_{\rm r}|te^{-\min\{\lm_{\rm r},0\}t}+\fr12e^{-\lm_{\rm r}t}\left(\lm_{\rm i}t\right)^2+e^{-\lm_{\rm r}t}|\lm_{\rm i}|t.
\end{align*}
Then there exist two positive constants $C_{\sig,1}$ and $C_{\sig,2}$ depending on the function $\sig(\cdot)$, but independent of $\gamma$,  such that for all $\gamma\in(0,\gamma_0]$,  and $(\lm_{\rm r},\lm_i,M)\in\mathsf{R}_\dl$, we have
\begin{align}\label{ub-|F|}
    |{\rm F}(\lm,M)|
    \le\nn& \fr{|\lm_{\rm r}|}{\gamma}\int_0^\infty  t^2 |\hat{\sigma}(t)| dt+\fr{\lm_{\rm i}^2}{2\gamma^2}\int_0^\infty t^3|\hat{\sig}(t)|dt+\fr{|\lm_{\rm i}|}{\gamma}\int_0^\infty e^{-\fr{\lm_{\rm r}t}{\gamma}}t^2|\hat{\sig}(t)|dt\\
    \nn&+ \gamma^2 |m_\sigma|\left( |\lm_{\rm r}|\int_0^\infty  t^2 e^{-\frac{t^2}{2}} dt+\fr{\lm_{\rm i}^2}{2}\int_0^\infty  t^3 e^{-\frac{t^2}{2}} dt+|\lm_{\rm i}|\int_0^\infty  t^2 e^{-\frac{t^2}{2}} dt\right)\\
    \le\nn&2\dl \int_0^\infty  t^2 |\hat{\sigma}(t)| dt+\fr{\dl^2}{2}\int_0^\infty t^3|\hat{\sig}(t)|dt\\
    &+\gamma^2|m_\sig|\left(2\dl \gamma\int_0^\infty  t^2 e^{-\fr{t^2}{2}}dt+\fr{\dl^2\gamma^2}{2}\int_0^\infty t^3e^{-\fr{t^2}{2}}dt\right)\le C_{\sig,1}\dl,
\end{align}
and
\begin{align}\label{ub-|G|}
    |{\rm G}(\lm,M)|
    \le\nn&\fr{2(M_0+1)}{\gamma}\left(2\dl \int_0^\infty  t^2 |\hat{\sigma}(t)| dt+\fr{\dl^2}{2}\int_0^\infty t^3|\hat{\sig}(t)|dt\right)\\
    &+2\left(2\dl \gamma\int_0^\infty  t^2 e^{-\fr{t^2}{2}}dt+\fr{\dl^2\gamma^2}{2}\int_0^\infty t^3e^{-\fr{t^2}{2}}dt\right)\le \fr{C_{\sig,2}\dl}{\gamma},
\end{align}
where we have used
\begin{align}\label{uni-gamma2}\fr12\le 1+M\gamma^2m_\sig\le 2,\quad{\rm for \  all}\ \gamma\in(0,\gamma_0],\ M\in[M_0,M_0+1].
\end{align}
Moreover, noting that
\begin{align*}
    \frac{\partial N_\gamma(\lm,M)}{\partial\lambda}=&-\frac{1}{\gamma}\int_0^\infty e^{-\frac{\lambda t}{\gamma} } t^2 \hat{\sigma}(t) dt + \gamma^2 m_\sigma \int_0^\infty e^{-\lambda t} t^2 e^{-\frac{t^2}{2}} dt,\\
    \frac{\partial D_\gamma(\lm,M)}{\partial\lambda}=&-(1+M\gamma^2 m_\sigma) \left[ \int_0^\infty e^{-\lambda t}t^3  e^{-\frac{t^2}{2}} dt + \frac{M}{\gamma^2} \int_0^\infty e^{-\frac{\lambda t}{\gamma} } t^3  \hat{\sigma}(t) dt \right],
\end{align*}
repeating the above argument, we find that 
there exist two positive constants $C_{\sig,3}$ and $C_{\sig,4}$ depending on the function $\sig(\cdot)$, but independent of $\gamma$,  such that for all  $\gamma\in(0,\gamma_0]$,  and $(\lm_{\rm r},\lm_i,M)\in\mathsf{R}_\dl$, there hold
\begin{align}\label{ub:dN-dD}
\left|\frac{\partial N_\gamma(\lm,M)}{\partial\lambda}\right|\le \fr{C_{\sig,3}}{\gamma},\quad
\left|\frac{\partial D_\gamma(\lm,M)}{\partial\lambda}\right|\le \fr{C_{\sig,4}}{\gamma^2}.
\end{align}

For all $\gamma\in(0,\gamma_0]$ and $M\in[M_0,M_0+1]$, it is easy to verify that
\begin{gather}
    \label{uni-gamma1}\fr{3}{4}|B_1|\le|B_1|\pm\gamma^2m_\sig\le\fr{5}{4}|B_1|,\quad \fr{8}{5|B_1|}\le M_0\le\fr{8}{3|B_1|},\\
    \label{uni-gamma3}\fr{1}{\gamma}\fr{|B_2|}{|B_1|}\le \frac{M}{\gamma} |B_2|-\sqrt{\fr{\pi}{2}}\le \fr{1}{\gamma}\left(\fr{8|B_2|}{3|B_1|}+|B_2|\right).
\end{gather}
Now we choose $\dl$ sufficiently small  such that
\begin{align}\label{small-dl1}
    C_{\sig,1}\dl\le \fr14|B_1|,\quad C_{\sig,2}\dl\le \fr14\fr{|B_2|}{|B_1|}.
\end{align}
Then it follows from  \eqref{ub-|F|}--\eqref{small-dl1} that
\begin{align}\label{lb-D}
    \left| D_\gamma(\lm,M)\right|
    \ge\fr12 \left ( \frac{M}{\gamma} |B_2|-\sqrt{\fr{\pi}{2}}  \right)-|{\rm G}(\lm,M)|\ge \fr{1}{2\gamma}\fr{|B_2|}{|B_1|}-\fr{C_{\sig,2}}{\gamma}\dl
    \ge\fr{1}{4\gamma}\fr{|B_2|}{|B_1|},
\end{align}
and
\begin{align}\label{ub-D}
    \left| D_\gamma(\lm,M)\right|\le\nn&2 \left ( \frac{M}{\gamma} |B_2|-\sqrt{\fr{\pi}{2}}  \right)+|{\rm G}(\lm,M)|\le \fr{2}{\gamma}\left(\fr{8|B_2|}{3|B_1|}+|B_2|\right)+\fr{C_{\sig,2}}{\gamma}\dl\\
    \le&\fr{1}{\gamma}\left(\fr{67}{12}\fr{|B_2|}{|B_1|}+2|B_2|\right).
\end{align}
and
\begin{align}\label{ub-N}
|N_\gamma(\lm,M)|\le |B_1-\gamma^2m_\sig|+|{\rm F}(\lm,M)|\le \fr32|B_1|.
\end{align}
Then
\begin{align}\label{Lip}
    \left|\fr{\partial}{\partial\lambda}\left(\fr{N_\gamma(\lm,M)}{D_\gamma(\lm,M)}\right)\right|=\left| \fr{\pr_\lm N_\gamma D_\gamma-N_\gamma\pr_\lm D_\gamma}{D_\gamma^2}\right|\le \fr{4C_{\sig,3}|B_1|}{|B_2|}+\fr{24|B_1|^3C_{\sig,4}}{|B_2|^2}=:{\rm Lip}.
\end{align}

Next, we prove (ii). To ensure a lifespan uniform in $\gamma$, we employ the Banach fixed-point theorem here, rather than appealing directly to the Picard-Lindel\"of theorem. Let us define
\begin{align}
    X:=\left\{\lm\in C([M_0,M_0+\varepsilon];\mathbb{C}): 0\le \lm_{\rm r}\le \dl\gamma, |\lm_{\rm i}|\le \dl\gamma\right\}, 
\end{align}
where $\varepsilon>0$  is to be determined later. Consider a mapping $\mathcal{T}$  on $X$:
\begin{align*}
(\mathcal{T}\lm)(M)=\int_{M_0}^M\fr{N_\gamma(\lm(s),s)}{D_\gamma(\lm(s),s)}ds.
\end{align*}
From \eqref{lb-D} and \eqref{ub-N}, we have
\begin{align*}
    \left|\fr{N_\gamma(\lm(s),s)}{D_\gamma(\lm(s),s)}\right|\le \fr{6|B_1|^2}{|B_2|}\gamma.
\end{align*}
Thus,
\begin{align}\label{bd-into}
    |(\mathcal{T}\lm)(M)|\le \varepsilon\fr{6|B_1|^2}{|B_2|}\gamma \le \dl\gamma, 
\end{align}
provide 
\begin{align}\label{small-vareps1}
    \varepsilon\le\fr{|B_2|}{6|B_1|^2}\dl.
\end{align}
Moreover, note that
\begin{align*}
    \frak{Re}\fr{N_\gamma}{D_\gamma}
    = \fr{\big(B_1-\gamma^2m_\sig+\frak{Re}{\rm F}\big)\left((1+M\gamma^2 m_\sigma) \big(  \sqrt{\fr{\pi}{2}} + \frac{M}{\gamma} B_2 \big)+\frak{Re}{\rm G}\right)+\frak{Im}{\rm F}\frak{Im}{\rm G}}{|D_\gamma|^2}.
\end{align*}
Then we infer from \eqref{ub-|F|}--\eqref{uni-gamma2},  \eqref{uni-gamma1}--\eqref{small-dl1} that
\begin{align*}
\nn&\big(B_1-\gamma^2m_\sig+\frak{Re}{\rm F}\big)\left((1+M\gamma^2 m_\sigma) \Big(  \sqrt{\fr{\pi}{2}} + \frac{M}{\gamma} B_2 \Big)+\frak{Re}{\rm G}\right)+\frak{Im}{\rm F}\frak{Im}{\rm G}\\
\ge\nn&\big(|B_1|+\gamma^2m_\sig-\frak{Re}{\rm F}\big)\left((1+M\gamma^2 m_\sigma) \Big(  \frac{M}{\gamma} |B_2|-\sqrt{\fr{\pi}{2}}   \Big)-\frak{Re}{\rm G}\right)-\fr{C_{\sig,1}C_{\sig,2}\dl^2}{\gamma}\\
\ge&\fr{|B_1|}{2}\cdot\fr{|B_2|}{4\gamma|B_1|}-\fr{|B_2|}{16\gamma}=\fr{1}{16\gamma}|B_2|.
\end{align*}
Combining this with \eqref{ub-D}, we are led to
\begin{align}
    \frak{Re}\fr{N_\gamma}{D_\gamma}\ge\fr{|B_1|^2}{16(6+2|B_1|)^2|B_2|}\gamma=:\underline{c}\gamma.
\end{align}
Consequently, 
\begin{align}\label{re-nonegtive}
    \frak{Re}(\mathcal{T}\lm)(M)\ge \underline{c}\gamma (M-M_0)\ge0.
\end{align}
Then one can see from \eqref{bd-into} an \eqref{re-nonegtive} that $\mathcal{T}$ maps $X$ into $X$.

On the other hand, thanks to \eqref{Lip}, we find that
\begin{align}
    |(\mathcal{T}\lm_1)(M)-(\mathcal{T}\lm_2)(M)|\le  \fr12\|\lm_1-\lm_2\|_{C[M_0,M_0+\varepsilon]},
\end{align}
provided
\begin{align}\label{small-vareps2}
\varepsilon\le \fr{1}{2{\rm Lip}}.
\end{align}
Combining \eqref{small-vareps1} with \eqref{small-vareps2}, now we take 
\begin{align}
    \varepsilon:=\min\left\{\fr{|B_2|}{6|B_1|^2}\dl,\fr{1}{2{\rm Lip}}\right\}.
\end{align}
Then $\mathcal{T}: X\rightarrow X$ is a contraction mapping. Taking $\bar{C}:=\fr{6|B_1|^2}{|B_2|}$, then \eqref{bd:lm_r} and \eqref{bd:lm_i} follow immediately.

Finally, we prove (iii). Let $\lm(M), M\in[M_0,M_0+\varepsilon]$ be the solution obtained in (ii). Consider a function
\[
\Xi(M):=\Psi_\gamma(\lm(M),M),\quad M\in[M_0,M_0+\varepsilon].
\]
One easily deduces that
\begin{align*}
    \fr{d\,\Xi}{dM}=\fr{\pr\Psi_\gamma}{\pr\lm}\fr{d\lm}{dM}+\fr{\pr \Psi_\gamma}{\pr M}=0.
\end{align*}
Combining this with \eqref{Psi0} yields \eqref{equiv}. This completes the proof of Lemma \ref{lem-lm}.
\end{proof}

In the following, let us fix $M=M_0+\varepsilon$ determined in Lemma \ref{lem-lm}. Now we  solve the eigenvalue problem \eqref{eigen} with 
\begin{align}\label{lm}
    \lm=\lm(M_0+\varepsilon).
\end{align}
Using the eigenfunction given in \eqref{exp-eigen}, we then construct a solution to \eqref{LVP} with exponential growth. We state the result in the following lemma.
\begin{lem}\label{lem:f*}
Let 
\begin{align*}
    f^*(t,x,v):=e^{-t\mathbb{VP}_{f^0}}\frak{Re}\big(e^{ix} e_{\gamma,\lm}(v)\big).
\end{align*}
Then
\begin{align}\label{def:slvp}
    f^*(t,x,v)=\frak{Re}\big({e^{ix+\lm t}e_{\gamma,\lm}(v)}\big)
\end{align}
solves the linear Vlasolv-Poisson equation \eqref{LVP}. Moreover, it holds that
\begin{align}\label{L2:f*}
    \|\la v\ra^m f^*(t)\|_{L^2_{x,v}}=e^{\lm_{\rm r}t}\|\la v\ra^m \big(e^{ix} e_{\gamma,\lm}(v)\big)\|_{L^2_{x,v}}, \quad {\rm for \  all} \ m\in\R.
\end{align}
\end{lem}
\begin{proof}
The expression \eqref{def:slvp} follows directly from the fact that $\lambda$ and $e^{ix}e_{
\gamma,\lm}$ are the eigenvalue and associated eigenfunction of the operator $\mathbb{VP}_{f^0}$, and \eqref{L2:f*} follows from \eqref{def:slvp} immediately.
\end{proof}

To conclude this section, we provide upper and lower bounds for the eigenfunction $e_{\lm,\gamma}$ in the weighted $L^2_v$-space.

\begin{lem}\label{lem-eigen}
    Let $e_{\gamma,\lm}(v)$ be the eigenfunction given in \eqref{exp-eigen} with $\lm=\lm(M_0+\varepsilon)$ determined in Lemma \ref{lem-lm}, and $\gamma\in(0,\gamma_0]$. Then for any $m\ge0$, and $j\in\N$, there are two positive constants $\frak{c}<\frak{C}$, independent of $\gamma$, such that 
    \begin{align}\label{lb-eign}
        \|e_{\gamma,\lm}\|_{L^2_v}\ge \frak{c}\gamma^{-\fr12}|\hat{e}_{\gamma,\lm}(0)|,
    \end{align}
    and 
    \begin{align}\label{ub-eign}
        \|\la v\ra^m \pr_v^je_{\gamma,\lm}\|_{L^2_v}\le \frak{C}\gamma^{-j-\fr12}|\hat{e}_{\gamma,\lm}(0)|.
    \end{align}
\end{lem}
\begin{proof}
To prove \eqref{lb-eign}, for $\xi\le0$, recalling \eqref{def-f0} and \eqref{exp-eigen}, we write
\begin{align}\label{lb-eigen1}
    \hat{e}_{\gamma,\lm}(\xi)=\nn& e^{\lm\xi}\Big(1+\int_0^\xi\fr{\zeta e^{-\fr{\zeta^2}{2}}}{1+M\gamma^2m_\sig}d\zeta\Big)\hat{e}_{\gamma,\lm}(0)+ e^{\lm\xi}\int_0^\xi\fr{ (e^{-\lm\zeta}-1)\zeta e^{-\fr{\zeta^2}{2}}}{1+M\gamma^2m_\sig}d\zeta\hat{e}_{\gamma,\lm}(0)\\
    &+e^{\lm\xi}\int_0^\xi \zeta e^{-\lm \zeta} \fr{M\gamma^2\hat{\sig}(\gamma\zeta)}{1+M\gamma^2m_\sig}d\zeta \hat{e}_{\gamma,\lm}(0).
\end{align}
Clearly, by using \eqref{uni-gamma2}, we have 
\begin{align}
   \nn& e^{\lm_{\rm r}\xi}\bigg|\int_0^\xi\fr{ (e^{-\lm\zeta}-1)\zeta e^{-\fr{\zeta^2}{2}}}{1+M\gamma^2m_\sig}d\zeta\bigg|\le e^{\lm_{\rm r}\xi}\int_0^{-\xi}\fr{ |e^{\lm t}-1|t e^{-\fr{t^2}{2}}}{1+M\gamma^2m_\sig}dt\\
   \nn\le&\fr{1}{1+M\gamma^2m_\sig}e^{\lm_{\rm r}\xi}\left[\int_0^{-\xi}(\lm_{\rm r}+|\lm_{\rm i}|)e^{\lm_{\rm r}t}t^2e^{-\fr{t^2}{2}}dt+\fr{\lm_{\rm i}^2}{2}e^{\lm_{\rm r}t}t^3e^{-\fr{t^2}{2}}dt\right]\\
   \le\nn&4\dl\gamma\int_0^{\infty}(t^2+t^3)e^{-\fr{t^2}{2}}dt,
\end{align}
where $\dl$ is the constant determined  in Lemma \ref{lem-lm}. Then
\begin{align}\label{ub-eigen1}
     \bigg\|e^{\lm_{\rm r}\xi}\int_0^\xi\fr{ (e^{-\lm\zeta}-1)\zeta e^{-\fr{\zeta^2}{2}}}{1+M\gamma^2m_\sig}d\zeta\bigg\|_{L^2_\xi([-\fr{c}{\gamma},0])}\le 4\dl(c\gamma)^{\fr12}\int_0^{\infty}(t^2+t^3)e^{-\fr{t^2}{2}}dt.
\end{align}
Moreover,
\begin{align*}
    \bigg|e^{\lm\xi}\int_0^\xi \zeta e^{-\lm \zeta} \fr{M\gamma^2\hat{\sig}(\gamma\zeta)}{1+M\gamma^2m_\sig}d\zeta\bigg|=&\fr{M\gamma^2}{1+M\gamma^2m_\sig}e^{\lm_{\rm r}\xi}\bigg|\int_0^{-\xi} t e^{\lm t} \hat{\sig}(\gamma t)dt\bigg|
    \le M\|\hat{\sig}\|_{L^\infty}\gamma^2\xi^2.
\end{align*}
Then
\begin{align}\label{ub-eigen2}
    \bigg\|e^{\lm\xi}\int_0^\xi \zeta e^{-\lm \zeta} \fr{M\gamma^2\hat{\sig}(\gamma\zeta)}{1+M\gamma^2m_\sig}d\zeta\bigg\|_{L^2_{\xi}[-\fr{c}{\gamma},0]}\le \fr{M\|\hat{\sig}\|_{L^\infty}c^{\fr52}}{\sqrt{5}}\gamma^{-\fr12}.
\end{align}
Noting that $\displaystyle\int_0^\xi\fr{\zeta e^{-\fr{\zeta^2}{2}}}{1+M\gamma^2m_\sig}d\zeta\ge0$, taking $c$ sufficiently small but independent of $\gamma$ (one can further reduce $\gamma_0$ if necessary),  we then 
infer from \eqref{lb-eigen1}--\eqref{ub-eigen2}, \eqref{bd:lm_r}  and \eqref{uni-gamma1} that 
\begin{align}
    \nn&\|e_{\gamma,\lm}\|_{L^2_v}\ge \|\hat{e}_{\gamma,\lm}\|_{L^2_{\xi}([-\fr{c}{\gamma},0])}\\
    \ge\nn& \left[\fr{1}{(2\lm_{\rm r})^\fr12}\sqrt{1-e^{-\fr{2\lm_{r}c}{\gamma}}}-4\dl(c\gamma)^{\fr12}\int_0^{\infty}(t^2+t^3)e^{-\fr{t^2}{2}}dt-\fr{M\|\hat{\sig}\|_{L^\infty}c^{\fr52}}{\sqrt{5}}\gamma^{-\fr12}\right]|\hat{e}_{\gamma,\lm}(0)|\\
    \ge\nn&\left[\left(\sqrt{\fr{\underline{c}}{2\bar{C}}}-4\dl\gamma_0 \int_0^{\infty}(t^2+t^3)e^{-\fr{t^2}{2}}dt \right)c^{\fr12}-\big(\fr{8}{3\sqrt{5}|B_1|}+\fr{1}{\sqrt{5}})\|\hat{\sig}\|_{L^\infty}c^{\fr52}\right]\gamma^{-\fr12}|\hat{e}_{\gamma,\lm}(0)|\\ 
    \ge\nn&\frak{c}\gamma^{-\fr12}|\hat{e}_{\gamma,\lm}(0)|,
\end{align}
for some $\frak{c}>0$ independent of $\gamma$. This proves \eqref{lb-eign}.

Next, we turn to prove \eqref{ub-eign}. From \eqref{eq:hat-e}, we find that for $\ell\in\N$,
\[
\hat{e}_{\gamma,\lm}^{(\ell)}(\xi)=\lm^{\ell}\hat{e}_{\gamma,\lm}(\xi)+\hat{e}_{\gamma,\lm}(0)\sum_{l=0}^{\ell-1}\lm^{\ell-1-l}\big(\xi\widehat{f^0}(\xi)\big)^{(l)}.
\]
Then for $m,j\in\N$,
\begin{align}
\big(\xi^j\hat{e}_{\gamma,\lm}(\xi)\big)^{(m)}
 =\nn& \hat{e}_{\gamma,\lambda}(\xi) \sum_{\ell=0}^{m} \binom{m}{\ell} \lambda^\ell (\xi^j)^{(m-\ell)} \\
\nn&+ \hat{e}_{\gamma,\lambda}(0) \sum_{l=0}^{m-1} \left( \sum_{\ell=l+1}^{m} \binom{m}{\ell} \lambda^{\ell-1-l} (\xi^j)^{(m-\ell)} \right)  \big( \xi \widehat{f^0}(\xi) \big)^{(l)}=I_1(\xi)+I_2(\xi),
\end{align}
where the second term $I_2(\xi)$ on the right hand side vanishes when $m=0$. For $I_1(\xi)$, recalling \eqref{exp-eigen}, we find that there exists positive constant $C_{m,j}$ depending on $m$ and $j$, but independent of $\gamma$, such that
\begin{align}
    |I_1(\xi)|\le\nn& C_{m,j}\la \xi\ra^j|\hat{e}_{\gamma,\lm}(\xi)|
    \le\nn C_{m,j}\la \xi\ra^j\left|\int^{+\infty}_\xi\zeta e^{\lm(\xi-\zeta)}\widehat{f^0}(\zeta)d\zeta\right| |\hat{e}_{\gamma,\lm}(0)|{\bf 1}_{\xi\ge0}\\
\nn&+C_{m,j}\la \xi\ra^j\left(|e^{\lm\xi}|+\left|\int_0^\xi e^{\lm(\xi-\zeta)}\zeta\widehat{f^0}(\zeta)d\zeta\right|\right)|\hat{e}_{\gamma,\lm}(0)|{\bf1}_{\xi\le0}=I_{1,+}(\xi)+I_{1,-}(\xi).
\end{align}
By Young's inequality for convolutions, we are led to
\begin{align}\label{est:I+}
\|I_{1,+}\|_{L^2_\xi}\les\nn& \left\| \int_{\R}{\bf1}_{\xi-\zeta\le0} \la \xi-\zeta\ra^j e^{\lm_{\rm r}(\xi-\zeta)}|\zeta\widehat{f^0}(\zeta)|d\zeta\right\|_{L^2_{\xi\ge0}}|\hat{e}_{\gamma,\lm}(0)|\\
\nn&+\left\| \int_{\R}{\bf1}_{\xi-\zeta\le0}  e^{\lm_{\rm r}(\xi-\zeta)}\la \zeta\ra^{j+1}|\widehat{f^0}(\zeta)|d\zeta\right\|_{L^2_{\xi\ge0}}|\hat{e}_{\gamma,\lm}(0)|\\
\les\nn&\left(\int_{-\infty}^0\big|\la\eta\ra^je^{\lm_r\eta}\big|^2d\eta\right)^\fr12\|\zeta\widehat{f^0}(\zeta)\|_{L^1_\zeta}|\hat{e}_{\gamma,\lm}(0)|\\
\nn&+\left(\int_{-\infty}^0\big|e^{\lm_r\eta}\big|^2d\eta\right)^\fr12\|\la\zeta\ra^{j+1}\widehat{f^0}(\zeta)\|_{L^1_\zeta}|\hat{e}_{\gamma,\lm}(0)|\\
\les& \gamma^{-j-\fr12}|\hat{e}_{\gamma,\lm}(0)|,
\end{align}
and 
\begin{align}\label{est:I-}
    \|I_{1,-}\|_{L^2_\xi}\les\nn& \gamma^{-j-\fr12}|\hat{e}_{\gamma,\lm}(0)|+\left\| \int_{\R}{\bf1}_{\xi-\zeta\le0} \la \xi-\zeta\ra^j e^{\lm_{\rm r}(\xi-\zeta)}|\zeta\widehat{f^0}(\zeta)|d\zeta\right\|_{L^2_{\xi\le0}}|\hat{e}_{\gamma,\lm}(0)|\\
&+\left\| \int_{\R}{\bf1}_{\xi-\zeta\le0}  e^{\lm_{\rm r}(\xi-\zeta)}\la \zeta\ra^{j+1}|\widehat{f^0}(\zeta)|d\zeta\right\|_{L^2_{\xi\le0}}|\hat{e}_{\gamma,\lm}(0)|\les \gamma^{-j-\fr12}|\hat{e}_{\gamma,\lm}(0)|.
\end{align}

To bound $I_2(\xi)$, let us denote 
\begin{align}\label{def-sig1}
    \sig_1(\eta):=\eta\hat{\sig}(\eta).
\end{align}
Then we can write
\[
\xi\widehat{f^0}(\xi)=\fr{\xi e^{-\fr{\xi^2}{2}}+M\gamma \sig_1(\gamma\xi)}{1+M\gamma^2m_\sig},
\]
and
\[
\big( \xi \widehat{f^0}(\xi) \big)^{(l)}=\fr{(\xi e^{-\fr{\xi^2}{2}})^{(l)}+M\gamma^{l+1} \sig_1^{(l)}(\gamma\xi)}{1+M\gamma^2m_\sig}.
\]
Then there exists a positive constant $C_{m,j}$ depending on $m$ and $j$ but independent of $\gamma$, such that
\begin{align*}
    |I_2(\xi)|\le C_{m,j}\la \xi\ra^j\sum_{l=0}^{m-1}\left(|(\xi e^{-\fr{\xi^2}{2}})^{(l)}|+\gamma^{l+1}|\sig_1^{(l)}(\gamma\xi)|\right)|\hat{e}_{\gamma,\lambda}(0)|.
\end{align*}
Consequently,
\begin{align}\label{est:I2}
    \|I_2\|_{L^2_\xi}\le C\Big(1+\gamma^{-j+\fr12}\|\la \eta\ra^j \sig_1^{(l)}(\eta)\|_{L^2_\eta}\Big)|\hat{e}_{\gamma,\lambda}(0)|.
\end{align}
Then \eqref{ub-eign} follows from \eqref{est:I+}, \eqref{est:I-} and \eqref{est:I2} immediately.
\end{proof}

\subsection{Semigroup estimates}
In this section, we incorporate the Fokker-Planck collision term $\nu L[\cdot]$ and establish the corresponding space-time estimates of the linearized equation of \eqref{eq-tlf'}. To this end,  let  $\mathbb{S}(t)$ denote the semigroup associated with its   homogeneous linearized version, namely, $\mathbb{S}(t)f_{\rm in}$ solves
\begin{align}\label{linear-f}
    \pr_tf+v\pr_xf-\pr_x(-\pr_{xx})^{-1}\frak{I}[f]\pr_vf^0-\nu L[f]=0,\quad f|_{t=0}=f_{\rm in},
\end{align}
where $\frak{I}[f]$ is defined as in \eqref{def-I}.
Then the solution $\tl{f}$ to \eqref{eq-tlf'} can be given as
\begin{align}\label{exp:tlf}
\tl{f}(t,x,v)=\mathbb{S}(t)\tl{f}_{\rm in}+\int_0^t\mathbb{S}(t-\tau)[\frak{N}(\tau,x,v)]d\tau.
\end{align}
Integrating w.r.t. $v$ over $\R$, we are led to
\begin{align}\label{exp:tlrho}
\tl{\rho}(t,x)=\frak{I}\circ\mathbb{S}(t)\tl{f}_{\rm in}+\int_0^t\frak{I}\circ\mathbb{S}(t-\tau)[\frak{N}(\tau,x,v)]d\tau.
\end{align}

The appearance of the non-local term in \eqref{linear-f} leads us to solve $\frak{I}[f]$ first. For this purpose, taking Fourier transform of \eqref{linear-f}, we have
\begin{align*}
\pr_t\hat{f}_k(t,\xi)-k\pr_\xi\hat{f}_k+\nu \xi\pr_\xi\hat{f}_k(t,\xi)-ik^{-1}\hat{f}_k(t,0)\widehat{\pr_vf^0}(\xi)+\nu\xi^2\hat{f}_k(\xi)=0.
\end{align*}
To absorb the transport terms $ -k\pr_\xi\hat{f}_k+\nu \xi\pr_\xi\hat{f}_k(t,\xi)$, like \cite{bedrossian2017suppression},  we introduce the corresponding characteristic curve  
\begin{align*}
\tl{\eta}(t;k,\eta)=e^{\nu t}(\eta-kt^{\rm ap}),\quad{\rm with}\quad t^{\rm ap} =\frac{1-e^{-\nu t}}{\nu}.
\end{align*}
Here `{\rm ap}' stands for `approximate'. Now we define 
\[
\hat{g}_k(t,\eta):=\hat{f}_k(t,\tl{\eta}(t;k,\eta)),\quad{\rm or}\quad \hat{f}_k(t,\xi)=\hat{g}_k(t,e^{-\nu t}\xi+kt^{\rm ap}).
\]
Then $\hat{g}_k(t,\eta)$ solves
\[
\pr_t\hat{g}_k(t,\eta)+\nu \tl{\eta}(t;k,\eta)^2\hat{g}_k(t,\eta)=ik^{-1}\hat{f}_k(t,0)\widehat{\pr_vf^0}(\tl{\eta}(t;k,\eta)).
\]
Let
\begin{align*}
{\bf S}_k(t,\tau; \eta):=\exp\left(-\nu\int_\tau^t|\tl{\eta}(s;k,\eta)|^2ds\right)=\exp\left(-\nu\int_\tau^t\left|e^{\nu s}\eta-k\fr{e^{\nu s}-1}{\nu}\right|^2ds\right).
\end{align*}
Then
\begin{align}\label{exp:g_k}
    \hat{g}_k(t,\eta)={\bf S}_k(t,0;\eta)\hat{g}_k(0,\eta)+ik^{-1}\int_0^t{\bf S}_k(t,\tau;\eta)\hat{f}_k(\tau,0)\widehat{\pr_vf^0}(\tl{\eta}(\tau;k,\eta))d\tau.
\end{align}
Noting that $\frak{I}[f_k](t)=\hat{f}_k(t,0)=\hat{g}_k(t,kt^{\rm ap})$,
taking $\eta=kt^{\rm ap}$ in \eqref{exp:g_k}, we obtain
\begin{align}\label{eq:Ifk}
    \frak{I}[f_k](t)={\bf S}_k(t)(\hat{f}_{\rm in})_k(kt^{\rm ap})+ik^{-1}\int_0^t{\bf S}_k(t-\tau)\frak{I}[f_k](\tau)ik(t-\tau)^{\rm ap}\widehat{f^0}(k(t-\tau)^{\rm ap})d\tau,
\end{align}
where
\begin{align*}
{\bf S}_k(t-\tau)={\bf S}_k(t,\tau):=&{\bf S}_k\left(t,\tau;  kt^{\rm ap}\right)
=\exp\left(-\fr{|k|^2}{\nu }\left[t-\tau-2(t-\tau)^{\rm ap}+\fr{1-e^{-2\nu(t-\tau)}}{2\nu}\right]\right).
\end{align*}
In particular,
\begin{align*}
{\bf S}_k(t):={\bf S}_k\left(t,0; kt^{\rm ap}\right)=\exp\left\{-\fr{|k|^2}{\nu }\left(t-2t^{\rm ap}+\fr{1-e^{-2\nu t}}{2\nu}\right)\right\}.
\end{align*}

Since the background $f^0$ in \eqref{eq:Ifk} is a perturbation of the global Maxwellian, which may fail to satisfy the Penrose stability condition, we define a new unknown 
\[
\theta_k(t):=\frak{I}[f_k(t)]e^{-C_0\gamma t}
\]
to factor out the potential growth $e^{C_0\gamma t}$, where $C_0$ is a positive constant independent of $\gamma$, to be determined later. 
Then we infer from \eqref{eq:Ifk} that $\theta_k(t)$ satisfies the following  Volterra equation
\begin{align}\label{eq:theta_k}
    \theta_k(t)=H_k(t)e^{-C_0\gamma t}+\int_0^t e^{-C_0\gamma(t-\tau)}{\bf K}^{\nu}_k(t-\tau)\theta_k(\tau)d\tau,
\end{align} 
where
\begin{align}
    \mathbf{K}_k^\nu(t):=- {\bf S}_k(t)t^{\rm ap}\widehat{f^0}(kt^{\rm ap}),\quad {\rm and}\quad H_k(t):={\bf S}_k(t)(\hat{f}_{\rm in})_k(kt^{\rm ap}).
\end{align}
Taking Fourier-Laplace transform of \eqref{eq:theta_k}, we are led to
\begin{align}\label{exp:Ltheta_k}
    \mathcal{L}[\theta_k](\lm)=\mathcal{L}[H_k](\lambda+C_0\gamma)+\mathcal{L}[{\bf K}^\nu_k](\lm+C_0\gamma)\mathcal{L}[\theta_k](\lm),
\end{align}
with
\begin{align}
    \mathcal{L}[{\bf K}_k^{\nu}](\lm+C_0\gamma)=\nn&\int_0^\infty e^{-(\lm+C_0\gamma)t}{\bf K}_k^{\nu}(t)dt
    =-\int_0^\infty e^{-(\lm+C_0\gamma)t}{\bf S}_k(t)t^{\rm ap}\widehat{f^0}(kt^{\rm ap})dt.
\end{align}

To solve $\mathcal{L}[\theta_k](\lm)$, we establish the following lemma.
\begin{lem}\label{lem:Penrose}
There exist small constants $\underline{\nu},\underline{\gamma}, \kappa_0,\kappa_1>0$, and a large constant $C_0\ge1$, all independent of $\nu$ and $\gamma$, such that for all $\nu\in[0,\underline{\nu})$, $\gamma\in (\fr{\nu^{\fr13}}{\kappa_1},\underline{\gamma})$, there holds
\begin{align}\label{Penrose}
    \inf_{k\in\Z_*,\frak{Re}\lm\ge0}|1-\mathcal{L}[{\bf K}^{\nu}_k](\lm+C_0\gamma)|\ge \fr{\kappa_0}{8}.
\end{align}
\end{lem}
\begin{proof}
We first consider the case $\nu=0$. Now
\begin{align*}
    \mathcal{L}[{\bf K}_k^{0}](\lm+C_0\gamma)=\int_0^\infty e^{-(\lm+C_0\gamma)t}{\bf K}_k^{0}(t)dt=-\int_0^\infty e^{-(\lm+C_0\gamma)t}t\widehat{f^0}(kt)dt.
\end{align*}
Recalling the definition of $f^0$ in \eqref{def-f0}, we can split $\mathcal{L}[{\bf K}_k^{0}](\lm+C_0\gamma)$ into two parts:
\[
\mathcal{L}[{\bf K}_k^{0}](\lm+C_0\gamma)=\fr{1}{1+M\gamma^2m_\sig}\Big[\mathcal{K}^{\rm M}_k(\lm+C_0\gamma)+\mathcal{K}^{\rm e}_k(\lm+C_0\gamma)\Big],
\]
where $\mathcal{K}^{\rm M}_k(\lm+C_0\gamma)$ is the contribution from the global Maxwellian $\fr{1}{(2\pi)^{\fr12}}e^{-\fr{v^2}{2}}$:
\begin{align}
\mathcal{K}^{\rm M}_k(\lm+C_0\gamma)=-\fr{1}{k^2}\int_0^\infty e^{-\fr{\lm+C_0\gamma}{|k|}t}te^{-\fr{t^2}{2}}dt,
\end{align}
and $\mathcal{K}^{\rm e}_k(\lm+C_0\gamma)$ is the contribution from the error $\tl{\mu}(v)$:
\begin{align}
\mathcal{K}^{\rm e}_k(\lm+C_0\gamma)=-\fr{M}{k^2}\int_0^\infty e^{-\fr{\lm+C_0\gamma}{|k|\gamma}t}t\hat{\sig}(t)dt.
\end{align}
Note that the one-dimensional global Maxwellian satisfies the Penrose condition, namely, there exists a positive constant $\kappa_0$, such that 
\begin{align*}
{\rm inf}_{k\in\Z_*,\frak{Re}\lm\ge0}\big|1-\mathcal{K}_k^{\rm M}(\lm)\big|\ge \kappa_0,
\end{align*}
On the other hand, if $\frak{Re}\lm\ge0$, we find that
\begin{align*}
    |\mathcal{K}^{\rm e}_k(\lm+C_0\gamma)|\le \fr{M}{k^2}\int_0^{\infty}e^{-\fr{C_0}{|k|}t}\big|t\hat{\sig}(t)|dt\le \fr{M}{C_0|k|}\|t\hat{\sig}(t)\|_{L^\infty_t}.
\end{align*}
It follows from the above two estimates and \eqref{uni-gamma2} that, for all $\lm\in\mathbb{C}$ with $\frak{Re}\lm\ge0$ and $k\in\Z_*$,  there holds
\begin{align}\label{lb-K0}
    \big|1-\mathcal{L}[{\bf K}_k^{0}](\lm+C_0\gamma)\big|\nn=&\big|1-\fr{1}{1+M\gamma^2m_\sig}\mathcal{K}_k^{\rm M}(\lm+C_0\gamma)\big|-\fr{1}{1+M\gamma^2m_\sig}|\mathcal{K}_k^{\rm e}(\lm+C_0\gamma)|\\
    \ge\nn&\fr{{\rm inf}_{k\in\Z_*,\frak{Re}\lm \ge0}\big|1-\mathcal{K}_k^{\rm M}(\lm+C_0\gamma)\big|-M\gamma^2|m_\sig|}{1+M\gamma^2m_\sig}\\
    \nn&-\fr{1}{1+M\gamma^2m_\sig}\fr{M}{C_0}\|t\hat{\sig}(t)\|_{L^\infty_t}\\
    \ge\nn&\fr{\kappa_0-M\gamma^2|m_\sig|}{1+M\gamma^2m_\sig}-\fr{2M}{C_0}\|t\hat{\sig}(t)\|_{L^\infty_t}\\
    \ge&\fr{2\kappa_0}{5},
\end{align}
provided $\underline{\gamma}>0$ is so small that
\begin{align}\label{under-gamma}
    \underline{\gamma}^2M |m_\sigma|\le \min\Big\{\fr14, \fr{\kappa_0}{4}\Big\},\quad{\rm and}\quad \underline{\gamma}\le \gamma_0,
\end{align}
and $C_0$ is sufficiently large, such that
\begin{align}\label{C0}
\fr{2M\|t\hat{\sigma}(t)\|_{L^\infty_t}}{C_0}\le\fr{\kappa_0}{5}.
\end{align}

For the  case $\nu>0$,  we investigate the difference
\begin{align}\label{difference}
    \nn&\mathcal{L}[{\bf K}^\nu_k](\lm+C_0\gamma)-\mathcal{L}[{\bf K}^0_k](\lm+C_0\gamma)\\
    =\nn&-\int_0^\infty e^{-(\lm+C_0\gamma)t}{\bf S}_k(t)t^{\rm ap}\widehat{f^0}(kt^{\rm ap})dt+\int_0^{\infty}e^{-(\lm+C_0\gamma)t}t\widehat{f^0}(kt)dt\\
    =\nn&-\fr{1}{1+M\gamma^2m_\sig}\left[\int_0^\infty e^{-(\lm+C_0\gamma)t}{\bf S}_k(t)t^{\rm ap}\hat{\mu}(kt^{\rm ap})dt-\int_0^\infty e^{-(\lm+C_0\gamma)t}t\hat{\mu}(kt)dt\right]\\
    \nn&-\fr{M\gamma^2}{1+M\gamma^2m_\sig}\left[\int_0^\infty e^{-(\lm+C_0\gamma)t }{\bf S}_k(t)t^{\rm ap}\hat{\sig}(\gamma kt^{\rm ap})dt-\int_0^\infty e^{-(\lm+C_0\gamma)t }t\hat{\sig}(\gamma kt)dt\right]\\
    =&I^{\rm M}(\lm;\nu,\gamma,k)+I^{\rm e}(\lm;\nu,\gamma,k).
\end{align}
To bound $I^{\rm M}(\lm;\nu,\gamma,k)$, we write
\begin{align}\label{split-IM}
    I^{\rm M}(\lm;\nu,\gamma,k)
    =\nn&-\fr{1}{1+M\gamma^2m_\sig}\int_0^{\infty}e^{-(\lm+C_0\gamma)t}[{\bf S}_k(t)-1]t\hat{\mu}( kt)dt\\
    \nn&-\fr{1}{1+M\gamma^2m_\sig}\int_0^\infty e^{-(\lm+C_0\gamma)t}{\bf S}_k(t) (t^{\rm ap}-t)\hat{\mu}( kt)dt\\
    \nn&-\fr{1}{1+M\gamma^2m_\sig}\int_0^\infty e^{-(\lm+C_0\gamma)t}{\bf S}_k(t) t^{\rm ap}[\hat{\mu}( kt^{\rm ap})-\hat{\mu}( kt)]dt\\
    =&I^{\rm M}_1(\lm;\nu,\gamma,k)+I^{\rm M}_2(\lm;\nu,\gamma,k)+I^{\rm M}_3(\lm;\nu,\gamma,k).
\end{align}
Before proceeding any further, we list the following useful facts:
\begin{gather}
\label{S-prop1}
 0<{\bf S}_k(t)<\exp\left(-\underline{\dl}\fr{|k|^2}{\nu^2}\min\left\{(\nu t)^3, \nu t\right\}\right), \ {\rm with}\ \ \underline{\dl}>0 \  {\rm independent \  of } \ \nu,\\
 \label{S_k-1}   |{\bf S}_k(t)-1|\le \fr{|k|^2}{\nu }\left(t-2t^{\rm ap}+\fr{1-e^{-2\nu t}}{2\nu}\right)
    \le\fr13\nu k^2 t^3,\\
\label{Taylor1}
|t-t^{\rm ap}|=t-t^{\rm ap}\le\fr{1}{2}\nu t^2,\\
\label{ulb}
     t^{\rm ap}\begin{cases}\ge\fr{t}{e},\quad{\rm for\ \ all}\ \  0\le \nu t\le1,\\
     \le t,\quad\,{\rm for\ \ all}\ \  \nu t\ge0.
\end{cases} 
\end{gather}
One can find the proof of \eqref{S-prop1} in \cite{bedrossian2017suppression}, and \eqref{S_k-1}--\eqref{ulb} can be verified directly.

Using \eqref{S_k-1} and \eqref{uni-gamma2}, we have
\begin{align}\label{est:IM1}
    \big|I^{\rm M}_1(\lm;\nu,\gamma,k)\big|\les\fr{1}{1+M\gamma^2m_\sig}\int_0^\infty \nu k^2t^4e^{-\fr{|kt|^2}{2}}dt\les\fr{\nu}{|k|^3}.
\end{align}
By \eqref{Taylor1} and \eqref{uni-gamma2}, one easily deduces that
\begin{align}\label{est:IM2}
    \big|I^{\rm M}_2(\lm;\nu,\gamma,k)\big|\les\fr{1}{1+M\gamma^2m_\sig}\int_0^\infty \nu t^2e^{-\fr{|kt|^2}{2}}dt\les\fr{\nu}{|k|^3}.
\end{align}
Thanks to \eqref{S-prop1}, \eqref{Taylor1}, \eqref{ulb} and \eqref{uni-gamma2}, we find that
\begin{align}\label{est:IM3}
    \big|I^{\rm M}_3(\lm;\nu,\gamma,k)\big|\le\nn&\fr{1}{1+M\gamma^2m_\sig}\int_0^\infty {\bf S}_k(t)t^{\rm ap}e^{-\fr{|kt^{\rm ap}|^2}{2}}\big(1-e^{-\fr{|kt|^2-|kt^{\rm ap}|^2}{2}}\big)dt\\
    \le\nn&\fr{1}{2(1+M\gamma^2m_\sig)}\int_0^\infty {\bf S}_k(t)|k|^2(t-t^{\rm ap})(t+t^{\rm ap})t^{\rm ap}e^{-\fr{|kt^{\rm ap}|^2}{2}}dt\\
    \le\nn&\int_0^{\nu^{-1}} |k|^2\nu t^4e^{-\fr{|kt|^2}{2e^2}}dt+\int^\infty_{\nu^{-1}} e^{-\underline{\dl}\fr{|k|^2}{\nu} t}|k|^2\nu t^4dt\\
    \les&\fr{\nu}{|k|^3}+\fr{\nu^6}{|k|^8}.
\end{align}

Next we turn to bound $I^{\rm e}(\lm;\nu,\gamma,k)$. Similar to \eqref{split-IM}, we write
\begin{align*}
    I^{\rm e}(\lm;\nu,\gamma,k)
    =I^{\rm e}_1(\lm;\nu,\gamma,k)+I^{\rm e}_2(\lm;\nu,\gamma,k)+I^{\rm e}_3(\lm;\nu,\gamma,k).
\end{align*}
We first deal with $I^{\rm e}_1(\lm;\nu,\gamma,k)$. From  \eqref{S_k-1}, \eqref{uni-gamma2} and \eqref{uni-gamma1}, we infer that, for all $\lm\in\mathbb{C}$, with $\frak{Re}\lm\ge0$, and $\gamma>\fr{\nu^{\fr13}}{\kappa_1}$,
\begin{align}\label{est:Ie1}
|I^{\rm e}_1(\lm;\nu,\gamma,k)|
\le\nn&\fr13\int_0^{\infty} e^{-C_0\gamma t}\nu k^2t^3\fr{M\gamma^2t|\hat{\sig}(\gamma kt)|}{1+M\gamma^2m_\sig}dt
=\fr{\nu}{3\gamma^3|k|^3}\fr{M}{1+M\gamma^2m_\sig}\int_0^\infty e^{-\fr{C_0t}{|k|}}t^4|\hat{\sig}(t)|dt\\
\le&\fr{\nu}{3\gamma^3|k|^2}\fr{M}{C_0(1+M\gamma^2m_\sig)}\|t^4\hat{\sig}(t)\|_{L^\infty_t}\le \fr{2\kappa_1^3}{3}\big(\fr{8}{3|B_1|}+1\big)\|t^4\hat{\sig}(t)\|_{L^\infty_t}.
\end{align}
For $I^{\rm e}_2(\lm;\nu,\gamma,k)$, in view of \eqref{Taylor1}, \eqref{uni-gamma2} and \eqref{uni-gamma1}, we have
\begin{align}\label{est:Ie2}
    \big|I^{\rm e}_{2}(\lm;\nu,\gamma,k)\big|\le\nn&\fr12\int_0^\infty e^{-C_0\gamma t}\nu t^2\fr{M\gamma^2|\hat{\sig}(\gamma kt)|}{1+M\gamma^2m_\sig}dt
    \le \fr{\nu}{\gamma|k|^3}M\int_0^\infty e^{-\fr{C_0}{|k|}t}t^2|\hat{\sig}(t)|dt\\
    \le& \fr{\nu}{\gamma|k|^2}\fr{M}{C_0}\|t^2\hat{\sig}\|_{L^\infty_t}\le \kappa_1\nu^{\fr23}\big(\fr{8}{3|B_1|}+1\big)\|t^2\hat{\sig}\|_{L^\infty_t}.
\end{align}
As for $I^{\rm e}_{3}(\lm;\nu,\gamma,k)$, we can further split  into two parts as well:
\begin{align*}
    I^{\rm e}_{3}(\lm;\nu,\gamma,k)=\nn&-\left(\int_0^{\nu^{-1}}+\int_{\nu^{-1}}^{\infty}\right) e^{-(\lm+C_0\gamma)t }{\bf S}_k(t)t^{\rm ap}\left[\fr{M\gamma^2\hat{\sig}(\gamma kt^{\rm ap})}{1+M\gamma^2m_\sig}-\fr{M\gamma^2\hat{\sig}(\gamma kt)}{1+M\gamma^2m_\sig}\right]dt\\
    =&I^{\rm e;S}_{3}(\lm;\nu,\gamma,k)+I^{\rm e;L}_{3}(\lm;\nu,\gamma,k).
\end{align*}
Here `S' stands for `short time' and `L' stands for `long time'. In view of \eqref{Taylor1},  \eqref{ulb}, \eqref{uni-gamma2} and \eqref{uni-gamma1}, by the mean value theorem, we have
\begin{align}\label{est:Ie3S}
    \Big|I^{\rm e;S}_{3}(\lm;\nu,\gamma,k)\Big|
    \le\nn&\fr{M\gamma^3}{2(1+M\gamma^2m_\sig)}\int_0^\infty e^{-C_0\gamma t}e^{-\underline{\dl}|k|^2\nu t^3}\nu |k|t^3dt\|\pr_\xi\hat{\sig}\|_{L^\infty_\xi}\\
    \le& \fr{\nu^{\fr12}\gamma^{\fr12}M}{\underline{\dl}^{\fr12}C_0^{\fr52}}\|\pr_\xi\hat{\sig}\|_{L^\infty_\xi}\int_0^\infty e^{-t}t^\fr{3}{2}dt\le \fr34\sqrt{\pi}\fr{\nu^{\fr12}\gamma^{\fr12}}{\underline{\dl}^{\fr12}}\big(\fr{8}{3|B_1|}+1\big)\|\pr_\xi\hat{\sig}\|_{L^\infty_\xi}.
\end{align}
The rest term $I^{\rm e;L}_{3}$ can be bounded more directly as follows:
\begin{align}\label{est:Ie3L}
\Big|I^{\rm e;L}_{3}(\lm;\nu,\gamma,k)\Big|\le\nn&\fr{2M\gamma \|t\hat{\sig}\|_{L^\infty_t}}{|k|(1+M\gamma^2m_\sig)}\int_{\nu^{-1}}^\infty e^{-C_0\gamma t}dt=\fr{2M \|t\hat{\sig}\|_{L^\infty_t}}{C_0|k|(1+M\gamma^2m_\sig)}e^{-C_0\gamma \nu^{-1}}\\
\le&4\big(\fr{8}{3|B_1|}+1\big)\|t\hat{\sig}\|_{L^\infty_t}e^{-\fr{1}{\kappa_1\nu^{2/3}}}.
\end{align}
Collecting the estimates \eqref{est:IM1}--\eqref{est:Ie3L}, we infer from \eqref{lb-K0} and \eqref{difference} that there exist sufficiently small constants $\underline{\nu}, \underline{\gamma}, \kappa_1$ and a large constant $C_0$, with $\underline{\gamma}$  and $C_0$ satisfying \eqref{under-gamma} and \eqref{C0} respectively, all independent of$\nu$ and $\gamma$, such that \eqref{Penrose} holds.
\end{proof}

\begin{rem}
    If $\gamma=\nu^{\fr13-\epsilon_0}$ with $0<\eps_0<\fr13$, then \eqref{Penrose} holds for all $0<\nu\leq \min\big\{\underline{\nu},\underline{\gamma}^{\fr{3}{1-3\epsilon_0}}, \kappa_1^{\fr{1}{\eps_0}}\big\}$. 
\end{rem}

Now we are in a position to establish the following property for the semigroup $\mathbb{S}$, which will be used to estimate $\tl{f}$ and $\tl{\rho}$.
\begin{prop}\label{prop-semi-S}
Assume that $\gamma\gg\nu$. Then 
for any $m\ge0$, there exists a sufficiently large constant $C_0$, depending on the function $\sig(\cdot)$ and $m$, such that the following estimates hold:
\begin{align}\label{bd:semi2}
\big\|e^{-C_0\gamma t}|\pr_x|^\fr12\frak{I}\circ\mathbb{S}(t)f_{\rm in}\big\|_{L^2_tL^2_x}\le C\|(f_{\rm in})_{\ne}\|_{L^2_{x,v}},
\end{align}

\begin{align}\label{bd:semi1}
\big\|\la v\ra^m e^{-C_0\gamma t}\big(\mathbb{S}(t)f_{\rm in}\big)_{\ne}\big\|_{L^2_{x,v}}\nn&+\gamma^{\fr12}\big\|\la v\ra^m e^{-C_0\gamma t} \big(\mathbb{S}(t)f_{\rm in}\big)_{\ne}\big\|_{L^2_tL^2_{x,v}}\\
&+\nu^{\fr12}\|\la v\ra^m e^{-C_0\gamma t}\big(\nb_v\mathbb{S}(t)f_{\rm in}\big)_{\ne}\big\|_{L^2_tL^2_{x,v}}\le C\big\|\la v\ra^m (f_{\rm in})_{\ne}\big\|_{L^2_{x,v}},
\end{align}

\begin{align}\label{bd:semi1-0}
\big\|\la v\ra^m e^{-C_0\nu t}\big(\mathbb{S}(t)f_{\rm in}\big)_{0}\big\|_{L^2_{v}}\nn&+\nu^{\fr12}\big\|\la v\ra^m e^{-C_0\nu t} \big(\mathbb{S}(t)f_{\rm in}\big)_{0}\big\|_{L^2_tL^2_{v}}\\
&+\nu^{\fr12}\|\la v\ra^m e^{-C_0\nu t}\big(\nb_v\mathbb{S}(t)f_{\rm in}\big)_{0}\big\|_{L^2_tL^2_{v}}\le C\big\|\la v\ra^m (f_{\rm in})_{0}\big\|_{L^2_{v}}.
\end{align}

\end{prop}

\begin{proof}
By Lemma \ref{lem:Penrose}, we infer from \eqref{exp:Ltheta_k} that
\begin{align*}
\mathcal{L}[\theta_k](\lm)=\fr{\mathcal{L}[H_k](\lm+C_0\gamma)}{1-\mathcal{L}[{\bf K}^{\nu}_k](\lm+C_0\gamma)}\quad {\rm for \ \ all}\ \ \lm\in\mathbb{C},\ \ {\rm with}\ \ \frak{Re}\lm\ge0.
\end{align*}
In particular, we have
\begin{align*}
\fr{\mathcal{L}[H_k](i\lm_{\rm i}+C_0\gamma)}{1-\mathcal{L}[{\bf K}^{\nu}_k](i\lm_{\rm i}+C_0\gamma)}=\mathcal{L}[\theta_k](i\lm_{\rm i})=\mathcal{F}[{\bf 1}_{t\ge0}\theta_k](\lm_{\rm i})\quad {\rm for \ \ all}\ \ \lm_{\rm i}\in\mathbb{R}.
\end{align*}
Taking the change of variable 
\[
\tau=|k|t^{\rm ap},
\]
and using \eqref{S-prop1}, we have
\begin{align*}
    \int_0^\infty e^{-2C_0\gamma t}|H_k(t)|^2dt
    =&\fr{1}{|k|}\int_0^{\fr{|k|}{\nu}}e^{-2C_0\gamma t}|{\bf S}_k(t)(\hat{f}_{\rm in})_k(\fr{k}{|k|}\tau)|^2 e^{\nu t}d\tau\\
    \le&\fr{1}{|k|}\sup_{|\om|=1}\int_{-\infty}^{\infty}|(\hat{f}_{\rm in})_k(\om\tau)|^2d\tau=\fr{1}{|k|}\|(f_{\rm in})_k\|^2_{L^2_v}.
\end{align*}
Then by Planchrel's theorem and \eqref{Penrose}, one deduces that
\begin{align}\label{es:theta_k}
\|\theta_k\|_{L^2_{t\ge0}}\les\nn& \|\mathcal{L}[H_k](i\lm_{\rm i}+C_0\gamma)\|_{L^2_{\lm_{\rm i}}}=\|\mathcal{F}[{\bf1}_{t\ge0}e^{-C_0\gamma t}H_k(t)]\|_{L^2_{\lm_{\rm i}}}\\
=&\|e^{-C_0\gamma t}H_k(t)\|_{L^2_{t\ge0}}\le |k|^{-\fr12}\|(f_{\rm in})_k\|_{L^2_v}.
\end{align}
Then \eqref{bd:semi2} follows immediately.

Next we turn to prove \eqref{bd:semi1} and \eqref{bd:semi1-0}. 
Let $h:=e^{-C_0\gamma t}f$. Then noting that $f^0=\mu_1+\mu_2$, we infer from \eqref{linear-f} that $h$ solves
\begin{align}\label{eq:h}
\pr_th+C_0\gamma h+v\pr_xh-\pr_x(-\pr_{xx})^{-1}\frak{I}[h]\pr_v\mu_1-\nu L[h]=\pr_x(-\pr_{xx})^{-1}\frak{I}[h]\pr_v\mu_2. 
\end{align}
Taking Fourier transform in $x$ yields
\begin{align}\label{eq:h_k}
   \pr_th_k+C_0\gamma h_k+ikvh_k-ik^{-1}\frak{I}[h_k]\pr_v\mu_1-\nu L[h_k]=ik^{-1}\frak{I}[h_k]\pr_v\mu_2. 
\end{align}

Like the stability estimates in section \ref{sec:sta}, we need to construct a wave operator to absorb the main non-local term $-ik^{-1}\frak{I}[h_k]\pr_v\mu_1$. To this end, let us denote 
\[
P_{k,\gamma}(v):=1-\fr{\pi\mathcal{H}[\tl{v}\mu(\tl{v})](v)}{k^2(1+M\gamma^2m_\sig)}, \quad Q_{k,\gamma}(v):=\fr{\pi v\mu(v)}{k^2(1+M\gamma^2m_\sig)}=\fr{Q_k(v)}{1+M\gamma^2m_{\sig}},
\]
and
\[
W_{k,\gamma}(v):=P_{k,\gamma}^2(v)+Q_{k,\gamma}^2.
\]
where $Q_k(\cdot)$ is defined in \eqref{def-QW}.
Similar to \eqref{exp-P_k}, we can rewrite $P_{k,\gamma}(v)$ as
\begin{align*}
  P_{k,\gamma}(v)=&1+\fr{1}{2k^2(1+M\gamma^2m_\sig)}\int_0^\infty (e^{i\xi v}+e^{-i\xi v})\xi e^{-\fr{\xi^2}{2}}d\xi\\
  =&\fr{1}{1+M\gamma^2m_\sig}[P_k(v)+M\gamma^2m_\sig].
\end{align*}
Combining this with  \eqref{inf-Pk}, we find that
\begin{align}\label{inf-Pkgamma}
\inf_{k\in \Z_*}\inf_{v\in\R}P_{k,\gamma}(v)>0,\quad \rm for\ \ all\ \ sufficienltly\ \ small\ \ \gamma. 
\end{align}
As an analogue of the wave operator defined in \eqref{WO}, we now introduce the corresponding wave operator 
\begin{align}\label{WO-1}
\tl{\mathbb{D}}_k[h_k](v):=h_k(v)+\fr{Q_{k,\gamma}(v)}{P_{k,\gamma}(v)}\mathcal{H}[h_k](v),
\end{align}
such that
\begin{align}\label{WO1}
\tl{\mathbb{D}}_k\big[vh_k-k^{-2}\frak{I}[h_k]\pr_v\mu_1\big]=v\tl{\mathbb{D}}_k[h_k].
\end{align}
To  derive the inverse of the wave operator $\tl{\mathbb{D}}_k$, similar to Section \ref{sec:inverseWO}, we compute
\[
P_{k,\gamma}(v)+iQ_{k,\gamma}(v)=1-\pi\mathcal{H}\left[Q_{k,\gamma}\right](v)+iQ_{k,\gamma}(v),
\]
and
\begin{align*}
    \fr{h_k+i\mathcal{H}[h_k]}{P_{k,\gamma}+iQ_{k,\gamma}}=\fr{P_{k,\gamma}\tl{\mathbb{D}}_k[h_k]}{W_{k,\gamma}}+i\fr{P_{k,\gamma}\mathcal{H}[h_k]-h_kQ_{k,\gamma}}{W_{k,\gamma}}.
\end{align*}
Then
\begin{align}
    \mathcal{H}\left[\fr{P_{k,\gamma}\tl{\mathbb{D}}_k[h_k]}{W_{k,\gamma}}\right]=\fr{P_{k,\gamma}\mathcal{H}[h_k]-h_kQ_{k,\gamma}}{W_{k,\gamma}}.
\end{align}
Consequently,
\begin{align}\label{inverse:WO}
    h_k=\tl{\mathbb{D}}_k[h_k]-Q_{k,\gamma}\left(\mathcal{H}\left[\fr{P_{k,\gamma}\tl{\mathbb{D}}_k[h_k]}{W_{k,\gamma}}\right]+\fr{Q_{k,\gamma}}{W_{k,\gamma}}\tl{\mathbb{D}}_k[h_k]\right).
\end{align}

Thanks to \eqref{WO1}, applying $\tl{\mathbb{D}}_k$ to \eqref{eq:h_k} yields
\begin{align}\label{eq:Whk}
   \pr_t\tl{\mathbb{D}}_k[h_k]+C_0\gamma \tl{\mathbb{D}}_k[h_k]+ikv\tl{\mathbb{D}}_k[h_k]-\nu \tl{\mathbb{D}}_k\circ L[h_k]=ik^{-1}\frak{I}[h_k]\tl{\mathbb{D}}_k[\pr_v\mu_2]. 
\end{align}
Taking the $L^2_v$ inner product of \eqref{eq:Whk} with $\la v\ra^{2m}\tl{\mathbb{D}}_k[h_k]$, and using the fact
\begin{align*}
   -\nu \frak{Re}\Big\la L\big[\tl{\mathbb{D}}_k[h_k]\big],\la v\ra^{2m}\tl{\mathbb{D}}_k[h_k]\Big\ra=&\nu\big\|\la v\ra^m\nb_v \tl{\mathbb{D}}_k[h_k]\big\|_{L^2_v}^2+\big(m-\fr12\big)\nu\big\|\la v\ra^\ell \tl{\mathbb{D}}_k[h_k]\big\|_{L^2_v}^2\\
   &+2m(m-1)\nu \big\|\la v\ra^{m-2}\tl{\mathbb{D}}_k[h_k]\big\|_{L^2_v}^2-2m^2\nu\big\|\la v\ra^{m-1}\tl{\mathbb{D}}_k[h_k]\big\|_{L^2_v}^2, 
\end{align*}
we are led to
\begin{align}\label{inner-pro-ne}
    \nn&\fr12\fr{d}{dt}\big\|\la v\ra^m\tl{\mathbb{D}}_k[h_k]\big\|_{L^2_v}^2+C_0\gamma\big\|\la v\ra^m\tl{\mathbb{D}}_k[h_k]\big\|_{L^2_v}^2+\nu\big\|\la v\ra^m\nb_v \tl{\mathbb{D}}_k[h_k]\big\|_{L^2_v}^2\\
    =\nn&\frak{Re}\Big\la ik^{-1}\frak{I}[h_k]\tl{\mathbb{D}}_k[\pr_v\mu_2],\la v\ra^{2m}\tl{\mathbb{D}}_k[h_k]\Big\ra+\nu\frak{Re}\Big\la[\tl{\mathbb{D}}_k,L]h_k,\la v\ra^{2m}\tl{\mathbb{D}}_k[h_k]\Big\ra\\
    \nn&-\big(m-\fr12\big)\nu\big\|\la v\ra^m \tl{\mathbb{D}}_k[h_k]\big\|_{L^2_v}^2-2m(m-1)\nu \big\|\la v\ra^{m-2}\tl{\mathbb{D}}_k[h_k]\big\|_{L^2_v}^2\\
    &+2m^2\nu\big\|\la v\ra^{m-1}\tl{\mathbb{D}}_k[h_k]\big\|_{L^2_v}^2.
\end{align}
Recalling the definition of $\mu_2$ in \eqref{def:mu2}, and noting that $\frak{I}[h_k]=\theta_k(t)$,  using \eqref{inf-Pkgamma} and \eqref{WO-1}, we have
\begin{align}
    \frak{Re}\Big\la ik^{-1}\frak{I}[h_k]\tl{\mathbb{D}}_k[\pr_v\mu_2],\la v\ra^{2m}\tl{\mathbb{D}}_k[h_k]\Big\ra\le\nn&|k|^{-1}|\theta_k(t)|\big\|\la v\ra^m\tl{\mathbb{D}}_k[\pr_v\mu_2]\big\|_{L^2_v}\big\|\la v\ra^m\tl{\mathbb{D}}_k[h_k]\big\|_{L^2_v}\\
    \les\nn&|k|^{-1}|\theta_k(t)|\big\|\la v\ra^m\pr_v\mu_2\big\|_{L^2_v}\big\|\la v\ra^m\tl{\mathbb{D}}_k[h_k]\big\|_{L^2_v}\\
    \les\nn&\fr{\gamma^{\fr12}M\|\la v\ra^m \pr_v\sig(v)\|_{L^2_v}}{1+M\gamma^2m_\sig}|k|^{-1}|\theta(t)|\big\|\la v\ra^m\tl{\mathbb{D}}_k[h_k]\big\|_{L^2_v}.
\end{align}
Next, by \eqref{WO-1}, we find that
\begin{align*}
    [\tl{\mathbb{D}}_k,L]h_k=&\fr{Q_{k,\gamma}(v)}{P_{k,\gamma}(v)}\mathcal{H}\big[L[h_k]\big]-L\left[\fr{Q_{k,\gamma}}{P_{k,\gamma}}\mathcal{H}[h_k]\right]\\
    =&\fr{Q_{k,\gamma}(v)}{P_{k,\gamma}(v)}\mathcal{H}\big[\Delta_vh_k\big]-\Delta_v\left[\fr{Q_{k,\gamma}}{P_{k,\gamma}}\mathcal{H}[h_k]\right]\\
    &+\fr{Q_{k,\gamma}(v)}{P_{k,\gamma}(v)}\mathcal{H}\big[\nb_v\cdot(vh_k)\big]-\nb_v\cdot\left(v\left[\fr{Q_{k,\gamma}}{P_{k,\gamma}}\mathcal{H}[h_k]\right]\right)\\
    =&-\Dl_v\left(\fr{Q_{k,\gamma}}{P_{k,\gamma}}\right)\mathcal{H}[h_k]-2\nb_v\left(\fr{Q_{k,\gamma}}{P_{k,\gamma}}\right)\cdot\mathcal{H}[\nb_vh_k]-v\cdot\nb_v\left(\fr{Q_{k,\gamma}}{P_{k,\gamma}}\right)\mathcal{H}[h_k],
\end{align*}
where we have used the fact $[\mathcal{H},\Lm_1]=0$ with $\Lm_1$ defined in Proposition \ref{prop-com}.  Combining this with \eqref{est:Pk}, \eqref{inf-Pkgamma} and \eqref{inverse:WO}, we find that
\begin{align*}
    &\nu\frak{Re}\Big\la[\tl{\mathbb{D}}_k,L]h_k,\la v\ra^{2m}\tl{\mathbb{D}}_k[h_k]\Big\ra\\
    \les&\nu \|h_k\|_{H^1_v}\big\|\la v\ra^m \tl{\mathbb{D}}_k[h_k]\big\|_{L^2_v}\les\nu \|\tl{\mathbb{D}}_k[h_k]\|_{H^1_v}\big\|\la v\ra^m \tl{\mathbb{D}}_k[h_k]\big\|_{L^2_v}. 
\end{align*}
Collecting the above estimates, \eqref{inner-pro-ne} reduces to
\begin{align*}
    &\fr12\fr{d}{dt}\big\|\la v\ra^m\tl{\mathbb{D}}_k[h_k]\big\|_{L^2_v}^2+C_0\gamma\big\|\la v\ra^m\tl{\mathbb{D}}_k[h_k]\big\|_{L^2_v}^2+\nu\big\|\la v\ra^m\nb_v \tl{\mathbb{D}}_k[h_k]\big\|_{L^2_v}^2\\
    \le&\fr14\nu\big\|\la v\ra^m \nb_v\tl{\mathbb{D}}_k[h_k]\big\|_{L^2_v}^2+C|k|^{-2}|\theta_k(t)|^2+\gamma \big\|\la v\ra^m\tl{\mathbb{D}}_k[h_k]\big\|_{L^2_v}^2\\
    &+C\nu \big\|\la v\ra^m\tl{\mathbb{D}}_k[h_k]\big\|_{L^2_v}^2,
\end{align*}
where the constant $C$ depending on $m$ and the function $\sig(\cdot)$, but independent of $\nu$ and $\gamma$.
It follows from this and \eqref{es:theta_k} that,  for sufficiently large $C_0$,   and $\gamma\gg \nu$, there holds
\begin{align}\label{est:Whk}
    \nn&\big\|\la v\ra^m\tl{\mathbb{D}}_k[h_k](t)\big\|_{L^2_v}^2+C_0\gamma\big\|\la v\ra^m\tl{\mathbb{D}}_k[h_k]\big\|_{L^2_tL^2_v}^2+\nu\big\|\la v\ra^m\nb_v \tl{\mathbb{D}}_k[h_k]\big\|_{L^2_tL^2_v}^2\\
    \le&C|k|^{-2}|\theta_k(t)|_{L^2_{t\ge0}}^2+\big\|\la v\ra^m\tl{\mathbb{D}}_k[h_k](0)\big\|_{L^2_v}^2
    \les |k|^{-3}\|(f_{\rm in})_k\|^2_{L^2_v}+\|\la v\ra^{m}(f_{\rm in})_k\|^2_{L^2_v}.
\end{align}
On the other hand, we infer from \eqref{inverse:WO} and \eqref{est:Pk} that
\begin{align}\label{bd:hWh1}
    \big\|\la v\ra^m h_k\big\|_{L^2_v}\les\big\|\la v\ra^m \tl{\mathbb{D}}_k[h_k]\big\|_{L^2_v},
\end{align}
and 
\begin{align}\label{bd:hWh2}
    \big\|\la v\ra^m \nb_vh_k\big\|_{L^2_v}\les \big\|\la v\ra^m \nb_v\tl{\mathbb{D}}_k[h_k]\big\|_{L^2_v}+\big\|\tl{\mathbb{D}}_k[h_k]\big\|_{L^2_v}.
\end{align}
Then \eqref{bd:semi1} follows from \eqref{est:Whk}--\eqref{bd:hWh2} immediately.

We are left to prove \eqref{bd:semi1-0}. In fact, the zero mode $f_0$ satisfies the following equation:
\begin{align}\label{linear-f0}
\pr_tf_0-\nu L[f_0]=0,\quad f_0|_{t=0}=(f_{\rm in})_0.
\end{align}
Then we have
\begin{align*}
    \nn&\fr12\fr{d}{dt}\|\la v\ra^m (e^{-C_0\nu t}f_0)\|_{L^2_v}^2+C_0\nu \|\la v\ra^m (e^{-C_0\nu t}f_0)\|_{L^2_v}^2+\nu\big\|\la v\ra^m\nb_v (e^{-C_0\nu t}f_0)\big\|_{L^2_v}^2\\
    =\nn&-\big(m-\fr12\big)\nu\big\|\la v\ra^m (e^{-C_0\nu t}f_0)\big\|_{L^2_v}^2-2m(m-1)\nu \big\|\la v\ra^{m-2}h_0\big\|_{L^2_v}^2+2m^2\nu\big\|\la v\ra^{m-1}(e^{-C_0\nu t}f_0)\big\|_{L^2_v}^2.
\end{align*}
Then, for sufficiently large $C_0$ depending on $m$, one easily deduces that \eqref{bd:semi1-0} holds. This completes the proof of Proposition \ref{prop-semi-S}.
\end{proof}

To conclude this section, we show that the linearized Vlasov-Poisson-Fokker-Planck equation \eqref{linear-f} admits a solution with exponential growth.

\begin{prop}\label{prop:lb-S}
Let $c_0=\underline{c}\varepsilon$ with $\underline{c}$ and $\varepsilon$ specified in Lemma \ref{lem-lm}, $\epsilon_0$ an arbitrary fixed small positive constant, and $\gamma=\nu^{\fr13-\epsilon_0}$.
Then there exists a small constant $\dl_0\ll \epsilon_0$ independent of $\gamma$, such that
    for all $m\ge0$, and $t\le T_0=\dl_0 \nu^{-\fr13+\epsilon_0}\ln \nu^{-1}$,  there holds
    \begin{align}\label{lb:linear}
        \big\|\la v\ra^m \mathbb{S}(t)\big[\frak{Re}\big(e^{ix} e_{\gamma,\lm}(v)\big)\big]\big\|_{L^2_{x,v}}\ge \fr12e^{c_0\gamma t}\big\|\la v\ra^m \frak{Re}\big(e^{ix} e_{\gamma,\lm}(v)\big)\big\|_{L^2_{x,v}}.
    \end{align}
\end{prop}
\begin{proof}
    Let us denote ${\rm f}(t,x,v):=\mathbb{S}(t)\big[\frak{Re}\big(e^{ix} e_{\gamma,\lm}(v)\big)\big]-f^*(t,x,v)$ with $f^*(t,x,v)$ defined as in \eqref{def:slvp}. Then ${\rm f}(t,x,v)$ solves
\begin{align*}
    \pr_t{\rm f}+v\pr_x{\rm f}-\pr_x(-\pr_{xx})^{-1}\frak{I}[{\rm f}]\pr_vf^0-\nu L[{\rm f}]=\nu L[f^*],\quad
    {\rm f}|_{t=0}=0.
\end{align*}
Thus,
\[
{\rm f}(t,x,v)=\nu \int_0^t\mathbb{S}(t-\tau)\big[L[f^*](\tau)\big]d\tau.
\]
Recalling that $\lm$ is chosen as in \eqref{lm}, in view of \eqref{bd:lm_r}, we have
\[
c_0\gamma=\underline{c}\varepsilon\le\lm_{\rm r}\le \bar{C}\varepsilon \gamma\le C_0\gamma.
\]
Then by \eqref{bd:semi1} and Lemma \ref{lem-eigen}, we have
\begin{align}
    \nn&\|\la v\ra^m e^{-c_0\gamma t}{\rm f}_{\ne}(t)\|_{L^2_{x,v}}
    \le C \nu e^{(C_0-c_0)\gamma t}\int_0^{t}\|\la v\ra^m e^{-C_0\gamma \tau}L[f^*](\tau)\|_{L^2_{x,v}}d\tau\\
    \le\nn&C\nu e^{(C_0-c_0)\gamma t}\gamma^{-2}t\|f^*_{\rm in}\|_{L^2_v}\le C\dl_0 e^{(C_0-c_0)\dl_0\ln\nu^{-1}}\nu^{3\epsilon_0}\ln\nu^{-1}\|f^*_{\rm in}\|_{L^2_v}\\
    \le\nn&C\dl_0\nu^{3\eps_0-(C_0-c_0)\dl_0}\ln\nu^{-1}\|f^*_{\rm in}\|_{L^2_v}\le \fr12\|f^*_{\rm in}\|_{L^2_v},
\end{align}
provided $\dl_0$ is so small that
\[
C\dl_0\nu^{3\eps_0-(C_0-c_0)\dl_0}\ln\nu^{-1}\le\fr12.
\]
Consequently, 
\begin{align*}
    \|\la v\ra^m e^{-c_0\gamma t}\mathbb{S}(t)f^*_{\rm in}\|_{L^2_v}\ge \|\la v\ra^\ell e^{-c_0\gamma t}f^*\|_{L^2_{x,v}} - \|\la v\ra^m e^{-c_0\gamma t}{\rm f}_{\ne}\|_{L^2_{x,v}}\ge \fr12\|\la v\ra^m f_{\rm in}^*\|_{L^2_{x,v}}.
\end{align*}
Thus, \eqref{lb:linear} holds.
\end{proof}

\subsection{Nonlinear estimates and the proof of Theorem \ref{Thm: main2}}The proof consists of two steps. In the following, we take $T_0=\dl_0\nu^{-\fr13+\eps_0}\ln \nu^{-1}$ and $\gamma=\nu^{\fr13-\eps_0}$ as in Theorem \ref{Thm: main2}.\par

{\it Step (I): the upper bounds}.
We first give an upper bound of the solution $\tl{f}(t,x,v)$ to \eqref{eq-tlf'} with initial data $\tl{f}_{\rm in}$ of size $O(\nu^{\beta}), \beta>\fr12$. In fact, thanks to \eqref{bd:semi1}, in view of \eqref{exp:tlf}, and using Minkowski's inequality, we have
\begin{align}\label{es:up-ne}
    \nn&\big\|\la v\ra^m e^{-C_0\gamma t}\tl{f}_{\ne}\big\|_{L^\infty_tL^2_{x,v}}+\gamma^{\fr12}\big\|\la v\ra^m e^{-C_0\gamma t}\tl{f}_{\ne}\big\|_{L^2_tL^2_{x,v}}+\nu^{\fr12}\big\|\la v\ra^m e^{-C_0\gamma t}\nb_v\tl{f}_{\ne}\big\|_{L^2_tL^2_{x,v}}\\
    \les&\big\|\la v\ra^m (\tl{f}_{\rm in})_{\ne}\big\|_{L^2_{x,v}}+\int_0^{T_0} \|\la v\ra^m e^{-C_0\gamma\tau}\frak{N}_{\ne}(\tau) \|_{L^2_{x,v}}d\tau.
\end{align}
Similarly,  for the zero mode $\tl{f}_0$, using \eqref{bd:semi1-0} and Minkowski's inequality, we find that
\begin{align}\label{es:up-0}
    \nn&\|\la v\ra^m e^{-C_0\nu t}\tl{f}_0\|_{L^\infty_t L^2_v}+\nu^{\fr12}\|\la v\ra^{m} e^{-C_0\nu t}\nb_v\tl{f}_0\|_{L^2_t L^2_v}\\
    \les&\big\|\la v\ra^m (\tl{f}_{\rm in})_{0}\big\|_{L^2_{v}}+\int_0^{T_0} \|\la v\ra^m e^{-C_0\nu\tau}\frak{N}_{0}(\tau) \|_{L^2_{v}}d\tau.
\end{align}
Moreover, by \eqref{exp:tlrho}, \eqref{bd:semi2}  and Minkowski's inequality, we have
\begin{align}\label{es:up-rho}
    \big\|e^{-C_0\gamma t}|\pr_x|^{\fr12}\tl{\rho}\big\|_{L^2_tL^2_x}\le\nn&\big\|e^{-C_0\gamma t}|\pr_x|^{\fr12}\frak{I}\circ\mathbb{S}(t)\tl{f}_{\rm in}\big\|_{L^2_tL^2_x}\\
    \nn&+\left\|\int_0^te^{-C_0\gamma (t-\tau)}|\pr_x|^{\fr12}\frak{I}\circ\mathbb{S}(t-\tau)[e^{-C_0\gamma\tau}\frak{N}(\tau,x,v)]d\tau\right\|_{L^2_t L^2_x}\\
    \les&\|(\tl{f}_{\rm in})_{\ne}\|_{L^2_{x,v}}+\int_0^{T_0}\|e^{-C_0\gamma\tau}\frak{N}_{\ne}(\tau)\|_{L^2_{x,v}}d\tau.
\end{align}

For $t\le T_0=\dl_0\nu^{-\fr13+\eps_0}\ln \nu^{-1}$ and $\gamma=\nu^{\fr13-\eps_0}$, there holds
\begin{align}
    e^{2C_0\gamma t}\le e^{2C_0\dl_0\ln\nu^{-1}}=\nu^{-2C_0\dl_0}.
\end{align}
Then by the one dimensional  Sobolev embedding $H^1(\T)\hookrightarrow L^\infty(\T)$,  one deduces that
\begin{align}\label{es:NL-ne1}
    \|\la v\ra^m e^{-C_0\gamma \tau}(\tl{E}\pr_v\tl{f})_{\ne}\|_{L^1_tL^2_{x,v}}\les\nn& \|e^{-C_0\gamma t}|\pr_x|^{\fr12}\tl{\rho}\|_{L^2_tL^2_x}\bigg(\| e^{C_0\gamma t}\la v\ra^m e^{-C_0\gamma t}\nb_v\tl{f}_{\ne}\|_{L^2_tL^2_{x,v}}\\
    \nn&+\|e^{C_0\nu t}\la v\ra^m e^{-C_0\nu t}\nb_v\tl{f}_{0}\|_{L^2_tL^2_{x,v}}\bigg)\\
    \les&\nu^{-(\fr12+2C_0\dl_0)}\|e^{-C_0\gamma t}|\pr_x|^{\fr12}\tl{\rho}\|_{L^2_tL^2_x}\bigg(\nu^{\fr12}\|\la v\ra^m e^{-C_0\gamma t}\nb_v\tl{f}_{\ne}\|_{L^2_tL^2_{x,v}}\\
    \nn&+\nu^{\fr12}\|\la v\ra^m e^{-C_0\nu t}\nb_v\tl{f}_{0}\|_{L^2_tL^2_{x,v}}\bigg),
\end{align}
and
\begin{align}\label{es:NL-0}
    \|\la v\ra^m e^{-C_0\nu \tau}(\tl{E}\pr_v\tl{f})_{0}\|_{L^1_tL^2_{x,v}}\les\nn& \big\|e^{ C_0\gamma t-C_0\nu t}e^{-C_0\gamma t}|\pr_x|^{\fr12}\tl{\rho}\big\|_{L^2_tL^2_x}\big\| e^{C_0\gamma t}\la v\ra^m e^{-C_0\gamma t}\pr_v\tl{f}_{\ne}\big\|_{L^2_tL^2_{x,v}}\\
    \les&\nu^{-(\fr12+2C_0\dl_0)}\big\|e^{-C_0\gamma t}|\pr_x|^{\fr12}\tl{\rho}\big\|_{L^2_tL^2_x}
    \Big(\nu^{\fr12}\big\|\la v\ra^m e^{-C_0\gamma t}\pr_v\tl{f}_{\ne}\big\|_{L^2_tL^2_{x,v}}\Big).
\end{align}
Finally, by virtue of \eqref{C.1} in Lemma \ref{lem-fe},
\begin{align}\label{es:NL-ne2}
 \|\la v\ra^m e^{-C_0\gamma\tau}\tl{E}\pr_v(f^{\rm e}-\mu_2)\|_{L^1_tL^2_{x,v}}   \le\nn&\big\|e^{-C_0\gamma t}|\pr_x|^{-1}\tl{\rho}\big\|_{L^2_tL^2_x}\big\|\la v\ra^m \pr_v(f^{\rm e}-\mu_2)\big\|_{L^2_tL^2_v}\\
\les&\dl_0\nu^{3\eps_0}(\ln\nu^{-1})^{\fr32}\big\|e^{-C_0\gamma t}|\pr_x|^{-1}\tl{\rho}\big\|_{L^2_tL^2_x}.
\end{align}
Recalling the definition of $\frak{N}(t,x,v)$ in \eqref{def:NL-N}, substituting \eqref{es:NL-ne1}--\eqref{es:NL-ne2} into \eqref{es:up-ne}--\eqref{es:up-rho}, for $\dl_0$ sufficiently small, we are led to
\begin{align}
    \nn&\big\|\la v\ra^m e^{-C_0\gamma t}\tl{f}_{\ne}\big\|_{L^\infty_tL^2_{x,v}}+\gamma^{\fr12}\big\|\la v\ra^m e^{-C_0\gamma t}\tl{f}_{\ne}\big\|_{L^2_tL^2_{x,v}}+\nu^{\fr12}\big\|\la v\ra^m e^{-C_0\gamma t}\nb_v\tl{f}_{\ne}\big\|_{L^2_tL^2_{x,v}}\\
    \nn&+\|\la v\ra^m e^{-C_0\nu t}\tl{f}_0\|_{L^\infty_t L^2_v}+\nu^{\fr12}\|\la v\ra^{m} e^{-C_0\nu t}\nb_v\tl{f}_0\|_{L^2_t L^2_v}+\big\|e^{-C_0\gamma t}|\pr_x|^{\fr12}\tl{\rho}\big\|_{L^2_tL^2_x}\\
    \le\nn&C\big\|\la v\ra^m \tl{f}_{\rm in}\big\|_{L^2_{x,v}}+C\nu^{-(\fr12+2C_0\dl_0)}\|e^{-C_0\gamma t}|\pr_x|^{\fr12}\tl{\rho}\|_{L^2_tL^2_x}\bigg(\nu^{\fr12}\|\la v\ra^m e^{-C_0\gamma t}\nb_v\tl{f}_{\ne}\|_{L^2_tL^2_{x,v}}\\
    \nn&+\nu^{\fr12}\|\la v\ra^m e^{-C_0\nu t}\nb_v\tl{f}_{0}\|_{L^2_tL^2_{x,v}}\bigg).
\end{align}
Then by the standard bootstrap argument, we conclude that
\begin{align}\label{es:UP}
    \nn&\big\|\la v\ra^m e^{-C_0\gamma t}\tl{f}_{\ne}\big\|_{L^\infty_tL^2_{x,v}}+\gamma^{\fr12}\big\|\la v\ra^m e^{-C_0\gamma t}\tl{f}_{\ne}\big\|_{L^2_tL^2_{x,v}}\\
    \nn&+\nu^{\fr12}\big\|\la v\ra^m e^{-C_0\gamma t}\nb_v\tl{f}_{\ne}\big\|_{L^2_tL^2_{x,v}}+\big\|\la v\ra^m e^{-C_0\nu t}\tl{f}_0\big\|_{L^\infty_t L^2_v}\\
    &+\nu^{\fr12}\big\|\la v\ra^{m} e^{-C_0\nu t}\nb_v\tl{f}_0\big\|_{L^2_t L^2_v}+\big\|e^{-C_0\gamma t}|\pr_x|^{\fr12}\tl{\rho}\big\|_{L^2_tL^2_x}
    \le C\big\|\la v\ra^m \tl{f}_{\rm in}\big\|_{L^2_{x,v}},
\end{align}
provided that
\begin{align}\label{up-small}
    \big\|\la v\ra^m \tl{f}_{\rm in}\big\|_{L^2_{x,v}}\le \eps\nu^\beta,\quad {\rm with}\ \ \beta>\fr12,
\end{align}
and $\eps>0$ sufficiently small independent of  $\nu$.

{\it Step (II): the lower bounds}. Now we bound the solution $\tl{f}$ to \eqref{eq-tlf'} with initial data 
\[
f_{\rm in}^*:=\gamma^{\fr12}\nu^{\beta}\mathfrak{Re}\Big(e^{i x}\frac{e_{\gamma,\lambda}(v)}{\hat{e}_{\gamma, \lambda}(0)}\Big)
\]
from below.  To this end, recalling the expression of $\tl{f}$ in \eqref{exp:tlf}, for $c_0$ determined in Proposition \ref{prop:lb-S}, we have
\begin{align*}
    \left\|\la v\ra^m e^{-c_0\gamma t}\tl{f}_{\ne}(t)\right\|_{L^2_{x,v}}\ge& \left\|\la v\ra^m e^{-c_0\gamma t}\mathbb{S}(t)f_{\rm in}^*\right\|_{L^2_{x,v}}\\
    &-e^{(C_0-c_0)\gamma t}\left\|\la v\ra^m \int_0^te^{-C_0(t-\tau)}\big(\mathbb{S}(t-\tau)\big[e^{-C_0\tau}\frak{N}(\tau,x,v)\big]\big)_{\ne}d\tau\right\|_{L^2_{x,v}}\\
    =&\mathsf{L}(t)-\mathsf{NL}(t).
\end{align*}
Clearly, $\mathsf{L}(t)$ can be bounded from below by \eqref{lb:linear}. For $\mathsf{NL}(t)$, one can see from \eqref{ub-eign} that the initial data $f^*_{\rm in}$ given above satisfies \eqref{up-small}. Then
in view of \eqref{bd:semi1}, using Minkowski's inequality, \eqref{es:NL-ne1}, \eqref{es:NL-ne2} and \eqref{es:UP}, we have
\begin{align*}
    \mathsf{NL}(t)\le&\nn e^{(C_0-c_0)\gamma T_0}\int_0^{T_0}\|\la v\ra^m e^{-C_0\gamma\tau}\frak{N}_{\ne}(\tau)\|_{L^2_{x,v}}d\tau\\
    \le\nn& C\nu^{-\dl_0(C_0-c_0)}\Big[\nu^{\beta-(\fr12+2C_0\dl_0)}+\dl_0\nu^{3\epsilon_0}(\ln\nu^{-1})^{\fr32}\Big]\|\la v\ra^m f^*_{\rm in}\|_{L^2_{x,v}}\\
    \le&\fr14\|\la v\ra^m f^*_{\rm in}\|_{L^2_{x,v}}
\end{align*} 
Then for $t\le T_0$, we are led to
\begin{align}\label{es:LOW}
    \left\|\la v\ra^m e^{-c_0\gamma t}\tl{f}_{\ne}(t)\right\|_{L^2_{x,v}}\ge\fr14\|\la v\ra^m f^*_{\rm in}\|_{L^2_{x,v}}.
\end{align}
Recalling  that $g(t,x,v)=\tl{f}(t,x,v)+\tl{f}^0(t,v)-\mu(v)$, and $g_{\rm in}=f^*_{\rm in}+f^0(v)-\mu(v)$,
Theorem \ref{Thm: main2} follows from \eqref{es:UP} and \eqref{es:LOW} directly.

\begin{appendix}

\section{Some elementary calculations}\label{app:calcu}
In this section, we study the coefficients appearing in the definition of the wave operator. 
\begin{lem}
Let $\al,\beta\in\N^n$. Denote
\[
\Phi_{1,k}(z)=1, \quad\Phi_{2,k}(z)=\frac{1}{P_k(z)},\quad \Phi_{3,k}(z)=\frac{Q_k(z)}{W_k(z)},\quad \Phi_{4,k}(z)=\frac{P_k(z)}{W_k(z)},
\]
and
\begin{align*}
    A_k^\al(v)=\frac{k\cdot v}{|k|}v^\al \mu(v), \quad w_{j,k}^{\alpha,\beta}(v)=&\partial_v^\beta\left(A^\alpha_k(v)\Phi_{j,k}\Big(\frac{k\cdot v}{|k|}\Big)\right),\quad j=1, 2, 3,4.
\end{align*}
In particular, we use $A_k(v)$ to denote $A^{\al}_k(v)$ when $|\al|=0$, and use $w_{j,k}(v)$ to denote $w^{\al,\beta}_{j,k}(v)$ when $|\al|=|\beta|=0$.
Then there exist positive constants $C_{\al,\beta}\ge1$ and $c_{\al,\beta}\in(0,\fr12)$, such that
\begin{align}\label{bd-w}
     |w_{j,k}^{\al,\beta}(v)|\le C_{\al,\beta}e^{-c_{\al,\beta}|v|^2}.
\end{align}
\end{lem}
\begin{proof}
It is easy to verify that $P_k(z)$ is smooth, and there exists a positive constant $M_m$ independent of $k$, such that
    \[
    \sup_{k\in\Z^n_*}\sup_{z\in\R}|P_k^{(m)}(z)|\le M_m, \quad m\in \N.
    \] 
Moreover, recalling the definition of $W_k(z)$ in \eqref{def-QW}, thanks to the uniform positive lower bound of $P_k(z)$ in \eqref{inf-Pk}, one deduces that \eqref{inf-Pk} still holds with $P_k$ replaced by $W_k$. Consequently, using Leibnitz's law and Fa\`{a} di Bruno formula, we find that for any integer \(m\ge0\), there exists constant \(C_m>0\) (independent of \(k\) for \(|k|\ge1\)) such that, with \(z\in\R\),
\begin{align}\label{est:Phi}
    \Big|\Phi^{(m)}_{j,k}(z)\Big| &\le C_m, \quad j=1,2,3,4.
\end{align}
On the other hand, for any $\tl{\beta}\in\N^n$, there exist positive constants $C_{\al,\tl{\beta}}$ and $M_{|\al|+|\tl\beta|}$, such that
\begin{align}\label{bd-Aalpha}
    |\partial_v^{\tl{\beta}} A_k^\al(v)|\le C_{\al,\tl{\beta}}\la v\ra^{M_{|\al|+|\tl\beta|}} e^{-\frac{|v|^2}{2}}.
\end{align}
Noting that
\begin{align*}
    w^{\alpha,\beta}_{j,k}(v)
    =&\sum_{\beta'\le\beta}C_\beta^{\beta'}\partial_v^{\beta-\beta'}A_k^\al(v)\Phi^{(|\beta'|)}_{j,k}\Big(\frac{k\cdot v}{|k|}\Big)\Big(\frac{k}{|k|}\Big)^{\beta'},
\end{align*}
then \eqref{bd-w} follows from \eqref{est:Phi} and \eqref{bd-Aalpha} immediately.
\end{proof}
In the following lemma, we give more precise upper bound estimates for $P_k$ and its derivatives, and by means of this, we show that $K_k$ lies in $H^s(\R)$ for any $s\ge0$.
\begin{lem}
For any $m\in\N$, there holds
    \begin{align}\label{est:Pk}
        |\partial_z^{(m)}(P_k(z)-1)|\le \frac{C_m}{|k|^2(1+|z|)^{m+2}}, 
    \end{align}
and for any fixed $s\ge0$,
    \begin{align}\label{est:Wk}
        \|K_k\|_{H^s}\le \fr{C_s}{|k|^2},
    \end{align}
    where the constants $C_m$ and $C_s$ depend on $m$ and $s$, respectively, but independent of $k$.
\end{lem}
\begin{proof}
    Let 
    \[
    I(z)=\int_0^\infty ye^{-\fr{y^2}{2}}\cos(zy)dy,\quad z\in\R.
    \]
    Then $P_k(z)=1+\fr{1}{|k|^2}I(z)$.
    Integrating by parts twice, we get
    \begin{align*}
        I(z)=-\frac{1}{z^2}-\frac{1}{z^2}\int_0^\infty (y^3-3y)e^{-\frac{y^2}{2}}\cos(zy)dy.
    \end{align*}
    Repeating this argument, we get \eqref{est:Pk}. Recalling the definition of $K_k(z)$ in \eqref{def-K}, we write
    \[
    K_k(z)=-\fr{(P_k-1)^2+Q_k^2}{P^2_k+Q_k^2}-\fr{P_k-1}{P_k^2+Q_k^2}.
    \]
    Then, in view of \eqref{inf-Pk} and \eqref{est:Pk}, using Leibniz's law and Fa\`{a} di Bruno formula, we arrive at  \eqref{est:Wk}.
\end{proof}

\section{A product estimate}
Recall the definition of the Fourier multipliers $\rm A$ and $\rm B$ defined in section \ref{sec: notation}. In this section, we discuss some basic properties of those operators. 
\begin{lem}\label{lem-produ}
For any $s\ge0$, it holds that
\begin{align}\label{log-produ}
\sum_{k\in\Z^3}\sum_{l\in\Z^3_*}{\rm A}_{k,s}(t)\fr{f_l}{|l|}g_{k-l} h_k\les\nn& \Big(\|{\rm B}_{k,0}(t)f_k\|_{L^2_k}\|{\rm A}_{k,s}(t)g_k{\bf1}_{k\ne0}\|_{L^2_k}\\
&\quad+\|{\rm B}_{k,s}(t)f_k\|_{L^2_k}\|{\rm A}_{k,0}(t)g_k\|_{L^2_k}\Big)\|h_k\|_{L^2_k}.
\end{align}
\end{lem}

\begin{proof}
    \begin{align*}
&\sum_{k\in\Z^3}\sum_{l\in\Z^3_*}{\rm A}_{k,s}(t)\fr{f_l}{|l|}g_{k-l} h_k\\
=& \sum_{k\in\Z^3}\sum_{\substack{l\in\Z^3_*,|l|< |k-l|}}{\rm A}_{k,s}(t)\fr{f_l}{|l|}g_{k-l} h_k+\sum_{k\in\Z^3}\sum_{\substack{l\in\Z^3_*,|l|\ge |k-l|}}{\rm A}_{k,s}(t)\fr{f_l}{|l|}g_{k-l} h_k=I_{\rm LH}+I_{\rm HL}.
\end{align*}
By Young's inequity for convolution,
\begin{align*}
    I_{\rm LH}\les& \sum_{k\in\Z^3}\sum_{\substack{l\in\Z^3_*,|l|< |k-l|}}\fr{|f_l|}{|l|}{\rm A}_{k-l,s}(t)|g_{k-l|} |h_k|\\
    \les&\sum_{k\in\Z^3_*}\fr{|f_k|}{|k|}\|{\rm A}_{k,s}(t)g_k{\bf1}_{k\ne0}\|_{L^2_k}\|h_k\|_{L^2_k}\\
    \les& \|{\rm B}_{k,0}(t)f_k\|_{L^2_k}\|{\rm A}_{k,s}(t)g_k{\bf1}_{k\ne0}\|_{L^2_k}\|h_k\|_{L^2_k},
\end{align*}
where we have used
\begin{align*}
    \Big(\sum_{k\in\Z^3_*}\fr{|f_k|}{|k|}\Big)^2\le\sum_{k\in\Z^3_*}\fr{1}{|k|^3(\ln(e+|k|))^2}\sum_{k\in\Z^3_*}(\ln(e+|k|))^2|k||f_k|^2\les \|{\rm B}_{k,0}(t)f_k\|_{L^2_k}^2.
\end{align*}
Similarly, we have
\begin{align*}
    I_{\rm HL}
    \les&\sum_{k\in\Z^3}\sum_{\substack{l\in\Z^3_*,|l|\ge |k-l|}}{\rm B}_{l,s}(t)|f_l|\fr{|g_{k-l}|}{\la k-l\ra^{\fr32}} |h_k|\\
\les&\|{\rm B}_{k,s}(t)f_k\|_{L^2_k}\|{\rm A}_{k,0}(t)g_k\|_{L^2_k}\|h_k\|_{L^2_k}.
\end{align*}
Thus, \eqref{log-produ} holds.
\end{proof}

\begin{lem}\label{lem-LK}
Let
    \begin{align*}
    J_k=\bigg\|{\rm B}_{k,s}(t)\int_0^t\int_{\R}\bar{\hat{K}}_k(\xi_1)S_k(t-\tau;\xi_*(t,k))\xi_*(t-\tau,k)\hat{h}_{k}(\tau,\xi_*(t-\tau,k))d\xi_1d\tau\bigg\|_{L^2_t}.
    \end{align*}
    Then it holds that
    \begin{align}\label{bd:L-J}
\|J_k\|_{L^2_k}\les \nu^{-\fr12}\|\la v\ra^2{\bf A}_{s}{h}_{\ne}\|_{L^2 _tL^2_{x,v}}.
    \end{align}
\end{lem}
\begin{proof}
Recalling that $\xi_*(t-\tau,k)=\fr{k}{|k|}\xi_1+k(t-\tau)$, then we infer from \eqref{semigroup-est} and \eqref{bar-eta-t-xi} that there exists a positive constant $\dl$ independent of $\nu$, such that
\begin{align*}
    S_k(t-\tau;\xi_*(t,k))|\xi_*(t-\tau,k)|\le& e^{-\dl \nu\big(|\xi_1|^2(t-\tau)+|k|^2(t-\tau)^3\big)}\bigg|\fr{k}{|k|}\xi_1+k(t-\tau)\bigg|\\
    \les&e^{-\dl\nu|k|^2(t-\tau)^3}|\xi_1|+e^{-\dl\nu|k|^2(t-\tau)^3}|k|(t-\tau)\\
    =:&\tl{S}_{k,\dl;1)}(t-\tau)|\xi_1|+\tl{S}_{k,\dl;2)}(t-\tau).
\end{align*}
Noting that $\la C_s\nu^{\fr13} |k|^{\fr23}t\ra^{\fr{\ell}{2}}\les \la C_s\nu^{\fr13} |k|^{\fr23}\tau\ra^{\fr{\ell}{2}}\la C_s\nu^{\fr13} |k|^{\fr23}(t-\tau)\ra^{\fr{\ell}{2}}$, and $\la C_s\nu^{\fr13} |k|^{\fr23}(t-\tau)\ra^{\fr{\ell}{2}}$ can be absorbed by $\tl{S}_{k,\dl;i)}(t-\tau)$, $i=1,2$, we have
\begin{align}\label{B-A}
    \nn&{\rm B}_{k,s}(t)S_k(t-\tau;\xi_*(t,k))|\xi_*(t-\tau,k)|\\
    \les&|k|^{\fr12}\sum_{i=1,2}\left(\tl{S}_{k,\fr{\dl}{2};1)}(t-\tau)|\xi_1|{\rm A}_{k,s}(\tau)
    +\tl{S}_{k,\fr{\dl}{2};2)}(t-\tau){\rm A}_{k,s}(\tau)\right).
\end{align}
Accordingly, we can split $J_k$ into two parts, and write: $J_k\les J_{k,1}+J_{k,2}$. Before proceeding any further,  similar to \eqref{bd-semi1} and \eqref{bd:Sk1}, it is easy to verify that
\begin{align}\label{bd-sem}
    \big\|\tl{S}_{k,\fr{\dl}{2};1)}\big\|_{L^1_t}\les \nu^{-\fr13}|k|^{-\fr23},\quad{\rm and}\quad \big\|\tl{S}_{k,\fr{\dl}{2};2)}\big\|_{L^2_t}\les \nu^{-\fr12}.
\end{align}

To estimate $J_{k,1}$, noting  that by the Sobolev trace theorem, we have
\begin{align*}
    \int_{\R}|\hat{h}_{k}(\tau,\xi_*(t-\tau,k))|^2d\xi_1=\int_{\R}|\hat{h}_{k}(\tau,\fr{k}{|k|}\xi_1+k(t-\tau))|^2d\xi_1\les \|\la v\ra^2h_{k}(\tau)\|_{L^2_v}^2.
\end{align*}
Then using Cauchy-Schwarz inequality w.r.t. $\xi_1$ variable, in view of \eqref{B-A}, one deduces that
\begin{align*}
    J_{k,1}\le&|k|^{\fr12}\big\||\xi_1|\hat{K}_k\big\|_{L^2_{\xi_1}}\bigg\|\int_0^t\tl{S}_{k,\fr{\dl}{2};1)}(t-\tau){\rm A}_{k,s}(\tau)\|\la v\ra^2{h}_{k}(\tau)\|_{L^2_{v}}d\tau\bigg\|_{L^2_t}\\
    \les&|k|^{\fr12}\|K'_k\|_{L^2_z}\big\|\tl{S}_{k,\fr{\dl}{2};1)}\big\|_{L^1_t}\big\|\la v\ra^2{\rm A}_{k,s}(\tau){h}_{k}(\tau)\big\|_{L^2_tL^2_{v}}.
\end{align*}
Combining this with \eqref{est:Wk} and \eqref{bd-sem}, we arrive at
\begin{align}\label{bd:J1}
\|J_{k,1}\|_{L^2_k}\les&\nu^{-\fr13}\|\la v\ra^2{\bf A}_{s}{h}_{\ne}\|_{L^2 _tL^2_{x,v}}.
\end{align}
For $J_{k,2}$, similar to \eqref{trace}, using Sobolev trace theorem again yields
\begin{align*}
    \int_{\R}|\hat{h}_{k}(\tau,\xi_*(t-\tau,k))|^2dt\les\fr{1}{|k|}\|\la v\ra^2 h_{k}(\tau)\|_{L^2_v}^2.
\end{align*}
Then  in view of \eqref{bd-sem}, similar to \eqref{es:rho-L}, we have 
\begin{align*}
&\int_0^{T^*}\Big[\int_0^t\tl{S}_{k,\fr{\dl}{2};2)}(t-\tau){\rm A}_{k,s}(\tau)|\hat{h}_{k}(\tau,\xi_*(t-\tau,k))|d\tau\Big]^2dt\\
%\le&\int_0^{T^*}\int_0^t\tl{S}^2_{k,\fr{\dl}{2};2)}(t-\tau)d\tau \int_0^t|{\rm A}_{k,s}(\tau)\hat{h}_{k}(\tau,\xi_*(t-\tau,k))|^2d\tau dt\\
\les&\nu^{-1}\int_0^{T^*}\int_\tau^{T^*}|{\rm A}_{k,s}(\tau)\hat{h}_{k}(\tau,\xi_*(t-\tau,k))|^2 dtd\tau\\
\les&\fr{\nu^{-1}}{|k|}\int_0^{T^*}\big\|\la v\ra^2 {\rm A}_{k,s}(\tau)h_{k}(\tau)\big\|_{L^2_v}^2 d\tau.
\end{align*}
Now by using Minkowski's inequality, we get
\begin{align*}
    J_{k,2}\le& |k|^{\fr12} \int_{\R}|\hat{K}_k(\xi_1)|\Big\|\int_0^t\tl{S}_{k,\fr{\dl}{2};2)}(t-\tau)|{\rm A}_{k,s}(\tau)\hat{h}_{k}(\tau,\xi_*(t-\tau,k))|d\tau\Big\|_{L^2_t}d\xi_1\\
    \les&\nu^{-\fr12}\|\hat{K}_k\|_{L^1_{\xi_1}}\big\|\la v\ra^2 {\rm A}_{k,s}(\tau)h_{k}(\tau)\big\|_{L^2_tL^2_v}.
\end{align*}
Then thanks to \eqref{est:Wk}, we arrive at
\begin{align}\label{bd:J2}
    \|J_{k,2}\|_{L^2_k}\les \nu^{-\fr12}\|\la v\ra^2{\bf A}_{s}{h}_{\ne}\|_{L^2 _tL^2_{x,v}}.
\end{align}
It follows from \eqref{bd:J1} and \eqref{bd:J2} that \eqref{bd:L-J} holds.
\end{proof}

\begin{coro}\label{coro-NLK}
Let
    \begin{align*}
    \tl{J}_k=\bigg\|{\rm B}_{k,\ell}(t)\int_0^t\int_{\R}\bar{\hat{K}}_k(\xi_1)S_k(t-\tau;\xi_*(t,k))\xi_*(t-\tau,k)\sum_{l\in\Z^n_*}\fr{f_l(\tau)}{|l|}\hat{g}_{k-l}(\tau,\xi_*(t-\tau,k))d\xi_1d\tau\bigg\|_{L^2_t}.
    \end{align*}
    Then it holds that
    \begin{align}\label{bd:J}
\|\tl{J}_k\|_{L^2_k}\les \nu^{-\fr12}\Big(\|{\bf B}_{s}f\|_{L^2_{t,x}}\|\la v\ra^2{\bf A}_{0}{g}\|_{L^\infty _tL^2_{x,v}}+\|{\bf B}_{0}f\|_{L^2_{t,x}}\|\la v\ra^2{\bf A}_{\ell}{g}_{\ne}\|_{L^\infty _tL^2_{x,v}}\Big).
    \end{align}
\end{coro}
\begin{proof}
From Lemma \ref{lem-LK}, we have
\begin{align*}
    \|\tl{J}_k\|_{L^2_k}\les& \nu^{-\fr12}\bigg\|\Big\|\la v\ra^2{\rm A}_{k,s}(\tau)\sum_{l\in\Z^n_*}\fr{f_l(\tau)}{|l|}g_{k-l}(\tau)\Big\|_{L^2 _tL^2_{v}}\bigg\|_{L^2_k}\\
    =&\nu^{-\fr12}\Big\|{\rm A}_{k,s}(\tau)\sum_{l\in\Z^n_*}\fr{|f_l(\tau)|}{|l|}\|\la v\ra^2g_{k-l}(\tau)\|_{L^2_{v}}\Big\|_{L^2 _tL^2_k}.
\end{align*}
Then \eqref{bd:J} follows immediately.
\end{proof}

\section{Estimate of $f^{\rm e}-\mu_2$}
The aim of this section is to bound $f^{\rm e}-\mu_2$ in appropriate spacetime norms. More precisely, the following lemma holds.
\begin{lem}\label{lem-fe}
    Let $f^{\rm e}$ be the solution to the equation \eqref{eq:fe}, and $m\in[0,\infty)$. Then for all $T\le\fr12 \nu^{-1}$, there exists a positive constant $C$, depending on the function $\sig=\sig(v)$ and $m$, but independent of $\nu$, such that
    \begin{align}\label{C.1}
        \|\la v\ra^m\pr_v(f^{\rm e}-\mu_2)\|_{L^2_t L^2_v}\le C(\nu T^{\fr32}\gamma^{-\fr32}+\nu^{\fr12}T\gamma^{\fr12}).
    \end{align}
\end{lem}
\begin{proof}
We first solve $f^{\rm e}$   explicitly. To this end, taking Fourier transform of \eqref{eq:fe} in $v$, we obtain
\begin{align*}
    \partial_t\widehat{f^{\rm e}}(t,\xi)+\nu \xi\partial_\xi\widehat{f^{\rm e}}(t,\xi)+\nu\xi^2\widehat{f^{\rm e}}(t,\xi)=0,\quad \widehat{f^{\rm e}}(0,\xi)=\fr{M\gamma^2\hat{\sig}(\gamma\xi)}{1+M\gamma^2m_\sig}.
\end{align*}
By the method of characteristics, we have
\begin{align*}
    \widehat{f^{\rm e}}(t,\xi)=\exp\left(-\fr12(1-e^{-2\nu t})\xi^2\right)\fr{M\gamma^2\hat{\sig}(e^{-\nu t}\gamma\xi)}{1+M\gamma^2m_\sig}.
\end{align*}
%%%%%%%%%%%%%%%%

%%%%%%%%%%%%%%%
Recalling the definition of $\sig_1(\cdot)$ in \eqref{def-sig1}, we can write
\begin{align*}
&\pr_\xi^m\big[\xi\big(\widehat{f^{\rm e}}(t,\xi)-\widehat{f^{\rm e}}(0,\xi)\big)\big]\\
=&\fr{M\gamma e^{\nu t}}{1+M\gamma^2m_\sig}\left[\exp\left(-\fr12(1-e^{-2\nu t})\xi^2\right)-1\right](e^{-\nu t}\gamma)^{m}\sig_1^{(m)}(e^{-\nu t}\gamma\xi)\\
&+\fr{M\gamma e^{\nu t}}{1+M\gamma^2m_\sig}\sum_{\ell=1}^{m} \binom{m}{\ell} \pr_{\xi}^\ell\left[\exp\left(-\fr12(1-e^{-2\nu t})\xi^2\right)\right](e^{-\nu t}\gamma)^{m-\ell}\sig_1^{(m-\ell)}(e^{-\nu t}\gamma\xi)\\
&+\left[\fr{M\gamma e^{\nu t}}{1+M\gamma^2m_\sig} (e^{-\nu t}\gamma)^m\sig_1^{(m)}(e^{-\nu t}\gamma\xi)-\fr{M\gamma}{1+M\gamma^2m_\sig}\gamma^m\sig_1^{(m)}(\gamma\xi)\right]\\
=&J_1(t,\xi)+J_2(t,\xi)+J_3(t,\xi).
\end{align*}
By virtue of the elementary inequality $1-e^{-x}\le x$ for all $x\ge0$, one deduces that
\begin{align}\label{C.2}
    \|J_1(t)\|_{L^2_\xi}\les \nu te^{\nu t}\gamma (e^{-\nu t}\gamma)^m\big\|\xi^2\sig_1^{(m)}(e^{-\nu t}\gamma \xi) \big\|_{L^2_\xi}=\nu te^{\fr72\nu t}\gamma^{-\fr32} (e^{-\nu t}\gamma)^m\big\|\eta^2\sig_1^{(m)}(\eta) \big\|_{L^2_\eta}.
\end{align}
Next, note that
\begin{align*}
    &\partial_{\xi}^\ell \left[ \exp\left(-\frac{1}{2}(1 - e^{-2\nu t}) \xi^2\right) \right]\\
    =& (-1)^\ell (1 - e^{-2\nu t})^{\fr{\ell}{2}} \, \mathrm{He}_\ell\left( \sqrt{1 - e^{-2\nu t}} \, \xi \right) \exp\left(-\frac{1}{2}(1 - e^{-2\nu t}) \xi^2\right),
\end{align*}
where
\begin{align*}
    \mathrm{He}_\ell(x) = \ell! \sum_{k=0}^{\lfloor \ell/2 \rfloor} \frac{(-1)^k}{k! \, (\ell-2k)! \, 2^k} \, x^{\ell-2k}.
\end{align*}
Then we have
    \begin{align*}
    \left|\partial_{\xi}^\ell \left[ \exp\left(-\frac{1}{2}(1 - e^{-2\nu t}) \xi^2\right) \right]\right|\les (\nu t)^{\fr{\ell}{2}},
\end{align*}
and hence
\begin{align}\label{C.3}
    \|J_2(t)\|_{L^2_{\xi}}\les e^{\fr32\nu t}\gamma^{\fr12}\sum_{\ell=1}^{m}(\nu t)^{\fr{\ell}{2}}(e^{-\nu t}\gamma)^{m-\ell}\big\|\sig_1^{(m-\ell)}\|_{L^2_\eta}.
\end{align}
To bound $J_3(t,\xi)$, we split it into three parts:
\begin{align*}
 J_3(t,\xi)
    =&\fr{M\gamma (e^{\nu t}-1)}{1+M\gamma^2m_\sig} (e^{-\nu t}\gamma)^m\sig_1^{(m)}(e^{-\nu t}\gamma\xi)+\fr{M\gamma }{1+M\gamma^2m_\sig} \left[(e^{-\nu t}\gamma)^m-\gamma^m\right]\sig_1^{(m)}(e^{-\nu t}\gamma\xi)\\
    &+\fr{M\gamma }{1+M\gamma^2m_\sig} \gamma^m\left[\sig_1^{(m)}(e^{-\nu t}\gamma\xi)-\sig_1^{(m)}(\gamma\xi)\right]=J_{3,1}(t,\xi)+J_{3,2}(t,\xi)+J_{3,3}(t,\xi).
\end{align*}
It is easy to verify that
\begin{align}\label{C.4}
    \|J_{3,1}(t)\|_{L^2_\xi}\les \nu te^{\fr32\nu t}(e^{-\nu t}\gamma)^m\gamma^{\fr12}\big\|\sig_1^{(m)}\big\|_{L^2_\eta}, 
\end{align}
and
\begin{align}\label{C.5}
    \|J_{3,2}(t)\|_{L^2_\xi}\les m\nu te^{\fr12\nu t}\gamma^{m+\fr12}\big\|\sig_1^{(m)}\big\|_{L^2_\eta}.
\end{align}
To bound $J_{3,3}(t,\xi)$, by the mean value theorem, we have
\[
\sig_1^{(m)}(e^{-\nu t}\gamma\xi)-\sig_1^{(m)}(\gamma\xi)=-\nu\int_0^te^{-\nu s}\gamma\xi\sig_1^{(m+1)}(e^{-\nu s}\gamma\xi)ds.
\]
Then using Minkowski's inequality, we are led to
\begin{align}\label{C.6}
\|J_{3,3}(t)\|_{L^2_\xi}\les\nn& \gamma^{m+1}\nu \int_0^t\|e^{-\nu s}\gamma\xi\sig_1^{(m+1)}(e^{-\nu s}\gamma\xi)\|_{L^2_\xi}ds\\
\les\nn&\gamma^{m+\fr12}\nu\int_0^t e^{\fr12\nu s}ds\|\eta\sig_1^{(m+1)}(\eta)\|_{L^2_\eta}=2\gamma^{m+\fr12}(e^{\fr12\nu t}-1)\|\eta\sig_1^{(m+1)}(\eta)\|_{L^2_\eta}\\
\les&\nu t\gamma^{m+\fr12}e^{\fr12\nu t}\|\eta\sig_1^{(m+1)}(\eta)\|_{L^2_\eta}.
\end{align}
Collecting the estimates in \eqref{C.2}--\eqref{C.6}, using the facts $0\le t\le T\le\fr12 \nu^{-1}$, $m\ge0$ and $0<\gamma\le1$, we conclude that \eqref{C.1} holds.
\end{proof}

\bigbreak
\noindent{\bf Acknowledgments}
RZ is partially supported by the NSF of China under  Grants 12571235 and 12222105.

\noindent{\bf Conflict of interest:}  We confirm that we do not have any conflict of interest. 

\noindent{\bf Data availability:} The manuscript has no associated data. 

\noindent{\bf Author Contributions Statement:} W.Z. and  R.Z. have the same contribution to the manuscript.

\end{appendix}

\bibliographystyle{plain} 
\bibliography{references.bib}

\end{document}